\documentclass{amsart}

\usepackage[utf8]{inputenc}
\usepackage[mathscr]{eucal}

\usepackage{amssymb}
\usepackage{amsmath}
\usepackage{amsthm}
\usepackage{enumitem}

\usepackage{pgf,tikz}
\usetikzlibrary{arrows,matrix}

\numberwithin{equation}{section}
\theoremstyle{plain}
\newtheorem{thm}{Theorem}[section]
\newtheorem{lemma}[thm]{Lemma}
\newtheorem{prop}[thm]{Proposition}
\newtheorem{cor}[thm]{Corollary}
\newtheorem{claim}[thm]{Claim}

\theoremstyle{definition}
\newtheorem{defn}[thm]{Definition}

\newtheorem{question}[thm]{Question}
\newtheorem{remark}[thm]{Remark}
\newtheorem{example}[thm]{Example}

\usepackage{xspace} 

\newcommand{\R}{\mathbf{R}}
\newcommand{\Z}{\mathbf{Z}}

\newcommand{\ehr}{\mathrm{Ehr}}

\DeclareMathOperator{\conv}{Conv}
\DeclareMathOperator{\sign}{sign}

\definecolor{light-gray}{gray}{0.8}
\definecolor{v}{rgb}{0.28,0,0.72}
\definecolor{e}{rgb}{0,1,0.2}
\definecolor{r}{rgb}{1,0,0}

\begin{document}

\title[Jaeger dissections for directed graphs]{Root polytopes and Jaeger-type dissections for directed graphs}

\author{Tam\'as K\'alm\'an}
\address{Department of Mathematics\\
Tokyo Institute of Technology\\
H-214, 2-12-1 Ookayama, Meguro-ku, Tokyo 152-8551, Japan}
\email{kalman@math.titech.ac.jp}

\author{Lilla T\'othm\'er\'esz}
\address{MTA-ELTE Egerv\'ary Research Group, P\'azm\'any P\'eter s\'et\'any 1/C, Budapest, Hungary}
\email{tmlilla@caesar.elte.hu}


\date{}

\begin{abstract}
We associate root polytopes 
to directed graphs and study them by using ribbon structures. Most attention is paid to what we call the semi-balanced case, i.e., when 
each cycle has the same number of edges pointing in the two directions. Given a ribbon structure, we identify a natural class of spanning trees and show that, in the semi-balanced case, they induce a shellable dissection of the root polytope into maximal simplices. This allows for a computation of the $h^*$-vector of the polytope and for showing some properties of this new graph invariant, such as a product formula and that 
in the planar 
case, 
the $h^*$-vector is equivalent to the greedoid polynomial of the dual graph.
We obtain a general recursion relation as well.
We also work out the case of layer-complete directed graphs, where our method recovers a previously known triangulation.
Indeed our dissection is often but not always a triangulation; we address this with a series of examples.

\end{abstract}

\maketitle

\section{Introduction}
\label{sec:intro}


This paper grew out of efforts to establish `signed versions' of our recent results \cite{hyperBernardi} regarding the interplay between root polytopes and ribbon structures of bipartite graphs (indeed, certain arguments extend almost word-by-word), but it ended up achieving quite a bit more.
We found significant simplifications and new connections, rounding eventually into a theory of a wide class of signed bipartite graphs and their `interior'  polynomial invariants.

It will be more convenient to work with directed graphs instead of signed ones --- in bipartite cases, the two ideas are equivalent.
The root polytope $\mathcal Q_G$ is naturally associated to any directed graph $G$, see Definition \ref{def:root_polytope}. If $G$ is connected then $\dim\mathcal Q_G=|V(G)|-1$, except for so-called \emph{semi-balanced graphs}, when we have $\dim\mathcal Q_G=|V(G)|-2$ (compare Figures \ref{fig:K4} and \ref{fig:K34}). 
We give several equivalent descriptions of this condition, one of which is that an integer-valued potential can be assigned to the vertices in such a way that for all edges, the potential is one higher at the head than at the tail. Such directed graphs and their root polytopes will be our central objects.

Important motivation is provided by the fact, due to Higashitani, Jochemko, and Micha{\l}ek \cite{arithm_symedgepoly}, that all facets of the so called symmetric edge polytope of any graph $H$ take the form $\mathcal Q_G$, where $G$ is a spanning subgraph of $H$ equipped with a semi-balanced orientation. We exploit this connection in the sequel \cite{symmetric}.

All semi-balanced graphs are bipartite and any bipartite graph has many semi-balanced orientations (cf.\ Example \ref{ex:K34}), for example the one in which all edges point from one color class to the other (i.e., when one color class is on potential $0$ and the other on potential $1$). The latter was the focus of our previous work \cite{hyperBernardi} and here we will refer to it as the \emph{standard orientation}.

Our method to study $\mathcal Q_G$  
is to endow $G$ with an arbitrary \emph{ribbon structure}, much like Bernardi did in his work on the Tutte polynomial \cite{Bernardi_first,Bernardi_Tutte}.
The key idea then is that if we also fix a base node and an incident base edge, then we can single out those spanning trees of $G$ that are such that, when we `walk around' them in a natural sense, then all non-edges are first encountered at their tail. We call these \emph{Jaeger trees} because the condition is adopted from a paper of F. Jaeger \cite{Jaeger} in knot theory.

Root polytopes of spanning trees of $G$ are unimodular simplices within $\mathcal Q_G$, which are also maximal if $G$ is semi-balanced. One of our main results claims for all semi-balanced graphs $G$ that the collection of simplices, obtained from Jaeger trees, is a \emph{dissection} of $\mathcal Q_G$.
(That is, the simplices are interior-disjoint and their union is $\mathcal Q_G$.)

It follows that the number of Jaeger trees of a semi-balanced graph is the normalized volume of the graph's root polytope. In particular, the number of Jaeger trees of a semi-balanced graph is independent of the base node and base edge; in fact, it is even independent of the ribbon structure.

We establish the dissection in a new way which is rather easier than our previous proof in the case of the standard orientation  \cite[Section 4]{hyperBernardi}. It involves a natural order on the set of all spanning trees of a directed ribbon graph.
Furthermore, we prove that the restriction of our tree-order to the set of Jaeger trees induces a \emph{shelling} of the dissection (see Section \ref{sec:shelling} for the definition), in such a way that the `$h$-vector statistic' of the shelling can be expressed in elementary graph-theoretical terms. 
This in turn allows for a computation of the $h^*$-vector of the root polytope.

An important takeaway here is that a shellable dissection into unimodular simplices is just as useful for Ehrhart theory as a triangulation; that is, it is not always necessary to assume that the intersection of any two simplices is a common face.

In the case of the standard orientation, one is able to express the $h^*$-vector as a generating function of hypertree activities \cite{KP_Ehrhart}, which is also known as the interior polynomial. We do not quite take this step in this paper, due to the lack of an appropriate notion of hypertree in the general (semi-balanced) case. 
Put another way, the best we can currently do is to \emph{define} the interior polynomial of a semi-balanced graph as the $h^*$-vector of its root polytope. We indicate that this is a meaningful graph invariant by establishing a product formula for 
graphs fused at a vertex or along one edge, as well as a recursion relation, cf.\ Theorem \ref{thm:rekurzio}.

We 
also find that in the case of a plane semi-balanced graph, the $h^*$-vector is equivalent to the \emph{greedoid polynomial} of the branching greedoid associated to the dual graph. 
Here the duals in question are exactly the plane Eulerian digraphs\footnote{In light of this connection, A. Frank suggested that semi-balanced graphs could also be called \emph{co-Eulerian}.}.

We work out in detail another class of examples, one that 
includes all semi-balanced complete bipartite graphs. We call these \emph{layer-complete} (see Figure \ref{fig:complete_bipartite_embedding} for an example) and equip them with a natural ribbon structure, for which we describe all Jaeger trees in a simple manner. This involves a non-crossing condition that first appeared in a triangulation of a product of two simplices \cite{GGP}. 
For all layer-complete directed graphs, these trees do induce triangulations.
In 
the complete bipartite case,
the same triangulation was derived by Higashitani, Jochemko, and Micha\l{}ek \cite{arithm_symedgepoly}, who relied on Gr\"obner basis techniques. 

As a matter of fact, in the aforementioned planar case too, our dissection turns out to be a triangulation (for any base node and base edge). 
This leads us to examine the landscape of whether Jaeger trees induce a triangulation in various cases. We find bipartite graphs 
for which each choice of ribbon structure, base node, and base edge results in a triangulation, as well as graphs for which none of them does. 
Despite these phenomena, it is easy to see that whether our method yields a triangulation typically depends not only on the graph but also on its ribbon structure.
Furthermore, 
in general the answer depends on the base node and base edge as well.
In another direction, we conjecture that Jaeger trees induce a triangulation if and only if their complements constitute the bases of a greedoid.

The paper is organized as follows. In Section \ref{sec:basics} we 
give several descriptions of semi-balanced graphs. Section \ref{sec:dimension} covers some basics of vector representations of graphs, including dimension counts for root polytopes. We define the interior polynomial of a semi-balanced graph in Section \ref{sec:int_poly}, where we also sketch out the Ehrhart theoretical underpinnings of our approach to computing it. In Section \ref{sec:Jaeger_trees} we introduce Jaeger trees and establish our dissection; the conclusion is reached in a way that is significantly easier than in \cite{hyperBernardi}. In Section \ref{sec:shelling} we show that the dissection is shellable and describe the corresponding $h$-vector, i.e., the interior polynomial. Section \ref{sec:computing_the_Jaeger_tree} takes a look at some algorithmic aspects of the dissection and simplifies the original proof of \cite[Theorem 4.5]{hyperBernardi}. In Section \ref{sec:planar} we examine semi-balanced plane graphs and draw a connection to greedoid polynomials, and in Section \ref{sec:layer_complete} we discuss layer-complete directed graphs.
Section \ref{sec:diss_vs_triang} concentrates on the distinction between dissections and triangulations. Finally, in Section \ref{sec:formula}, we  treat disconnected cases and establish our product and recursion formulas for the interior polynomial.

Acknowledgements: This paper would not exist without the  
many illuminating conversations that TK had over the years with Alex Postnikov about, among other things, root polytopes associated to bipartite graphs and signed graphs. We would like to thank Seunghun Lee and Thomas Zaslavsky for suggesting the question that led to Theorem \ref{thm:exist_graph_with_no_nice_ribbon_str}. TK was supported by a Japan Society for the Promotion of Science (JSPS) Grant-in-Aid for Scientific Research C (no.\ 17K05244). LT was supported by the National Research, Development and Innovation Office of Hungary -- NKFIH, grants no.\ 132488 and 128673, by the János Bolyai Research Scholarship of the Hungarian Academy of Sciences, and by the ÚNKP-21-5 New National Excellence Program of the Ministry for Innovation and Technology, Hungary. 
LT was also partially supported by the Counting in Sparse Graphs Lendület Research Group of Rényi Institute. 

\section{Directed, signed, and bipartite graphs}
\label{sec:basics}

In this paper, we will consider graphs and directed graphs without loops and multiple edges. We will always suppose that our graphs are (weakly) connected. 

A \emph{directed graph} is one in which the two endpoints of each edge are designated as \emph{tail} and \emph{head}. We draw the edge as an arrow pointing from tail to head, and write it as
$\overrightarrow{xy}$, where $x$ is the tail and $y$ is the head.
We may write $xy$ if we do not want to specify the head and the tail, or use a single Greek character when we wish to suppress the endpoints altogether.

Most of what happens in this paper can also be told in the language of signed graphs.
A \emph{signed graph} is an undirected graph where 
each edge has either a positive or a negative sign associated to it. For an edge $\varepsilon$, we take $\sign(\varepsilon)=1$ if the sign is positive and $\sign(\varepsilon)=-1$ if the sign is negative.

We will mostly be interested in 
\emph{bipartite graphs}, that is graphs with
two vertex classes \emph{black} and \emph{white} so that each edge has one black and one white endpoint. We denote the set of black vertices by $U$ and the set of white ones by $W$. 

In the bipartite case, signed graphs and directed graphs can be identified in the following way: Given a signed graph, orient each positive edge from its black endpoint to its white endpoint, and orient each negative edge from its white endpoint to its black endpoint.
We will prefer the terminology of directed graphs as it leads to simpler notation.

In the predecessor \cite{hyperBernardi} of this paper, we worked with
undirected bipartite graphs. The setting of that paper corresponds to signed graphs with all positive signs, or, in the directed language, to bipartite graphs where each edge is oriented from $U$ to $W$. We will call this the \emph{standard orientation}.

Let $G$ be a  
graph on the vertex set $V$. For a vertex $v\in V$, we will write $\mathbf v$ for the corresponding standard basis vector of the vector space $\mathbf{R}^V$.
We define the root polytope of a directed graph in the following way, which is an obvious extension of Postnikov's definition \cite{alex} in the case of a bipartite graph with its standard orientation. 

\begin{defn}\label{def:root_polytope}
The \emph{root polytope} of a directed graph $G=(V,E)$ is the convex hull
\[\mathcal{Q}_G=\conv(\{\,\mathbf{h} - \mathbf{t}\mid \overrightarrow{th}\in E\,\}).\]
\end{defn}

There is an obvious rule $\mathcal Q_{-G}=-\mathcal Q_G$, where the \emph{reverse graph} $-G$ means $G$ with the orientation of every edge reversed.
As a shorthand, we also use the notation $\mathbf x_\varepsilon$ for $\mathbf{h} - \mathbf{t}$ where $\varepsilon=\overrightarrow{th}$. 
Note that the generators given above are all vertices (extremal points) of $\mathcal{Q}_G$. One way to see this is to consider the linear functional defined over $\mathbf{R}^V$ by $\mathbf{h}\mapsto1$, $\mathbf{t}\mapsto-1$, and all other vertices mapping to $0$. Let us also note that with respect to the natural basis of $\R^V$, the sum of the coordinates is $0$ for all vertices of $\mathcal Q_G$ (i.e., $\mathcal Q_G\subset\langle\mathbf1\rangle^\perp$, where $\mathbf1=\sum_{v\in V}\mathbf v$) and thus $\dim\mathcal Q_G\le|V|-1$.

\begin{example}
\label{ex:K4}
The complete graph $K_4$, up to isomorphism and reversal, has three inequivalent orientations. The corresponding root polytopes, shown in Figure \ref{fig:K4}, are all three-dimensional but mutually non-isometric, in fact not even combinatorially equivalent.
\end{example}

\begin{figure}
    \centering
    \begin{tikzpicture}[scale=.09]
		\tikzstyle{o}=[circle,draw,minimum width = 17pt,scale=0.8]
		\node [o] (1) at (5,38) {\small $a$};
		\node [o] (2) at (17,38) {\small $b$};
		\node [o] (3) at (17,50) {\small $c$};
		\node [o] (4) at (5,50) {\small $d$};
		\path [thick,->,>=stealth] (1) edge node {} (2);
		\path [thick,<-,>=stealth] (1) edge node {} (3);
		\path [thick,<-,>=stealth] (1) edge node {} (4);
		\path [thick,<-,>=stealth] (2) edge node {} (3);
		\path [thick,<-,>=stealth] (2) edge node {} (4);
		\path [thick,<-,>=stealth] (3) edge node {} (4);
		\draw [fill] (8,0) circle [radius=.5];
		\node [below] (5) at (8,0) {\tiny $\mathbf b-\mathbf c$};
		\draw [fill] (22,2) circle [radius=.5];
		\node [below right] (6) at (22,2) {\tiny $\mathbf b-\mathbf a$};
		\draw [fill] (0,11) circle [radius=.5];
		\node [left] (7) at (0,11) {\tiny $\mathbf a-\mathbf c$};
		\draw [fill] (16,9) circle [radius=.5];
		\node [left] (8) at (16,9) {\tiny $\mathbf b-\mathbf d$};
		\draw [fill] (8,20) circle [radius=.5];
		\node [above left] (10) at (8,20) {\tiny $\mathbf a-\mathbf d$};
		\draw [fill] (22,22) circle [radius=.5];
		\node [right] (11) at (22,22) {\tiny $\mathbf c-\mathbf d$};
		\draw [thick] (8,0) -- (22,2);
		\draw [thick] (8,20) -- (22,22);
		\draw [thick] (22,2) -- (22,22);
		\draw [thick] (16,9) -- (8,0);
		\draw [thick] (16,9) -- (22,2);
		\draw [thick] (16,9) -- (22,22);
		\draw [thick] (16,9) -- (8,20);
		\draw [thick] (0,11) -- (8,20);
		\draw [thick] (0,11) -- (8,0);
		\draw [thick] (0,11) -- (9.6,15.8);
		\draw [thick] (11.6,16.8) -- (22,22);
		\draw [thick] (0,11) -- (12.1,6.05);
		\draw [thick] (14.3,5.15) -- (22,2);
	\end{tikzpicture}
	\hspace{.5cm}
	\begin{tikzpicture}[scale=.09]
    	\tikzstyle{o}=[circle,draw,minimum width = 17pt,scale=0.8]
		\node [o] (1) at (5,38) {\small $a$};
		\node [o] (2) at (17,38) {\small $b$};
		\node [o] (3) at (17,50) {\small $c$};
		\node [o] (4) at (5,50) {\small $d$};
		\path [thick,->,>=stealth] (1) edge node {} (2);
		\path [thick,<-,>=stealth] (1) edge node {} (3);
		\path [thick,<-,>=stealth] (1) edge node {} (4);
		\path [thick,->,>=stealth] (2) edge node {} (3);
		\path [thick,<-,>=stealth] (2) edge node {} (4);
		\path [thick,<-,>=stealth] (3) edge node {} (4);
		\draw [fill] (22,2) circle [radius=.5];
		\node [below right] (6) at (22,2) {\tiny $\mathbf b-\mathbf a$};
		\draw [fill] (0,11) circle [radius=.5];
		\node [left] (7) at (0,11) {\tiny $\mathbf a-\mathbf c$};
		\draw [fill] (16,9) circle [radius=.5];
		\node [below] (8) at (16,9) {\tiny $\mathbf b-\mathbf d$};
		\draw [fill] (8,20) circle [radius=.5];
		\node [above left] (10) at (8,20) {\tiny $\mathbf a-\mathbf d$};
		\draw [fill] (22,22) circle [radius=.5];
		\node [right] (11) at (22,22) {\tiny $\mathbf c-\mathbf d$};
		\draw [fill] (20,26) circle [radius=.5];
		\node [above] (12) at (20,26) {\tiny $\mathbf c-\mathbf b$};
		\draw [thick] (20,26) -- (20.3,22.4);
		\draw [thick] (20.4,21.2) -- (20.5,20);
		\draw [thick] (20.7,17.6) -- (22,2);
		\draw [thick] (8,20) -- (22,22);
		\draw [thick] (22,2) -- (22,22);
		\draw [thick] (22,22) -- (20,26);
		\draw [thick] (16,9) -- (22,2);
		\draw [thick] (16,9) -- (22,22);
		\draw [thick] (16,9) -- (8,20);
		\draw [thick] (0,11) -- (8,20);
		\draw [thick] (0,11) -- (8.8,17.6);
		\draw [thick] (10,18.5) -- (12,20);
		\draw [thick] (14,21.5) -- (20,26);
		\draw [thick] (20,26)  -- (8,20);
		\draw [thick] (0,11) -- (22,2);
		\draw [thick] (0,11) -- (16,9);
	\end{tikzpicture}
	\hspace{.5cm}
	\begin{tikzpicture}[scale=.09]
		\tikzstyle{o}=[circle,draw,minimum width = 17pt,scale=0.8]
		\node [o] (1) at (5,38) {\small $a$};
		\node [o] (2) at (17,38) {\small $b$};
		\node [o] (3) at (17,50) {\small $c$};
		\node [o] (4) at (5,50) {\small $d$};
		\path [thick,->,>=stealth] (1) edge node {} (2);
		\path [thick,<-,>=stealth] (1) edge node {} (3);
		\path [thick,<-,>=stealth] (1) edge node {} (4);
		\path [thick,<-,>=stealth] (2) edge node {} (3);
		\path [thick,->,>=stealth] (2) edge node {} (4);
		\path [thick,<-,>=stealth] (3) edge node {} (4);
		\draw [fill] (8,0) circle [radius=.5];
		\node [below] (5) at (8,0) {\tiny $\mathbf b-\mathbf c$};
		\draw [fill] (22,2) circle [radius=.5];
		\node [below right] (6) at (22,2) {\tiny $\mathbf b-\mathbf a$};
		\draw [fill] (0,11) circle [radius=.5];
		\node [left] (7) at (0,11) {\tiny $\mathbf a-\mathbf c$};
		\draw [fill] (12,17) circle [radius=.5];
		\node [right] (9) at (12,17) {\tiny $\mathbf d-\mathbf b$};
		\draw [fill] (8,20) circle [radius=.5];
		\node [above left] (10) at (8,20) {\tiny $\mathbf a-\mathbf d$};
		\draw [fill] (22,22) circle [radius=.5];
		\node [right] (11) at (22,22) {\tiny $\mathbf c-\mathbf d$};
		\draw [thick] (8,0) -- (22,2);
		\draw [thick] (8,20) -- (22,22);
		\draw [thick] (22,2) -- (22,22);
		\draw [thick] (8,0) -- (8,20);
		\draw [thick] (12,17) -- (8,0);
		\draw [thick] (12,17) -- (22,2);
		\draw [thick] (12,17) -- (22,22);
		\draw [thick] (12,17) -- (8,20);
		\draw [thick] (0,11) -- (8,20);
		\draw [thick] (0,11) -- (8,0);
		\draw [thick] (0,11) -- (7,14.5);
		\draw[thick] (9,15.5) -- (12,17);
	\end{tikzpicture}
    \caption{Orientations of $K_4$ with their root polytopes.}
    \label{fig:K4}
\end{figure}
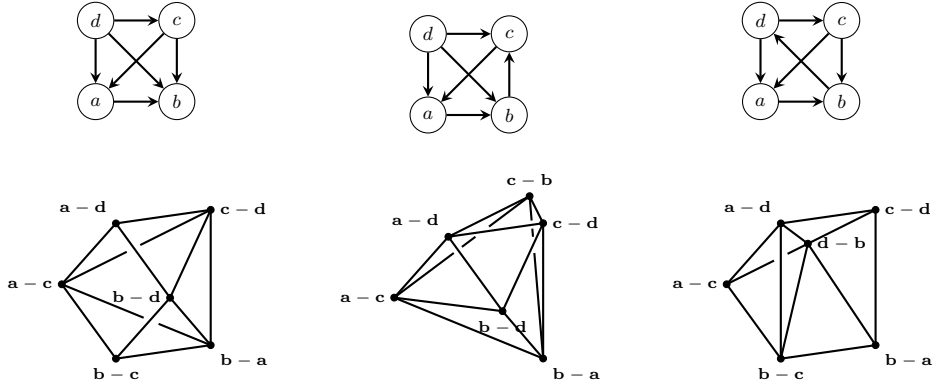

\begin{remark} \label{rem:signed_root_poly_def}
	For a signed graph $G$ one could define, taking a clue  from Ohsugi and Hibi's notion of \emph{edge polytope} \cite{hibi}, the root polytope of $G$ as $\conv(\{\sign(uw)\cdot(\mathbf{u} + \mathbf{w})\mid uw\in E(G)\})$. For signed bipartite graphs $G$ with $V(G)=U\sqcup W$, this is 
	isometric to the root polytope of Definition \ref{def:root_polytope} by sending $\mathbf u$ to $-\mathbf u$ for each $u\in U$. 
    On the other hand, the signed non-bipartite case and the directed non-bipartite case are unrelated.
    In this paper, we will focus on bipartite graphs. 
    Note that Zaslavsky, in his classic paper 
    \cite{Zaslavsky}, also introduces a vector representation for signed graphs, that is, he assigns a vector to each signed edge. This vector representation is different from the way we assign points of $\R^V$ to signed edges, and the properties of the resulting structures seem very different.
\end{remark}

An important special case of Definition \ref{def:root_polytope} is a bipartite graph with its
standard orientation, 
where we obtain 
the root polytope of the underlying undirected (bipartite) graph as defined in \cite{alex}. This case was examined in our previous paper \cite{hyperBernardi} (note that in \cite{hyperBernardi} we associated the point $\mathbf u + \mathbf v$ to the edge $uv$ instead of $\mathbf u - \mathbf v$, but the resulting polytopes are isometric because of bipartiteness).
In this case, as already noted in \cite{alex}, connectedness of the graph implies that the root polytope has dimension $|V|-2$.

We will see in the next section that for general directed graphs, even for bipartite ones, the root polytope typically has
dimension $|V|-1$, cf.\ Example \ref{ex:K4}. We will however be mostly interested in cases when the dimension remains $|V|-2$. 
Let us introduce what will turn out to be a necessary and sufficient condition for this.

\begin{defn}
In a directed graph, we say that a cycle is \emph{semi-balanced} if it has the same number of edges going in the two directions around the cycle. We call a cycle \emph{un-semibalanced} if it is not semi-balanced. 
We call a directed graph \emph{semi-balanced} if all of its cycles are semi-balanced. 
\end{defn}

Clearly, a semi-balanced graph is acyclic and bipartite.
We use the name semi-balanced because we wish to emphasize the connection of this notion to semi-activities, see also Theorem \ref{thm:shelling}. 

Let us discuss two more equivalent descriptions of semi-balancedness for a directed graph. Recall that a \emph{cut} in a graph is a set of edges obtained by partitioning the vertex set into two classes and taking all edges between the classes. A cut in a directed graph is called \emph{directed} if it arises from a partition so that each edge of the cut has its head in the same class.

\begin{defn}
	For a cycle $C$ of a directed graph $G$, we say that two edges 
	of $C$ are \emph{opposite} if their orientations in $G$ are opposite with respect to $C$.
\end{defn}

\begin{thm}\label{thm:char_semibalanced}
	The following statements are equivalent for a bipartite directed graph $G=(U\sqcup W,E)$:
	\begin{enumerate}
	\item\label{egy} $G$ is semi-balanced. 
	\item\label{kettoesfel} $G$ can be obtained from the standard orientation 
	by successively reversing directed cuts.
	\item\label{ketto} $G$ can be obtained from the standard orientation by reversing (at once) some disjoint directed cuts.
	\item\label{harom} There is a function $l\colon U\sqcup W \to \Z$ such that we have $l(h)-l(t)=1$ for each edge $\overrightarrow{th}$ of $G$.
	\end{enumerate}
\end{thm}

When \eqref{harom} holds, we say that $l$ and 
$G$ are \emph{compatible}.
We will refer to the function $l$ of \eqref{harom} as a \emph{layering} or a \emph{potential} on the vertex set. See Figure \ref{fig:layers} for an illustration. We note that $l$ has a unique extension to a linear function $l\colon\R^V\to\R$ which evaluates to $l(\mathbf h-\mathbf t)=1$ on all vertices of $\mathcal Q_G$. Thus the root polytope of a semi-balanced 
graph always lies in a hyperplane, given as
\begin{equation}\label{eq:inhomogen}
\Pi_l=\{\mathbf x\in\R^V\mid l(\mathbf x)=1\},
\end{equation}
that avoids the origin. Together with the homogeneous relation mentioned after Definition \ref{def:root_polytope}, this ensures that $\dim\mathcal Q_G\le|V|-2$ in semi-balanced cases.

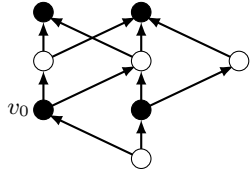
\begin{figure}
\begin{tikzpicture}[scale=.65]
\tikzset{>=latex}

\draw [->, thick] (0, 0) -- (0,0.8);
\draw [->, thick] (0, 0) -- (1.9,0.9);
\draw [->, thick] (2, 0) -- (2,0.8);
\draw [->, thick] (2, 0) -- (3.9,0.9);
\draw [->, thick] (0, 1) -- (0,1.8);
\draw [->, thick] (0, 1) -- (1.9,1.9);
\draw [->, thick] (2, 1) -- (0.1,1.9);
\draw [->, thick] (2, 1) -- (2,1.8);
\draw [->, thick] (4, 1) -- (2.1,1.9);
\draw [->, thick] (2, -1) -- (0.1,-0.1);
\draw [->, thick] (2, -1) -- (2,-0.2);
\draw [fill=white] (2, -1) circle [radius=0.2];
\draw [fill=white] (0, 1) circle [radius=0.2];
\draw [fill=white] (2, 1) circle [radius=0.2];
\draw [fill=white] (4, 1) circle [radius=0.2];
\draw [fill] (0, 0) circle [radius=0.2];
\draw [fill] (2, 0) circle [radius=0.2];
\draw [fill] (0, 2) circle [radius=0.2];
\draw [fill] (2, 2) circle [radius=0.2];

\node at (-.5, 0) {{\small{$v_0$}}};
\end{tikzpicture}
\caption{
A semi-balanced graph of four layers, each arranged on a horizontal line.
}
\label{fig:layers}
\end{figure}

\begin{proof}
	Clearly all three conditions hold if and only if they hold for each connected component. Hence we may suppose that $G$ is connected.
	
	\eqref{harom} $\Rightarrow$ \eqref{ketto}. If there is a layering $l$ then by bipartiteness, $\bigcup_{i \text{ even}} l^{-1}(i)$ gives one of the color classes, say $U$, while odd layers give the other color class $W$. 
	Then $G$ can be obtained from the orientation with each edge oriented toward $W$ by reversing the cuts between $l^{-1}(2k-1)$ and $l^{-1}(2k)$ for each $k$ (note that these are disjoint directed cuts).
	
	The implication \eqref{ketto}$  \Rightarrow  $\eqref{kettoesfel} is clear.
	
	\eqref{kettoesfel} $\Rightarrow$ \eqref{egy}. It is enough to see that the standard orientation 
	is semi-balanced (which is obvious), and that for a semi-balanced graph $G$, the reversal of a directed cut preserves the semi-balanced property. For the latter, note that any cut intersects any cycle in an even number of edges, and if the cut is directed then half of those edges are opposite to the other half. Hence after reversing the edges of a directed cut, still half of the edges of the cycle will go in each direction.
	
	\eqref{egy} $\Rightarrow$ \eqref{harom}.
	Take an arbitrary vertex $v_0$. As $G$ is connected, each vertex $v$ has a walk (a potentially self-intersecting path) leading to it from $v_0$. Choose any walk from $v_0$ to $v$ and let 
\begin{multline*}
	l(v)=\sharp\{\text{edges of the walk pointing toward }v\} \\
	- \sharp\{\text{edges of the walk pointing toward }v_0\} 
\end{multline*}
	This is well defined, since if we take a different walk from $v_0$ to $v$, the symmetric difference of the two walks consists of cycles, and by the semi-balanced property, the above counts cancel each other on cycles. Also, for each edge of $G$, when defining $l$ for one of its endpoints, we can choose a walk that uses the given edge. This implies that $l(h)-l(t)=1$ for each $\overrightarrow{th}\in E(G)$.
	\end{proof}
	
\begin{example}
\label{ex:K34}
The complete bipartite graph $K_{3,4}$ has seven semi-balanced orientations up to isomorphism. 
Figure \ref{fig:K34} shows four of them and we can get the remaining three 
from the first three orientations of the figure by reversing each edge. (The fourth 
orientation is symmetric.) 
Their root polytopes are five-dimensional.

The same enumeration for $K_{2,3}$ finds five 
semi-balanced orientations, with three-dimensional root polytopes. Three of them are shown in Figure \ref{fig:K23}. The first two polytopes are combinatorially equivalent, although not isometric.
(The latter claim is meant with respect to the standard metric on $\R^V$ which, admittedly, is not too relevant to our investigation.)
\end{example}

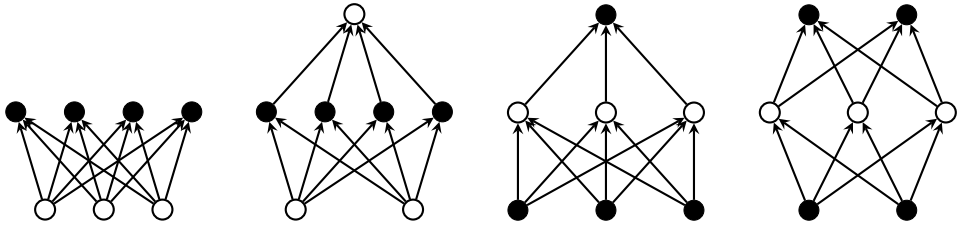
\begin{figure} 
	\begin{center}
	\begin{tikzpicture}[scale=.13]
		\node [circle,fill,scale=.8,draw] (1) at (0,10) {};
		\node [circle,fill,scale=.8,draw] (2) at (6,10) {};
		\node [circle,fill,scale=.8,draw] (3) at (12,10) {};
		\node [circle,fill,scale=.8,draw] (4) at (18,10) {};
		\node [thick,circle,scale=.8,draw] (5) at (3,0) {};		
		\node [thick,circle,scale=.8,draw] (6) at (9,0) {};		
		\node [thick,circle,scale=.8,draw] (7) at (15,0) {};	
		\path [thick,->,>=stealth] (5) edge node {} (1);
		\path [thick,->,>=stealth] (5) edge node {} (2);
		\path [thick,->,>=stealth] (5) edge node {} (3);
		\path [thick,->,>=stealth] (5) edge node {} (4);
		\path [thick,->,>=stealth] (6) edge node {} (1);
		\path [thick,->,>=stealth] (6) edge node {} (2);
		\path [thick,->,>=stealth] (6) edge node {} (3);
		\path [thick,->,>=stealth] (6) edge node {} (4);
		\path [thick,->,>=stealth] (7) edge node {} (1);
		\path [thick,->,>=stealth] (7) edge node {} (2);
		\path [thick,->,>=stealth] (7) edge node {} (3);
		\path [thick,->,>=stealth] (7) edge node {} (4);
		\end{tikzpicture}
\hspace{.5cm}
\begin{tikzpicture}[scale=.13]
\node [circle,fill,scale=.8,draw] (1) at (0,10) {};
		\node [circle,fill,scale=.8,draw] (2) at (6,10) {};
		\node [circle,fill,scale=.8,draw] (3) at (12,10) {};
		\node [circle,fill,scale=.8,draw] (4) at (18,10) {};
		\node [thick,circle,scale=.8,draw] (5) at (3,0) {};		
		\node [thick,circle,scale=.8,draw] (6) at (9,20) {};		
		\node [thick,circle,scale=.8,draw] (7) at (15,0) {};
		\path [thick,->,>=stealth] (5) edge node {} (1);
		\path [thick,->,>=stealth] (5) edge node {} (2);
		\path [thick,->,>=stealth] (5) edge node {} (3);
		\path [thick,->,>=stealth] (5) edge node {} (4);
		\path [thick,<-,>=stealth] (6) edge node {} (1);
		\path [thick,<-,>=stealth] (6) edge node {} (2);
		\path [thick,<-,>=stealth] (6) edge node {} (3);
		\path [thick,<-,>=stealth] (6) edge node {} (4);
		\path [thick,->,>=stealth] (7) edge node {} (1);
		\path [thick,->,>=stealth] (7) edge node {} (2);
		\path [thick,->,>=stealth] (7) edge node {} (3);
		\path [thick,->,>=stealth] (7) edge node {} (4);
		\end{tikzpicture}
\hspace{.5cm}
\begin{tikzpicture}[scale=.13]
\node [circle,fill,scale=.8,draw] (1) at (9,20) {};
		\node [circle,fill,scale=.8,draw] (2) at (0,0) {};
		\node [circle,fill,scale=.8,draw] (3) at (9,0) {};
		\node [circle,fill,scale=.8,draw] (4) at (18,0) {};
		\node [thick,circle,scale=.8,draw] (5) at (0,10) {};		
		\node [thick,circle,scale=.8,draw] (6) at (9,10) {};		
		\node [thick,circle,scale=.8,draw] (7) at (18,10) {};	
		\path [thick,->,>=stealth] (5) edge node {} (1);
		\path [thick,<-,>=stealth] (5) edge node {} (2);
		\path [thick,<-,>=stealth] (5) edge node {} (3);
		\path [thick,<-,>=stealth] (5) edge node {} (4);
		\path [thick,->,>=stealth] (6) edge node {} (1);
		\path [thick,<-,>=stealth] (6) edge node {} (2);
		\path [thick,<-,>=stealth] (6) edge node {} (3);
		\path [thick,<-,>=stealth] (6) edge node {} (4);
		\path [thick,->,>=stealth] (7) edge node {} (1);
		\path [thick,<-,>=stealth] (7) edge node {} (2);
		\path [thick,<-,>=stealth] (7) edge node {} (3);
		\path [thick,<-,>=stealth] (7) edge node {} (4);
		\end{tikzpicture}
\hspace{.5cm}
\begin{tikzpicture}[scale=.13]
\node [circle,fill,scale=.8,draw] (1) at (4,0) {};
		\node [circle,fill,scale=.8,draw] (2) at (14,0) {};
		\node [circle,fill,scale=.8,draw] (3) at (4,20) {};
		\node [circle,fill,scale=.8,draw] (4) at (14,20) {};
		\node [thick,circle,scale=.8,draw] (5) at (0,10) {};		
		\node [thick,circle,scale=.8,draw] (6) at (9,10) {};		
		\node [thick,circle,scale=.8,draw] (7) at (18,10) {};	
		\path [thick,<-,>=stealth] (5) edge node {} (1);
		\path [thick,<-,>=stealth] (5) edge node {} (2);
		\path [thick,->,>=stealth] (5) edge node {} (3);
		\path [thick,->,>=stealth] (5) edge node {} (4);
		\path [thick,<-,>=stealth] (6) edge node {} (1);
		\path [thick,<-,>=stealth] (6) edge node {} (2);
		\path [thick,->,>=stealth] (6) edge node {} (3);
		\path [thick,->,>=stealth] (6) edge node {} (4);
		\path [thick,<-,>=stealth] (7) edge node {} (1);
		\path [thick,<-,>=stealth] (7) edge node {} (2);
		\path [thick,->,>=stealth] (7) edge node {} (3);
		\path [thick,->,>=stealth] (7) edge node {} (4);
\end{tikzpicture}
	\end{center}
	\caption{Some semi-balanced orientations of $K_{3,4}$, starting from the standard one. The vertical coordinate serves as potential.
	}
	\label{fig:K34}
\end{figure} 

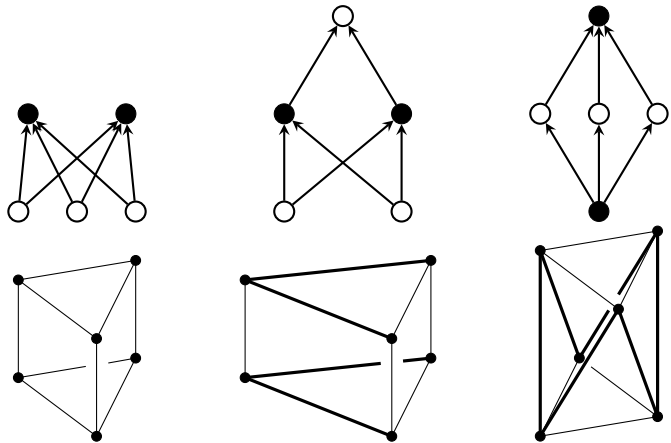
\begin{figure}
    \centering
		\begin{tikzpicture}[scale=.13]
		\node [circle,fill,scale=.8,draw] (2) at (4,10) {};
		\node [circle,fill,scale=.8,draw] (3) at (14,10) {};
		\node [thick,circle,scale=.8,draw] (5) at (3,0) {};		
		\node [thick,circle,scale=.8,draw] (6) at (9,0) {};		
		\node [thick,circle,scale=.8,draw] (7) at (15,0) {};	
		\path [thick,->,>=stealth] (5) edge node {} (2);
		\path [thick,->,>=stealth] (5) edge node {} (3);
		\path [thick,->,>=stealth] (6) edge node {} (2);
		\path [thick,->,>=stealth] (6) edge node {} (3);
		\path [thick,->,>=stealth] (7) edge node {} (2);
		\path [thick,->,>=stealth] (7) edge node {} (3);
		\draw [fill] (3,-7) circle [radius=.5];
		\draw [fill] (15,-5) circle [radius=.5];
		\draw [fill] (11,-13) circle [radius=.5];
		\draw [fill] (3,-17) circle [radius=.5];
		\draw [fill] (15,-15) circle [radius=.5];
		\draw [fill] (11,-23) circle [radius=.5];
		\draw (3,-7) -- (15,-5);
		\draw (3,-7) -- (11,-13);
		\draw (11,-13) -- (15,-5);
		\draw (3,-7) -- (3,-17);
		\draw (15,-15) -- (15,-5);
		\draw (11,-23) -- (11,-13);
		\draw (3,-17) -- (11,-23);
		\draw (11,-23) -- (15,-15);
		\draw (3,-17) -- (9.9,-15.85);
		\draw (12.2,-15.5) -- (15,-15);
		\end{tikzpicture}
\hspace{1cm}
\begin{tikzpicture}[scale=.13]
\node [circle,fill,scale=.8,draw] (1) at (3,10) {};
		\node [circle,fill,scale=.8,draw] (4) at (15,10) {};
		\node [thick,circle,scale=.8,draw] (5) at (3,0) {};		
		\node [thick,circle,scale=.8,draw] (6) at (9,20) {};		
		\node [thick,circle,scale=.8,draw] (7) at (15,0) {};
		\path [thick,->,>=stealth] (5) edge node {} (1);
		\path [thick,->,>=stealth] (5) edge node {} (4);
		\path [thick,<-,>=stealth] (6) edge node {} (1);
		\path [thick,<-,>=stealth] (6) edge node {} (4);
		\path [thick,->,>=stealth] (7) edge node {} (1);
		\path [thick,->,>=stealth] (7) edge node {} (4);
		\draw [fill] (-1,-7) circle [radius=.5];
		\draw [fill] (18,-5) circle [radius=.5];
		\draw [fill] (14,-13) circle [radius=.5];
		\draw [fill] (-1,-17) circle [radius=.5];
		\draw [fill] (18,-15) circle [radius=.5];
		\draw [fill] (14,-23) circle [radius=.5];
		\draw [very thick] (-1,-7) -- (18,-5);
		\draw [very thick] (-1,-7) -- (14,-13);
		\draw (14,-13) -- (18,-5);
		\draw (-1,-7) -- (-1,-17);
		\draw (18,-15) -- (18,-5);
		\draw (14,-23) -- (14,-13);
		\draw [very thick] (-1,-17) -- (14,-23);
		\draw (14,-23) -- (18,-15);
		\draw [very thick] (-1,-17) -- (12.87,-15.54);
		\draw [very thick] (15.15,-15.3) -- (18,-15);
		\end{tikzpicture}
\hspace{1cm}
\begin{tikzpicture}[scale=.13]
\node [circle,fill,scale=.8,draw] (1) at (9,19) {};
		\node [circle,fill,scale=.8,draw] (3) at (9,-1) {};
		\node [thick,circle,scale=.8,draw] (5) at (3,9) {};		
		\node [thick,circle,scale=.8,draw] (6) at (9,9) {};		
		\node [thick,circle,scale=.8,draw] (7) at (15,9) {};	
		\path [thick,->,>=stealth] (5) edge node {} (1);
		\path [thick,<-,>=stealth] (5) edge node {} (3);
		\path [thick,->,>=stealth] (6) edge node {} (1);
		\path [thick,<-,>=stealth] (6) edge node {} (3);
		\path [thick,->,>=stealth] (7) edge node {} (1);
		\path [thick,<-,>=stealth] (7) edge node {} (3);
		\draw [fill] (3,-5) circle [radius=.5];
		\draw [fill] (15,-3) circle [radius=.5];
		\draw [fill] (11,-11) circle [radius=.5];
		\draw [fill] (3,-24) circle [radius=.5];
		\draw [fill] (15,-22) circle [radius=.5];
		\draw [fill] (7,-16) circle [radius=.5];
		\draw (3,-5) -- (15,-3);
		\draw (3,-5) -- (11,-11);
		\draw (11,-11) -- (15,-3);
		\draw (3,-24) -- (7,-16);
		\draw (3,-24) -- (15,-22);
		\draw (8.2,-16.9) -- (15,-22);
		\draw [very thick] (3,-24) -- (11,-11);
		\draw [very thick] (7,-16) -- (10,-11.125);
		\draw [very thick] (11,-9.5) -- (15,-3);
		\draw [very thick] (3,-24) -- (3,-5);
		\draw [very thick] (15,-22) -- (11,-11);
		\draw [very thick] (15,-22) -- (15,-3);
		\draw [very thick] (7,-16) -- (3,-5);
		\end{tikzpicture}
    \caption{Some semi-balanced orientations of $K_{2,3}$ with their root polytopes. In the latter, thicker edges are of length $2$ and thinner ones are of length $\sqrt2$.}
    \label{fig:K23}
\end{figure}

\begin{remark}
Semi-balanced directed graphs and their root polytopes have recently turned up independently in the literature.

Namely, Setiabrata's paper \cite{Setiabrata}, which appeared as a preprint shortly after the first version of ours, investigates faces of root polytopes of type $\tilde{\mathcal{Q}}_G=\conv(\{\,\mathbf0, \,\mathbf{h} - \mathbf{t}\mid \overrightarrow{th}\in E\,\})$ for acyclic digraphs $G$. Semi-balanced subgraphs of $G$ (called \emph{path consistent} in \cite{Setiabrata}) also play a role in this setting: it is proved that some of the faces of $\tilde{\mathcal{Q}}_G$ are of the form $\mathcal{Q}_H$ where $H$ is a semi-balanced subgraph of $G$ satisfying a certain extra condition. 

We also mention that in \cite{Zaslavsky_biased_graphs}, Zaslavsky investigates \emph{biased graphs} and their matroids for certain abstract notions of bias. One example he gives is what he calls \emph{poise bias} for directed graphs. A cycle in a directed graph is poise balanced if and only if it is semi-balanced. Hence semi-balanced graphs are exactly the unbiased graphs for poise bias.
\end{remark}

\section{Dimension count for root polytopes}
\label{sec:dimension}

We present the following basic observation on affine relations between the vertices of the root polytope $\mathcal{Q}_G$ of the 
directed graph $G=(V,E)$.
For now, $G$ is not assumed to be bipartite.

\begin{prop}\label{prop:aff_indep_in_root_polytope}
The vertices of $\mathcal{Q}_G$ corresponding to an edge-set $A\subset E$ form an affine independent set if and only if $A$ contains at most one 
cycle and no semi-balanced cycle.
\end{prop}

\begin{proof}
	We start by showing that a cycle $C$ yields an affine independent vertex set if and only if it is not semi-balanced. 
	Suppose there is a linear relation $\sum_{\varepsilon\in C}\lambda_\varepsilon \mathbf x_\varepsilon=\mathbf0$.
	For consecutive edges $\varepsilon, \varepsilon'$ we have to have $\lambda_\varepsilon+\lambda_{\varepsilon'}=0$ 
	if their orientations in $C$ agree, and $\lambda_\varepsilon-\lambda_{\varepsilon'}=0$ if they are oppositely oriented in $C$. (Otherwise the generator of $\R^V$ that corresponds to their common endpoint does not cancel.) Let $C^+$ be the set of edges oriented consistently with some fixed edge, and $C^-$ the set of edges oriented opposite to the fixed edge.
	By the above, the non-trivial linear relations of $\{\mathbf x_\varepsilon\mid\varepsilon\in C\}$ are exactly 
	\begin{equation}\label{eq:kombinacio}
	\sum_{\varepsilon\in C^+}\lambda \mathbf x_\varepsilon-\sum_{\varepsilon\in C^-}\lambda \mathbf x_\varepsilon=\mathbf0\quad(\lambda\ne0).
	\end{equation}
	Since the sum of the coefficients is $\lambda\cdot (|C^+|-|C^-|)$, we have a non-trivial affine relation if and only if $C$ is semi-balanced.
	
	Now it is easy to show that subgraphs containing at most one un-semibalanced cycle and no semi-balanced cycle correspond to affine independent sets. Indeed, if such a subgraph is not an un-semibalanced cycle, then there will be a degree one node in it. In any affine combination giving $\mathbf{0}$, an edge attached to a degree one node has to have coefficient $0$, for otherwise the component of the degree one node would not be $0$. Successively applying this argument, we get that only the edges of the cycle (if it exists) might have nonzero coefficients, but that is also impossible by the previous paragraph.
	
	Also, since semi-balanced cycles yield affine dependent sets, any subgraph containing a semi-balanced cycle does so, too. 
	Finally, we show that subgraphs containing at least two un-semibalanced cycles give affine dependent sets. Take the two un-semibalanced cycles. They might have edges in common, but there is at least one edge that only occurs in one of them. Consider the linear relations \eqref{eq:kombinacio} for both cycles. By choosing the two $\lambda$ values appropriately, the coefficients will sum to zero, that is by adding the two conditions we get an 
	affine relation. In 
	it, the coefficient of the edge occurring in only one of the cycles will be nonzero.
\end{proof}

\begin{cor}\label{cor:max_simplices:of_root_polytope}
	If $G$ is connected and semi-balanced, then the maximal affine independent vertex sets in the root polytope correspond to spanning trees, and the dimension of the root polytope is $|V|-2$. If $G$ is connected and  un-semibalanced, then the dimension is $|V|-1$. If $G$ is semi-balanced and has $c(G)$ connected components, then the dimension of its root polytope is $|V|-1-c(G)$.
\end{cor}

This certainly matches Examples \ref{ex:K34} and \ref{ex:K4}.
    Note that in the above statement, and throughout this paper, a \emph{spanning tree} means a subgraph that is a spanning tree in the undirected sense (i.e., connected and cycle-free), regardless of how the edges of the tree are oriented. 

Since we will frequently use the affine dependence of the vectors corresponding to the edges of a semi-balanced cycle, let us state that result as a separate lemma.

\begin{lemma}\label{l:dependence_of_semibalanced_cycle}
	For any semi-balanced cycle $C$, 
	the relation 
	\[\sum_{\varepsilon\in C^+} \mathbf x_\varepsilon - \sum_{\varepsilon\in C^-} \mathbf x_\varepsilon  = \mathbf 0,\] 
	cf.\ \eqref{eq:kombinacio}, is an affine dependence.
\end{lemma}

\begin{remark}
    We may consider analogous questions 
    for root polytopes of signed graphs, as defined in Remark \ref{rem:signed_root_poly_def}. The bipartite case is equivalent to the above. In the non-bipartite case, it is easy to check that an odd cycle yields an affine independent vertex set in this alternative root polytope. Thus the polytope of a connected, non-bipartite signed graph $G$ always has dimension $|V(G)|-1$.
\end{remark}

Trees form an important special class of bipartite graphs, and all directed trees are automatically semi-balanced. We have seen that any directed tree $T$ on the vertex set $V$ yields a root polytope $\mathcal Q_T$ that is a simplex of dimension $|V|-2$. By adding the origin as a vertex, we get simplices $\tilde{\mathcal Q}_T$ of dimension $|V|-1$.
Some standard arguments, very similar to Proposition \ref{prop:aff_indep_in_root_polytope}, yield the following.

\begin{lemma}\label{lem:unimodularitas}
For any directed tree $T$, the simplices $\mathcal Q_T$ and $\tilde{\mathcal Q}_T$ are unimodular. In particular, 
for all (spanning) trees $T$ on a fixed vertex set $V$, the simplices $\tilde{\mathcal Q}_T$ have the same $(|V|-1)$-dimensional volume.
\end{lemma}

Here unimodularity is meant with respect to the lattices $\langle\mathbf1\rangle^\perp\cap\Pi_l\cap\Z^V$ and $\langle\mathbf1\rangle^\perp\cap\Z^V$, respectively, where $l\colon V\to\Z$ is a potential compatible with $T$, cf.\ \eqref{eq:inhomogen}. Recall that a \emph{unimodular simplex} is one whose vertices form an affine basis of the lattice, which in turn means that the vectors, pointing from an arbitrarily fixed vertex to the other vertices, constitute a linear basis of the lattice over $\Z$. We note that the root polytope of an un-semibalanced cycle is a simplex that fails to be unimodular.

\begin{proof}
In the case of $\tilde{\mathcal Q}_T$, it suffices to prove that the collection $\{\mathbf x_\varepsilon\mid\varepsilon\in T\}$ generates $\langle\mathbf1\rangle^\perp\cap\Z^V$ using integer coefficients. Indeed, one inductively finds the (unique, integer) solution to any such system of equations by working inward from the leaf edges of $T$. As $\mathcal Q_T$ is a face of $\tilde{\mathcal Q}_T$, it is itself unimodular. Finally, as unimodular simplices are transformed into one another by translations and invertible $\Z$-linear transformations (thus of determinant $\pm1$), they have equal volume.
\end{proof}

\begin{cor}
\label{cor:unimod}
All cycle-free sets of edges $S$ give rise to unimodular simplices $\mathcal Q_S$ and $\tilde{\mathcal Q}_S$.
\end{cor}

\begin{proof}
Any such set $S$ of edges can be extended to a tree $T$ so that $\mathcal Q_S$ and $\tilde{\mathcal Q}_S$ are faces of the unimodular simplices $\mathcal Q_T$ and $\tilde{\mathcal Q}_T$, respectively.
\end{proof}

\begin{cor}\label{cor:azonos_terfogat}
For all directed trees $T$ on the vertex set $V$, compatible with a fixed layering $l\colon V\to\Z$ (such as the spanning trees of a fixed semi-balanced graph), the simplices $\mathcal Q_T$ have the same $(|V|-2)$-dimensional volume.
\end{cor}

\begin{proof}
This follows either by unimodularity (with respect to the same lattice) directly, or by noting that the corresponding simplices $\tilde{\mathcal Q}_T$ have identical volumes as well as heights (with respect to the origin).
\end{proof}


The contents of this section were not meant to surprise the experts. For example,
instances of 
Lemma \ref{lem:unimodularitas} and Corollary \ref{cor:azonos_terfogat} play a big role in Postnikov's treatment \cite{alex} of 
bipartite graphs with a standard orientation.

\section{Ehrhart theory and the interior polynomial}
\label{sec:int_poly}

The interior polynomial first surfaced as an invariant of hypergraphs \cite{hiperTutte}. It is an elementary concept, inspired directly by the Tutte polynomial of graphs. Any bipartite graph induces two hypergraphs and it turns out that their interior polynomials coincide, in fact they both agree with the so called $h^*$-vector or $h^*$-polynomial of the root polytope of the standard orientation of the graph \cite{KP_Ehrhart}.

The purpose of this section is to recall the $h^*$-vector and some surrounding notions, and to introduce the following obvious extension.

\begin{defn}[Interior polynomial]
Let $G$ be a  connected 
semi-balanced (in particular, directed and bipartite) graph. Its \emph{interior polynomial} is defined to be the $h^*$-polynomial of the root polytope of $G$, and it is denoted by $I_G$. When $G$ is single isolated point (so that $\mathcal Q_G=\emptyset$), then we let its interior polynomial be $1$.
\end{defn}

Most of the rest of the paper is devoted to the computation and properties of the interior polynomial $I$.
We will discuss how to extend $I$ to disconnected semi-balanced graphs 
at the end of the paper, cf.\ \eqref{eq:disconnect}.
The only reason why we do not make the above definition for all directed graphs is that we do not have any claims to make at that level of generality. We do have a surprisingly strong tool though in the semi-balanced case: ribbon structures, which we will start discussing in the next section.

\begin{remark}
Kato \cite{kato_signed} defined the \emph{signed interior polynomial} $I^+$ for signed bipartite graphs, which is also a generalization of the interior polynomial of (unsigned/standardly oriented) bipartite graphs. His polynomial is different from our interior polynomial (even if we switch to the language of signed graphs).
\end{remark}

Let us now recall 
$h^*$-polynomials and 
a particular way of computing them, using a shellable dissection into unimodular simplices.

For any $d$-dimensional 
polytope $Q\subset\R^n$ with vertices in $\Z^n$, its \emph{$h^*$-polynomial} $\sum_{i=0}^d h^*_i t^i$, also commonly called the \emph{$h^*$-vector}, is defined by Ehr\-hart's identity 
\begin{equation}
\label{eq:h-csillag}
\sum_{i=0}^d h^*_i t^i = (1-t)^{d+1} \ehr_Q(t),
\quad\text{where}\quad 
\ehr_Q(t)=\sum_{k=0}^\infty|(k\cdot Q)\cap\Z^n|\,t^k
\end{equation}
is known as the \emph{Ehrhart series} of $Q$.

Now suppose that 
$Q$ is subdivided into unimodular simplices of the lattice.
If the subdivision is a \emph{triangulation}, i.e., any two of its simplices intersect in a common face, then the $h^*$-polynomial is related to the $h$-vector of the triangulation via
\begin{equation}
\label{eq:h-star}
\sum_{i=0}^d h^*_i t^i =t^{d+1} h(1/t).
\end{equation}
Here the \emph{$h$-vector} or \emph{$h$-polynomial} is defined as
\begin{equation}
\label{eq:h-vektor}
h(x)=f(x-1),\quad\text{where}\quad f(y)=y^{d+1}
+\sum_{\begin{subarray}{c}
F:\text{ a non-empty face}\\
\text{in the triangulation}\end{subarray}}y^{d-\dim F}.
\end{equation}
The $h$- and $f$-polynomials can be, and often are, considered for abstract simplicial complexes as well, but the $h^*$-polynomial is specific to lattice polytopes. 

We have already used the technique in \cite{hyperBernardi}, but it is worth repeating (and making explicit in Proposition \ref{prop:h-dissect} below), that the $h$-vector of a 
shellable dissection fits into \eqref{eq:h-star}
equally well.

A \emph{dissection} of a polytope is a set of mutually interior-disjoint maximal dimensional simplices whose union is the polytope.
Furthermore, a dissection is \emph{shellable} if it admits a \emph{shelling order}. That in turn means that the simplices of the dissection are listed as $\sigma_1,\sigma_2,\ldots,\sigma_N$, in such a way that for each $i=2,3,\ldots,N$, the intersection of $\sigma_i$ with the `earlier' simplices,
\[ \sigma_i\cap\left(\bigcup_{j=1}^{i-1}\sigma_j\right), \]
coincides with the union of a positive number of facets (codimension $1$ faces) of $\sigma_i$. 

Let us denote the number of said facets by $r_i$ for $i=2,3,\ldots,N$ and let us also put $r_1=0$. Then the \emph{$h$-vector} associated to the shelling order is the distribution of the statistic $r_i$. 
We may write it as the finite sequence $(h_0,h_1,\ldots)$, where 
\begin{equation}
\label{eq:h-egyutthato}    
\text{$h_k$ is the number of simplices $\sigma_i$ with $r_i=k$.}
\end{equation}
Note that $h_0=1$ and since each maximal simplex has $d+1$ facets, the subscript of the last non-zero term is at most $d+1$. In fact, it is at most $d$ by a topological argument: each $\sigma_i$ with $r_i=d+1$ (i.e., a $d$-cell attached along its entire boundary)
would add an infinite cyclic summand to the $d$-dimensional homology group of the polytope, but that group is $0$.

Alternatively, we may express the $h$-vector in the polynomial form
\begin{equation}
\label{eq:h-hejazott}
h(x)=h_{d}\,x+h_{d-1}\,x^2+\cdots+h_1\,x^{d}+h_0\,x^{d+1}.
\end{equation}
Mind the flip of order and that the constant term is $0$; the latter choice is related to the fact that for a triangulated polytope, \eqref{eq:h-vektor} provides $h(0)=f(-1)=0$ for Euler characteristic reasons.
Indeed, when the dissection is a triangulation, \eqref{eq:h-hejazott} is consistent with 
\eqref{eq:h-vektor}; in fact, our definition \eqref{eq:h-hejazott} is a direct extension of a well known description of the $h$-vector for shellable triangulations. In particular, the $h$-vector is independent of the shelling order in the case of triangulations.

Now in the special case of unimodular simplices, all formulas of this section so far are essentially the same.

\begin{prop}
\label{prop:h-dissect}
For any shellable dissection of a lattice polytope into unimodular simplices, and for any shelling order, the $h$-vector \eqref{eq:h-hejazott} is equivalent, via \eqref{eq:h-star}, to the $h^*$-vector \eqref{eq:h-csillag} of the polytope.
\end{prop}

\begin{proof}
This will be a straightforward adaptation of the standard proof of \eqref{eq:h-star} for shellable triangulations. The key ingredient is that a $k$ times dilated, $d$-dimensional unimodular simplex contains ${k+d \choose d}$ lattice points, and if the points along $r$ of its facets are not to be counted, then this drops to $k-r+d \choose d$. Another detail is that by $d$ times differentiating the identity $\frac{1}{1-t}=\sum_{p=0}^\infty t^p$, we obtain $\frac{1}{(1-t)^{d+1}}=\sum_{p=0}^\infty{p+d \choose d}t^p$. With that, we do get
\begin{multline*}
h^*(t)=(1-t)^{d+1}\sum_{k=0}^\infty|(k\cdot Q)\cap\Z^n|\,t^k
=(1-t)^{d+1}\sum_{k=0}^\infty\left[\sum_{i=1}^N{k-r_i+d \choose d}\right]t^k\\
=(1-t)^{d+1}\sum_{k=0}^\infty\left[\sum_{j=0}^d h_j{k-j+d \choose d}\right]t^k
=\sum_{j=0}^d h_j\left[(1-d)^{d+1}\sum_{k=j}^\infty{k-j+d \choose d}t^k\right]\\
=\sum_{j=0}^d h_j\,t^j\left[(1-d)^{d+1}\sum_{p=0}^\infty{p+d \choose d}t^p\right]=\sum_{j=0}^d h_j\,t^j=t^{d+1}h(1/t),
\end{multline*}
just as we claimed.
\end{proof}

\begin{remark}
We set up the $f$-vector as we did in order to make its relation \eqref{eq:h-vektor} to the $h$-vector as simple as possible. This choice determined the form of \eqref{eq:h-hejazott} and also resulted in the presence of $1/t$ in \eqref{eq:h-star}. In the end, however, the two reversals of order cancel and we end up with the completely natural formula \begin{equation}
\label{eq:generator}
h^*(t)=\sum_{j=0}^d h_j\,t^j,
\end{equation}
where $h_j$ is as in \eqref{eq:h-egyutthato}, for the $h^*$-vector of a lattice polytope with a shellable dissection into unimodular simplices.
\end{remark}

\begin{example}
\label{ex:trees}
Any orientation of a tree $T$ is semi-balanced and, according to Lemma \ref{lem:unimodularitas}, in this case $\mathcal Q_T$ is a single unimodular simplex. In particular, trees have the interior polynomial $I_T(x)=1$.
\end{example}

In the next two sections, we implement the strategy above in the case of the root polytope of an arbitrary semi-balanced graph.

\section{Jaeger trees and dissections}
\label{sec:Jaeger_trees}

We have seen in Proposition \ref{prop:aff_indep_in_root_polytope} that for semi-balanced graphs, the maximal simplices in the root polytope correspond to the spanning trees. We have also just pointed out that when analyzing polytopes, it is very helpful to dissect them into simplices. Hence it would be useful to identify sets of spanning trees whose simplices dissect the root polytope. 

In the case of the standard orientation, our previous work showed that the so-called Jaeger trees induce a dissection of the root polytope, moreover, this dissection is shellable \cite{hyperBernardi}. 
Here we generalize Jaeger trees for directed graphs, and show that the analogous statements hold in the semi-balanced case, too. Another goal of this section is to simplify some of the proofs of \cite{hyperBernardi}.

To define Jaeger trees, we need to fix a ribbon structure. For a (not necessarily directed) 
graph $G$, a \emph{ribbon structure} is a family of cyclic permutations: namely for each vertex $x$ of $G$, a cyclic permutation of the edges incident to $x$ is given.
For an edge $xy$ of $G$, we use the following notations: 
\begin{itemize}
	\item $yx^+_G$: the edge following $yx$ at $y$
	\item $xy^+_G$: the edge following $xy$ at $x$.
\end{itemize}
If $G$ is clear from the context, we omit the subscript.

In addition to a ribbon structure, we also need to fix a \emph{basis} $(b_0,b_0b_1)$, where $b_0$ is an arbitrary node of the graph and $b_0b_1$ is an arbitrary edge incident to $b_0$. 
Even if $G$ is directed, it does not matter whether $b_0$ is the head or the tail of $b_0b_1$.

Suppose that a ribbon structure and a basis are fixed. Then any spanning tree $T$ of $G$ gives us a natural ``walk'' in the graph. This was defined by Bernardi \cite{Bernardi_first}, and following him we call it the \emph{tour} of $T$.

\begin{defn}[Tour of a spanning tree] \label{def:tour_of_a_tree}
Let $G$ be a ribbon graph with a basis $(b_0,b_0b_1)$, and let $T$ be a spanning tree of $G$.
The \emph{tour} of $T$ is a sequence of node-edge pairs, starting with $(b_0, b_0b_1)$. If the current node-edge pair is $(x,xy)$ and $xy\notin T$, then the current node-edge pair of the next step is $(x,xy^+)$. If the current node-edge pair is $(x,xy)$ and $xy\in T$, then the current node-edge pair of the next step is $(y,yx^+)$. In the first case we say that the tour \emph{skips} $xy$ and in the second case we say that the tour \emph{traverses} $xy$. The tour stops right before when $(b_0,b_0b_1)$ would once again become the current node-edge pair. 
\end{defn}

\begin{figure} 
	\begin{center}
		\begin{tikzpicture}[-,>=stealth',auto,scale=0.4,
		thick]
		\tikzstyle{o}=[circle,draw]
		\node[o] (1) at (8, 0) {{\small $v_1$}};
		\node[o] (2) at (4, 1.5) {{\small $v_2$}};
		\node[o] (3) at (0, 0) {{\small $v_3$}};
		\node[o] (4) at (4, -1.5) {{\small $v_4$}};
		\path[->,every node/.style={font=\sffamily\small}, line width=0.8mm]
		(1) edge node [above] {$\varepsilon_1$} (2)
		(3) edge node [below] {$\varepsilon_3$} (4)
		(2) edge node {$\varepsilon_5$} (4);
		\path[->,every node/.style={font=\sffamily\small},dashed]
		(4) edge node [below] {$\varepsilon_4$} (1)
		(2) edge node [above] {$\varepsilon_2$} (3);
		\end{tikzpicture}
	\end{center}
	\caption{The tour of a spanning tree. Let the ribbon structure be the one induced by the positive orientation of the plane. The edges of the tree are drawn by thick lines, the non-edges by dashed lines. With $b_0=v_1, b_1=v_2$, we get the tour $(v_1,\varepsilon_1)$, $(v_2,\varepsilon_2)$, $(v_2,\varepsilon_5)$, $(v_4,\varepsilon_3)$, $(v_3,\varepsilon_2)$, $(v_3,\varepsilon_3)$, $(v_4,\varepsilon_4)$, $(v_4,\varepsilon_5)$, $(v_2,\varepsilon_1)$, $(v_1,\varepsilon_4)$. 
	This is a Jaeger tree because $(v_2,\varepsilon_2)$ precedes $(v_3,\varepsilon_2)$ and $(v_4,\varepsilon_4)$ precedes $(v_1,\varepsilon_4)$.}
	\label{fig:tour_of_a_tree}
\end{figure}
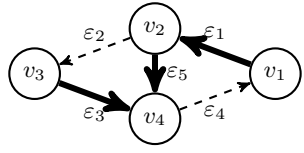 

See Figure \ref{fig:tour_of_a_tree} for an example. 
Bernardi proved {\cite[Lemma 5]{Bernardi_first}} that in the tour of a spanning tree $T$, each edge $xy$ of $G$ becomes current edge twice, in one case with $x$ as current vertex, and in the other case with $y$ as current vertex.
This naturally induces the following orders.

\begin{defn}[$T$-order of the edges of $G$] \label{def:T-order}
In a ribbon directed graph with a basis, we associate an ordering of the edges to any spanning tree $T$, by saying that $\overrightarrow{th}$ is smaller than $\overrightarrow{t'h'}$ if in the tour of $T$, the node-edge pair $(t,th)$ becomes current before the node-edge pair $(t',t'h')$.
We call this the \emph{$T$-order}, and denote it by $\leq_T$. 
\end{defn}

We will often think of the tour of a spanning tree $T$ either as a way of `cutting $T$ out from the graph,' or as a way of `constructing $T$.' In particular, a non-edge $xy$ of $T$ (one that the tour skips) is thought of as `cut from $G$,' and if we have $(x,xy)<_T(y,xy)$, then we say that \emph{$xy$ is cut from the direction of $x$}.

With this we arrive at 
the following notion, which plays a crucial role in this paper. We trace it back to Fran\c cois Jaeger's work \cite{Jaeger} on the Homfly polynomial.

\begin{defn}[Jaeger tree]
	In a directed graph with a ribbon structure and basis, we call a spanning tree $T$ a \emph{Jaeger tree} if for each edge $\overrightarrow{th}\notin T$, the tour of $T$ has $(t,th)$ as a current node-edge pair before $(h,th)$. In other words, in the tour of a Jaeger tree, each non-edge is first seen at, and cut from the direction of, its tail.
\end{defn}

For examples, see Figures \ref{fig:tour_of_a_tree}, \ref{fig:dissect}, and \ref{f:Jaeger_tree_and_duality}. (The graph in Figure \ref{fig:tour_of_a_tree} is un-semibalanced, but Jaeger trees make sense in that case, too.)
We explored Jaeger trees of bipartite graphs equipped with the standard orientation in our earlier work \cite{hyperBernardi}. 
Apart from having desirable combinatorial and geometric properties, Jaeger trees also generalize some well-known constructions.
For instance, 
the so called non-crossing trees of complete bipartite graphs are
exactly the
Jaeger trees for a natural ribbon structure; see Section \ref{sec:layer_complete} for more on this point. 

The typical situation is that we fix a ribbon structure and a basis, and look at all Jaeger trees with this data. For different bases, the set of Jaeger trees can be very different. We hardly care about the relationship between Jaeger trees with different bases. We do however sometimes like to fix the color of the base point. In these cases, the following observation is useful.

\begin{lemma}
\label{lem:atcsuszo_kezdopont}
Suppose that in our ribbon directed graph, the base node $b_0$ is the head of the base edge $\overrightarrow{b_1b_0}$. Then the Jaeger trees for this basis agree with those for the basis $(b_1,b_1b_0^+)$.
\end{lemma}

\begin{proof}
With the original basis, the Jaeger condition forces $b_0b_1$ to be part of all Jaeger trees. The same holds for the other basis, for if $b_0b_1$ was a non-edge of the Jaeger tree $T$, then by the Jaeger condition it would be cut at its tail $b_1$, whereby the next current pair in the tour of $T$ would be the initial $(b_1,b_1b_0^+)$. That means that the tour would finish without including the pair $(b_0,b_0b_1)$, contradicting {\cite[Lemma 5]{Bernardi_first}}.

Now for spanning trees that contain $b_0b_1$, it is easy to see that any non-edge
becomes current in conjunction with its head and with its tail in the same order, regardless of which of the two tours of the tree we consider. 
\end{proof}

For any spanning tree $T$ of $G$ and edge $\varepsilon\in T$, the graph $T-\varepsilon$ has two connected components, and the edges of $G$ connecting them form the \emph{fundamental cut} $C^*(T,\varepsilon)$. We call the component of $T-\varepsilon$ containing $b_0$ the \emph{base component}. 
The following property of fundamental cuts of Jaeger trees (an updated version of \cite[Lemma 6.14]{hyperBernardi}) makes them very useful to us.

\begin{lemma}
	\label{l:char_Jaeger_cuts}
	Let $T$ be a Jaeger tree of a directed ribbon graph, $\varepsilon \in T$,
	and let $\varepsilon_1, \varepsilon_2$ be edges in the fundamental cut $C^*(T,\varepsilon)$. If $\varepsilon_1$ has its tail in the base component of $T-\varepsilon$ and $\varepsilon_2$ has its head in the base component of $T-\varepsilon$, then we have
	\begin{enumerate}[label=(\roman*)]
		\item \label{egyegy} $\varepsilon_1 <_{T} \varepsilon_2$, and
		\item \label{kettoketto} $\varepsilon_1 \leq_{T} \varepsilon$.
	\end{enumerate}
	In particular, if $\varepsilon$ has its tail in the base component, then $\varepsilon_1\leq_T \varepsilon<_T \varepsilon_2$. On the other hand, if $\varepsilon$ has its head in the base component, then $\varepsilon_2\leq_T \varepsilon$.
\end{lemma}

\begin{proof}
	Let us call $T_0$ the component of $T-\varepsilon$ containing $b_0$ and let $T_1$ be the other component. 
	Until the tour of $T$ reaches $\varepsilon$, the current node is in $T_0$. After the first and before the second traversal of $\varepsilon$, the current node is in $T_1$. Finally, after the second traversal of $\varepsilon$ it is once again in $T_0$. Hence each edge of $C^*(T,\varepsilon)$ becomes current with its endpoint in $T_1$ as current node between the two traversals of $\varepsilon$. Since $T$ is a Jaeger tree, each non-tree edge first becomes current with its tail as current node. This means that those edges that have their head in $T_1$ become current with their tail as current node before the first traversal of $\varepsilon$. This implies both statements \ref{egyegy} and \ref{kettoketto} of the Lemma. Finally, if $\varepsilon$ has its head in $T_0$, then it only becomes current with its tail as current node at the time of its second traversal, which is after all other edges of the fundamental cut.
\end{proof}

\begin{prop}\label{p:trees_interior_disjoint}
	Let $G$ be a directed graph and fix a ribbon structure with basis $(b_0,b_0b_1)$.
	Then the corresponding Jaeger trees $T$ give rise to simplices $\mathcal Q_T$ whose (relative) interiors are disjoint in $\mathcal Q_G$. 
\end{prop}

Note that in un-semibalanced cases, the simplices in question are of codimension $1$ in $\mathcal Q_G$. When the graph is semi-balanced, they are of codimension $0$. The proof of Proposition \ref{p:trees_interior_disjoint} is essentially the same as that of \cite[Lemma 7.7]{hyperBernardi}.


\begin{proof}
	Take two Jaeger trees $T_1$ and $T_2$, and suppose that their tours coincide until reaching the edge $\varepsilon$, where $\varepsilon\notin T_1$ and $\varepsilon\in T_2$. It suffices to show that the simplices that correspond to $T_1$ and $T_2$ are separated by a hyperplane. Such a separating hyperplane can be constructed using the fundamental cut $C^*(T_2,\varepsilon)$ of $G$. 
	
	The cut $C^*(T_2,\varepsilon)$ divides $G$ into two components. Let $N_0$ be the set of nodes in the component containing $b_0$, and $N_1$ the nodes of the other component.
    Let us associate the real number $0$ to basis vectors corresponding to elements of $N_0$, and let us associate $1$ to elements of $N_1$. This has a unique linear extension $\kappa=\kappa_{T_2,\varepsilon}\colon\R^V\to\R$. The value of $\kappa$ at the vertices of $\mathcal Q_G$ (described in terms of the corresponding edges of $G$) is then
	\begin{itemize}
		\item $0$ for edges outside the fundamental cut $C^*(T_2,\varepsilon)$: These edges $\overrightarrow{th}$ connect vertices with the same $\kappa$ value, giving $\kappa(\mathbf{h}-\mathbf{t})=0$.
		\item $1$ for $\varepsilon$: As $T_1$ is a Jaeger tree and $\varepsilon\notin T_1$, the edge $\varepsilon$ has to have its tail in $N_0$ and its head in $N_1$. Hence $\kappa(\mathbf x_\varepsilon)=1-0=1$.
		\item $-1$ for edges of $C^*(T_2,\varepsilon)\cap T_1$: By Lemma \ref{l:char_Jaeger_cuts} applied to $T_2$, by the time $\varepsilon$ is traversed in the tour of $T_2$, all the edges of $C^*(T_2,\varepsilon)$ having their tail in $N_0$ are cut. As the tours of $T_1$ and $T_2$ coincide until reaching $\varepsilon$, these edges are also cut in the tour of $T_1$. Hence all the edges in $C^*(T_2,\varepsilon)\cap T_1$ have their heads in $N_0$ and their tails in $N_1$, and the value of $\kappa$ on such edges is indeed $-1$.
	\end{itemize}
	Thus, all the vertices of $\mathcal Q_{T_1}$ lie either in the hyperplane $\kappa = 0$ or on the side where $\kappa$ is negative, and all the vertices of $\mathcal Q_{T_2}$ lie either in the hyperplane $\kappa=0$ or on the side where $\kappa$ is positive. Hence indeed, the hyperplane $\kappa=0$ shows that the simplices corresponding to $T_1$ and $T_2$ are interior disjoint.
\end{proof}

\begin{remark}
\label{rem:hyperplane_from_cut}
More generally, for any cut in any directed graph, one may consider the linear functional $\kappa$ that takes the values $0$ and $1$ on generators corresponding to vertices on the two `sides' of the cut. The kernel of $\kappa$ then contains all those vertices of the root polytope that correspond to edges \emph{not} in the cut. The vectors representing the edges of the cut lie on either side of the kernel, depending on the direction in which the edge `crosses the cut.'
\end{remark}

The main claim of this section is the following.

\begin{thm}\label{t:J-trees_form_quasitr}
	For a connected, semi-balanced directed graph, with arbitrarily fixed ribbon structure and basis, the simplices that correspond to Jaeger trees dissect the root polytope.
\end{thm}

\begin{example}
\label{ex:dissect}
Let us consider the complete bipartite graph $K_{2,3}$ with one of its semi-balanced orientations from Example \ref{ex:K34}. Referring to Figure \ref{fig:dissect}, we let the ribbon structure be given by counterclockwise rotations around each vertex. 
Using the notation of the Figure, we choose $p$ to be the base node and $pc$ to be the base edge. Then, there exist four Jaeger trees with this data. We list them and the corresponding maximal simplices of the root polytope, which do indeed form a dissection. It is in fact a triangulation, by four tetrahedra that are incident to a fixed main diagonal of the octahedron.
\end{example}

\begin{figure}
    \centering
	\begin{tikzpicture}[scale=.13]
\node [circle,fill,scale=.8,draw] (1) at (9,17) {\color{white}$q$};
		\node [circle,fill,scale=.8,draw] (3) at (9,3) {\color{white}$p$};
		\node [thick,circle,scale=.8,draw] (5) at (3,10) {$a$};	
		\node [thick,circle,scale=.8,draw] (6) at (9,10) {$b$};	
		\node [thick,circle,scale=.8,draw] (7) at (15,10) {$c$};	
		\path [thick,->,>=stealth] (5) edge node {} (1);
		\path [thick,<-,>=stealth] (5) edge node {} (3);
		\path [thick,->,>=stealth] (6) edge node {} (1);
		\path [thick,<-,>=stealth] (6) edge node {} (3);
		\path [thick,->,>=stealth] (7) edge node {} (1);
		\path [thick,<-,>=stealth] (7) edge node {} (3);
		\draw [<-,>=stealth] (9,0) arc [radius=3, start angle=270, end angle=210];
		\draw [fill] (3,-5) circle [radius=.5];
		\node [above] at (3,-5) {\small $\mathbf q-\mathbf a$};
		\draw [fill] (15,-3) circle [radius=.5];
		\node [above] at (15,-3) {\small $\mathbf q-\mathbf c$};
		\draw [fill] (11,-11) circle [radius=.5];
		\node [right] at (11,-11) {\small $\mathbf q-\mathbf b$};
		\draw [fill] (3,-24) circle [radius=.5];
		\node [below] at (3,-24) {\small $\mathbf c-\mathbf p$};
		\draw [fill] (15,-22) circle [radius=.5];
		\node [below] at (15,-22) {\small $\mathbf a-\mathbf p$};
		\draw [fill] (7,-16) circle [radius=.5];
		\node [left] at (7,-16) {\small $\mathbf b-\mathbf p$};
		\draw (3,-5) -- (15,-3);
		\draw (3,-5) -- (11,-11);
		\draw (11,-11) -- (15,-3);
		\draw (3,-24) -- (7,-16);
		\draw (3,-24) -- (15,-22);
		\draw (8.2,-16.9) -- (15,-22);
		\draw (3,-24) -- (11,-11);
		\draw (7,-16) -- (10,-11.125);
		\draw (11,-9.5) -- (15,-3);
		\draw (3,-24) -- (3,-5);
		\draw (15,-22) -- (11,-11);
		\draw (15,-22) -- (15,-3);
		\draw (7,-16) -- (3,-5);
		\end{tikzpicture}
		\hspace{.5cm}
		\raisebox{.4cm}{
		\begin{tikzpicture}[scale=.13]
\node [circle,fill,scale=.8,draw] (1) at (9,17) {};
		\node [circle,fill,scale=.8,draw] (3) at (9,3) {};
		\node [thick,circle,scale=.8,draw] (5) at (3,10) {};	
		\node [thick,circle,scale=.8,draw] (6) at (9,10) {};	
		\node [thick,circle,scale=.8,draw] (7) at (15,10) {};	
		\path [thick,->,>=stealth] (5) edge node {} (1);
		\path [thick,<-,>=stealth] (5) edge node {} (3);
		\path [thick,->,>=stealth] (6) edge node {} (1);
		\path [thick,->,>=stealth] (7) edge node {} (1);
		\draw [fill] (3,-5) circle [radius=.5];
		\draw [fill] (15,-3) circle [radius=.5];
		\draw [fill] (11,-11) circle [radius=.5];
		\draw [lightgray,fill] (3,-24) circle [radius=.5];
		\draw [fill] (15,-22) circle [radius=.5];
		\draw [lightgray,fill] (7,-16) circle [radius=.5];
		\draw [very thick] (3,-5) -- (15,-3);
		\draw [very thick] (3,-5) -- (11,-11);
		\draw [very thick] (11,-11) -- (15,-3);
		\draw [lightgray] (3,-24) -- (7,-16);
		\draw [lightgray] (3,-24) -- (15,-22);
		\draw [lightgray] (8.2,-16.9) -- (15,-22);
		\draw [lightgray]
		(3,-24) -- (11,-11);
		\draw [lightgray]
		(7,-16) -- (8.36,-13.79);
		\draw [lightgray] (9.24,-12.36) -- (9.88,-11.32);
		\draw [lightgray]
		(10.76,-9.89) -- (15,-3);
		\draw [lightgray]
		(3,-24) -- (3,-5);
		\draw [very thick] 
		(15,-22) -- (11,-11);
		\draw [very thick] 
		(15,-22) -- (15,-3);
		\draw [lightgray]
		(7,-16) -- (3,-5);
		\draw [very thick] (3,-5) -- (9,-13.5);
		\draw [very thick] (9.6,-14.35) -- (15,-22);
	\end{tikzpicture}
	\hspace{.3cm}
	\begin{tikzpicture}[scale=.13]
\node [circle,fill,scale=.8,draw] (1) at (9,17) {};
		\node [circle,fill,scale=.8,draw] (3) at (9,3) {};
		\node [thick,circle,scale=.8,draw] (5) at (3,10) {};	
		\node [thick,circle,scale=.8,draw] (6) at (9,10) {};	
		\node [thick,circle,scale=.8,draw] (7) at (15,10) {};	
		\path [thick,->,>=stealth] (5) edge node {} (1);
		\path [thick,<-,>=stealth] (5) edge node {} (3);
		\path [ultra thick,<-,>=stealth] (6) edge node {} (3);
		\path [thick,->,>=stealth] (7) edge node {} (1);
		\draw [fill] (3,-5) circle [radius=.5];
		\draw [fill] (15,-3) circle [radius=.5];
		\draw [lightgray,fill] (11,-11) circle [radius=.5];
		\draw [lightgray,fill] (3,-24) circle [radius=.5];
		\draw [fill] (15,-22) circle [radius=.5];
		\draw [fill] (7,-16) circle [radius=.5];
		\draw [very thick] (3,-5) -- (15,-3);
		\draw [lightgray] (3,-5) -- (11,-11);
		\draw [lightgray] (11,-11) -- (15,-3);
		\draw [lightgray] (3,-24) -- (7,-16);
		\draw [lightgray] (3,-24) -- (15,-22);
		\draw [very thick] (8.2,-16.9) -- (15,-22);
		\draw [lightgray]
		(3,-24) -- (11,-11);
		\draw [very thick] 
		(7,-16) -- (8.36,-13.79);
		\draw [very thick] (9.24,-12.36) -- (9.88,-11.32);
		\draw [very thick] 
		(10.76,-9.89) -- (15,-3);
		\draw [lightgray]
		(3,-24) -- (3,-5);
		\draw [lightgray]
		(15,-22) -- (11,-11);
		\draw [very thick] 
		(15,-22) -- (15,-3);
		\draw [very thick] 
		(7,-16) -- (3,-5);
		\draw [very thick] (3,-5) -- (9,-13.5);
		\draw [very thick] (9.6,-14.35) -- (15,-22);
	\end{tikzpicture}
	\hspace{.3cm}
	\begin{tikzpicture}[scale=.13]
\node [circle,fill,scale=.8,draw] (1) at (9,17) {};
		\node [circle,fill,scale=.8,draw] (3) at (9,3) {};
		\node [thick,circle,scale=.8,draw] (5) at (3,10) {};	
		\node [thick,circle,scale=.8,draw] (6) at (9,10) {};	
		\node [thick,circle,scale=.8,draw] (7) at (15,10) {};	
		\path [thick,->,>=stealth] (5) edge node {} (1);
		\path [thick,<-,>=stealth] (5) edge node {} (3);
		\path [thick,->,>=stealth] (6) edge node {} (1);
		\path [ultra thick,<-,>=stealth] (7) edge node {} (3);
		\draw [fill] (3,-5) circle [radius=.5];
		\draw [lightgray,fill] (15,-3) circle [radius=.5];
		\draw [fill] (11,-11) circle [radius=.5];
		\draw [fill] (3,-24) circle [radius=.5];
		\draw [fill] (15,-22) circle [radius=.5];
		\draw [lightgray,fill] (7,-16) circle [radius=.5];
		\draw [lightgray] (3,-5) -- (15,-3);
		\draw [very thick] (3,-5) -- (11,-11);
		\draw [lightgray] (11,-11) -- (15,-3);
		\draw [lightgray] (3,-24) -- (7,-16);
		\draw [very thick] (3,-24) -- (15,-22);
		\draw [lightgray] (8.2,-16.9) -- (15,-22);
		\draw [very thick] 
		(3,-24) -- (11,-11);
		\draw [lightgray]
		(7,-16) -- (8.36,-13.79);
		\draw [lightgray] (9.24,-12.36) -- (9.88,-11.32);
		\draw [lightgray]
		(10.76,-9.89) -- (15,-3);
		\draw [very thick] 
		(3,-24) -- (3,-5);
		\draw [very thick] 
		(15,-22) -- (11,-11);
		\draw [lightgray]
		(15,-22) -- (15,-3);
		\draw [lightgray]
		(7,-16) -- (3,-5);
		\draw [very thick] (3,-5) -- (9,-13.5);
		\draw [very thick] (9.6,-14.35) -- (15,-22);
	\end{tikzpicture}
	\hspace{.3cm}
	\begin{tikzpicture}[scale=.13]
\node [circle,fill,scale=.8,draw] (1) at (9,17) {};
		\node [circle,fill,scale=.8,draw] (3) at (9,3) {};
		\node [thick,circle,scale=.8,draw] (5) at (3,10) {};	
		\node [thick,circle,scale=.8,draw] (6) at (9,10) {};	
		\node [thick,circle,scale=.8,draw] (7) at (15,10) {};	
		\path [thick,->,>=stealth] (5) edge node {} (1);
		\path [thick,<-,>=stealth] (5) edge node {} (3);
		\path [ultra thick,<-,>=stealth] (6) edge node {} (3);
		\path [ultra thick,<-,>=stealth] (7) edge node {} (3);
		\draw [fill] (3,-5) circle [radius=.5];
		\draw [lightgray,fill] (15,-3) circle [radius=.5];
		\draw [lightgray,fill] (11,-11) circle [radius=.5];
		\draw [fill] (3,-24) circle [radius=.5];
		\draw [fill] (15,-22) circle [radius=.5];
		\draw [fill] (7,-16) circle [radius=.5];
		\draw [lightgray] (3,-5) -- (15,-3);
		\draw [lightgray] (3,-5) -- (11,-11);
		\draw [lightgray] (11,-11) -- (15,-3);
		\draw [very thick] (3,-24) -- (7,-16);
		\draw [very thick] (3,-24) -- (15,-22);
		\draw [very thick] (8.2,-16.9) -- (15,-22);
		\draw [lightgray]
		(3,-24) -- (11,-11);
		\draw [lightgray]
		(7,-16) -- (8.36,-13.79);
		\draw [lightgray] (9.24,-12.36) -- (9.88,-11.32);
		\draw [lightgray]
		(10.76,-9.89) -- (15,-3);
		\draw [very thick] 
		(3,-24) -- (3,-5);
		\draw [lightgray]
		(15,-22) -- (11,-11);
		\draw [lightgray]
		(15,-22) -- (15,-3);
		\draw [very thick] 
		(7,-16) -- (3,-5);
		\draw [very thick] (3,-5) -- (9,-13.5);
		\draw [very thick] (9.6,-14.35) -- (15,-22);
	\end{tikzpicture}}
    \caption{Dissection induced by a ribbon structure.}
    \label{fig:dissect}
\end{figure}
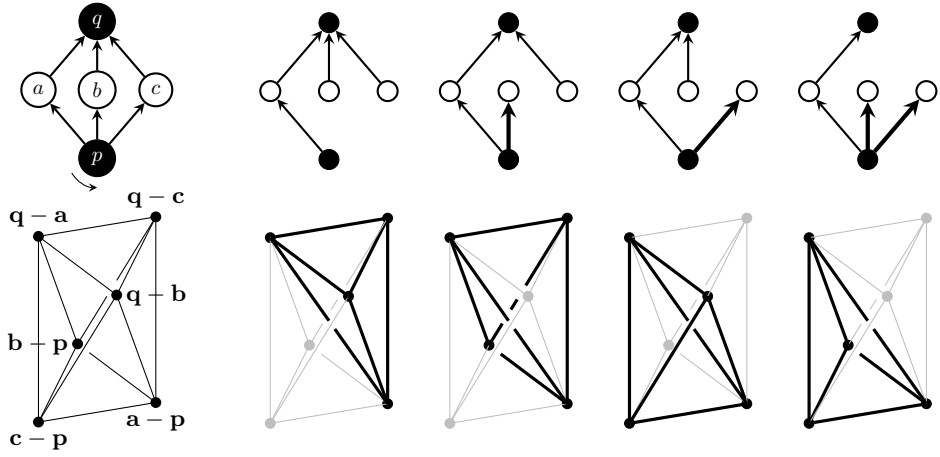

In order to prove Theorem \ref{t:J-trees_form_quasitr}, we define a linear order on spanning trees of a directed ribbon graph (that depends on the basis). The restriction of this order to Jaeger trees will turn out to be a shelling order of the dissection (see Theorem \ref{thm:shelling}).

The idea behind the order has already appeared in the proof of Proposition \ref{p:trees_interior_disjoint}. 
For any two spanning trees $T$ and $T'$ of a ribbon graph, 
their respective tours have maximal initial segments that coincide. The last element of this common segment is a pair $(x,\varepsilon)$, where the edge $\varepsilon$ is part of one tree but not part of the other. We will refer to $\varepsilon$ as the \emph{first difference} between the trees.

\begin{defn}\label{def:treeorder}
	For two spanning trees $T$ and $T'$ of a directed ribbon graph, equipped with a basis, we write $T' \prec T$ if their first difference $\varepsilon$
	satisfies one of the following:
	\begin{itemize}
		\item 
		$\varepsilon$ is skipped in the tour of $T'$ in accordance with the Jaeger rule (from the tail direction), but $\varepsilon \in T$,
		\item 
		$\varepsilon$ is skipped in the tour of $T$ violating the rule (from the head direction), but $\varepsilon \in T'$.
	\end{itemize}
\end{defn}


\begin{lemma}
The relation $\prec$ is a complete ordering of the spanning trees. 
\end{lemma}

\begin{proof} 
As any two trees are comparable, we only need to prove transitivity. Suppose that $T_1\prec T_2$ and $T_2\prec T_3$. Assume that the first difference between $T_1$ and $T_2$ is the edge $\varepsilon$ and the first difference between $T_2$ and $T_3$ is the edge $\varepsilon'$. 
It is easy to see that $\varepsilon=\varepsilon'$ is impossible. 
If $\varepsilon$ is reached before $\varepsilon'$ in the tour of $T_2$, then the first difference between $T_1$ and $T_3$ is $\varepsilon$, and this orders them as $T_1$ and $T_2$ are ordered, i.e., $T_1\prec T_3$. If $\varepsilon'$ is reached before $\varepsilon$ in the tour of $T_2$, then the first difference between $T_1$ and $T_3$ is $\varepsilon'$, and this implies that they are ordered as $T_2$ and $T_3$ are, that is, again, $T_1\prec T_3$.
\end{proof}

\begin{proof}[Proof of Theorem \ref{t:J-trees_form_quasitr}]
Corollary \ref{cor:max_simplices:of_root_polytope} implies that Jaeger trees correspond to maximal 
simplices in the root polytope.
Proposition \ref{p:trees_interior_disjoint} tells us that simplices corresponding to Jaeger trees are interior disjoint. Hence we only need to prove that any point in the root polytope is inside the simplex of some Jaeger tree.

Let $G$ be a connected semi-balanced graph and let us take a point $\mathbf{p}\in \mathcal Q_G=\conv\{\mathbf{h} - \mathbf{t}\mid \overrightarrow{th}\in 
E(G)\}$. As the dimension of $\mathcal Q_G$ is $|V(G)|-2$, by Caratheodory's theorem, $\mathbf{p}$ is in the convex hull of at most $|V(G)|-1$ affine independent vertices, that is, there is a spanning tree $T$ of $G$ such that $\mathbf{p}=\sum_{\overrightarrow{th}\in T} \lambda_{th} (\mathbf{h} - \mathbf{t})$ with some $0\leq \lambda_{th}\leq 1$ for each edge $th\in T$, so that $\sum_{th\in T}\lambda_{th}=1$. Let us also put $\lambda_\eta=0$ for edges $\eta\notin T$. 

If $T$ is a Jaeger tree, we are done. Suppose that $T$ is not a Jaeger tree.
We show that in this case we can find another spanning tree $T' \prec T$ with $\mathbf{p}\in\mathcal Q_{T'}$.
As $T$ is not a Jaeger tree, in the tour of $T$ there is a non-tree edge that is first encountered at its head. 
Suppose that 
$\varepsilon=\overrightarrow{xy}$ is one such edge. 
The subgraph $T\cup\{\varepsilon\}$ has exactly one cycle $C$ (the so called fundamental cycle $C=C(T,\varepsilon)$), which is semi-balanced since $G$ is semi-balanced. We may assume that $\varepsilon\in C^+$. 

Let $\mu=\min_{\eta\in C^-} \lambda_{\eta}$ and let $\varepsilon'\in C^-$ be one of the edges where the minimum $\mu$ is realized.
If we define
$$\lambda'_\eta = \left\{\begin{array}{cl} 
\lambda_\eta & \text{if $\eta\notin C$}  \\
\lambda_\eta - \mu  & \text{if $\eta\in C^-$}\\
\lambda_\eta + \mu  & \text{if $\eta\in C^+$}
\end{array} \right.,
$$
then by Lemma \ref{l:dependence_of_semibalanced_cycle}, we have $\mathbf{p}=\sum_{\eta\in E}\lambda_\eta \mathbf x_\eta = \sum_{\eta\in E}\lambda'_\eta \mathbf x_\eta$.

By the choice of $\mu$ we have $\lambda'_{\eta}\geq 0$ for each edge of $G$, and $\lambda'_{\varepsilon'} = 0$. Hence for $T'=T\cup\{\varepsilon\}\setminus\{\varepsilon'\}$, we have $\{\eta\mid\lambda_\eta > 0\}\subset T'$, meaning $\mathbf p\in\mathcal Q_{T'}$. Also, $T'$ is once again a spanning tree since $\varepsilon'\in C(T,\varepsilon)$. 

We now prove that $T' \prec T$. 
If $(y,\varepsilon)$ comes before $\varepsilon'$ (paired with either endpoint) in the tour of $T$, then $\varepsilon$ is the first difference between $T$ and $T'$. As we have $\varepsilon\in T'$,
while on the other hand $\varepsilon$ is skipped in the tour of $T$ violating the Jaeger rule, we do have $T' \prec T$.

Else if $\varepsilon'$, with its endpoint $r$ as current node, appears before $\varepsilon$ in the tour of $T$, then $\varepsilon'$ becomes the first difference between $T$ and $T'$. After traversing $\varepsilon'$, the tour of $T$ proceeds along (with possible detours that do not involve vertices along $C$) the unique path in $T$ that connects the other endpoint $q$ of $\varepsilon'$ with the node $y$. As that path is part of $C$, from the assumption that $\varepsilon'\in C^-$, we see that $q$ must be the head and $r$ must be the tail of $\varepsilon'$.
This means that $\varepsilon'$ is removed from $T'$ in accordance with the Jaeger rule. Hence we have $T'\prec T$ in this case as well.

Now from the finiteness of the set of spanning trees, it is clear that 
there exists a Jaeger tree $T$ such that $\mathbf p$ is in the simplex corresponding to $T$.
\end{proof}

The idea of the above proof enables us to give a much shorter justification of the `correctness' of the Bernardi algorithm from \cite{hyperBernardi}. We do this in Subsection \ref{ss:Bernardi_alg_well_def}.

Let us also 
mention the following corollary
of the previous argument.

\begin{lemma}\label{l:smallest_tree_is_Jaeger}
	For each point $\mathbf p\in\mathcal Q_G$, among the spanning trees $\{T\mid \mathbf p\in\mathcal Q_T\}$, the minimal one according to $\prec$ is a Jaeger tree.
\end{lemma}

The order $\prec$ has interesting symmetries. For example, one could re-define Jaeger trees ``symmetrically,'' demanding that in the tour of $T$ each non-edge be first seen at its \emph{head}. Then, for each point $\mathbf p\in\mathcal Q_G$, the \emph{maximal} tree in $\{T\mid \mathbf p\in\mathcal Q_T\}$, according to $\prec$, would be one of these ``head-cut Jaeger trees.'' 

Another type of symmetry for Jaeger trees is the following extension of \cite[Lemma 6.5]{hyperBernardi}.

\begin{prop}
\label{p:reverse_structure} Let $G$ be a connected directed graph with a ribbon structure and basis $(b,\varepsilon)$. Let $\delta$ be the edge that precedes $\varepsilon$ in the cyclic order around $b$. Now let us also consider the graph $-G$, that is $G$ with all edge orientations reversed, and equip it with the ribbon structure that is the reverse of the previous, i.e., the cyclic orders around all edges are inverted. Then, the Jaeger trees of $-G$ with respect to the basis $(b,\delta)$ are exactly the same as the Jaeger trees of $G$ with respect to $(b,\varepsilon)$.
\end{prop}

\begin{proof}
When we compare the tours of a spanning tree $T$ in $G$ and in $-G$, we see that for any non-edge $xy$ of $T$, the orders in which $(x,xy)$ and $(y,xy)$ appear are opposite. But since the orientation of $xy$ also gets reversed, the two Jaeger conditions are equivalent.
\end{proof}

Let $\prec'$ be the order induced on spanning trees of $-G$, with respect to the reversed ribbon structure as in Proposition \ref{p:reverse_structure}.
Then it seems that (apart from having the same set of Jaeger trees) $\prec'$ has little connection to $\prec$.

\begin{question}
    Does the ordering $\prec$ have other interesting symmetries? Do these have meaningful corollaries? What is the relationship of $\prec$ and $\prec'$?
\end{question}

Recall that by Lemma \ref{lem:unimodularitas}, the simplex $\mathcal Q_T$ has the same volume for any spanning tree $T$. Therefore Theorem \ref{t:J-trees_form_quasitr} implies the following.

\begin{cor}
\label{c:invariance}
	The number of Jaeger trees of a semi-balanced directed graph is independent of the ribbon structure and the basis. 
\end{cor}

A somewhat imprecise phrasing of the reason is that the number of Jaeger trees of $G$ is proportional to the volume of $\mathcal Q_G$. 
We will soon also see that the number of Jaeger trees of $G$ is $I_G(1)$, i.e., the sum of the coefficients of the interior polynomial.

We note that 
Corollary \ref{c:invariance} fails without semi-balancedness. For example, take a properly directed cycle. This has a unique ribbon structure. If the base node is the head of the base edge, we do not have Jaeger trees. If the base node is the tail of the base edge, all spanning trees are Jaeger.

\begin{question}
	Is there a good analogue of the dissection by Jaeger trees for un-semibalanced directed graphs? In other words, is there a similar way to dissect the root polytope of an un-semibalanced directed graph into simplices? What about the case of (not necessarily bipartite) signed graphs?
\end{question}

\section{Shelling}
\label{sec:shelling}

In this section we show that the dissection, induced by Jaeger trees, of the root polytope of a semi-balanced (thus bipartite and directed) graph is shellable.
(See Section \ref{sec:int_poly} for the relevant definitions.)
Specifically, the restriction of $\prec$ (Definition \ref{def:treeorder}) to Jaeger trees yields a shelling order.

This implies that an $h$-vector, cf.\ \eqref{eq:h-hejazott}, may be considered.
To be able to describe it combinatorially, we need the following definition.

\begin{defn}[Internal semi-passivity with respect to a fixed order]
	Let $G$ be a directed graph and let us assume that an order is given on its edge set. With respect to a spanning tree $T$ of $G$, we call an edge $\varepsilon\in T$ \emph{internally semi-passive} if it stands opposite to the largest edge $\varepsilon'$ in the fundamental cut $C^*(T,\varepsilon)$ (that is, $\varepsilon$ and $\varepsilon'$ have their tails in different components of $T-\varepsilon$).
	
	We say that an edge is \emph{internally semi-active} in $T$ if it is not internally semi-passive.
\end{defn}

If we are given a ribbon structure and a basis for our directed graph, we can associate the $T$-order $\leq_T$ of Definition \ref{def:T-order} to any tree $T$. 
Unless specified otherwise, we will always consider internal semi-passivity  of an edge $\varepsilon$ in a tree $T$ with respect to this order.
That is, the order is not fixed once and for all; instead, the ribbon structure and the basis are. This point of view is adopted from Bernardi's work \cite{Bernardi_first,Bernardi_Tutte} on the Tutte polynomial.

We note that in the prequel \cite{hyperBernardi}, internal semi-passivity was defined using the \emph{smallest} instead of the \emph{largest} edge, hence the corresponding theorems in that paper are slightly different.

The main theorem of this section is as follows.

\begin{thm}\label{thm:shelling}
Let $G$ be a semi-balanced 
graph, with an arbitrarily fixed ribbon structure and basis.
	The restriction of the order $\prec$ of Definition \ref{def:treeorder} to the Jaeger trees of $G$ induces a shelling order of the dissection of Theorem \ref{t:J-trees_form_quasitr}. For each Jaeger tree $T$, the value of the $h$-vector statistic (i.e., the number of facets of the corresponding simplex $\mathcal Q_{T}$ that lie in the union of the earlier simplices of the shelling) equals the number of internally semi-passive edges in $T$ with respect to the $T$-order.
\end{thm}

From here on we will often abbreviate `internally semi-passive edges of $T$ with respect to the $T$-order' to just `semi-passive edges.'

\begin{example}
Figure \ref{fig:dissect} lists the Jaeger trees 
in the order 
$\prec$.
The 
semi-passive edges of each tree are thickened.
(The status of an edge is easiest to determine by using condition \ref{elso} of the Lemma below.) Regarding the facets of each associated simplex, the facets along which the 
simplex attaches to the earlier maximal simplices are exactly the ones whose opposite vertices correspond to 
semi-passive edges. 
\end{example}

We start the proof of Theorem \ref{thm:shelling} with the following  characterization of semi-passive edges.
We note that even for general directed graphs (that is, not necessarily semi-balanced or even bipartite), its parts \ref{masodik} and \ref{otodik} are equivalent, and follow from \ref{elso}. 

\begin{lemma}
	\label{l:activities_description}
	Let $G$ be a semi-balanced 
	graph with a fixed ribbon structure and basis.
	For a Jaeger tree $T$ of $G$ and edge $\varepsilon \in T$, the following statements are equivalent.
	\begin{enumerate}[label=(\roman*)]
		\item \label{elso}
		$\varepsilon$ arises as the first difference between $T$ and some Jaeger tree preceding $T$ in $\prec$, i.e., there exists a  Jaeger tree $T'$ such that $\varepsilon\not\in T'$ but the tours of $T$ and $T'$ coincide until reaching $\varepsilon$.
		\item \label{masodik}
		$\varepsilon$ is internally semi-passive in the spanning tree T with respect to the $T$-order of the edges of $G$. 
		\item \label{otodik} 
		$\varepsilon$ has its tail in the base component of $T-\varepsilon$ and there exists an edge in the fundamental cut $C^*(T,\varepsilon)$ that stands opposite to $\varepsilon$.
	\end{enumerate}
\end{lemma}

\begin{proof}
	\ref{elso} $\Rightarrow$ \ref{otodik}.
	As the tours of $T$ and $T'$ coincide until reaching $\varepsilon$, they both reach $\varepsilon$ at the same endpoint. As $\varepsilon\notin T'$, this endpoint must be the tail of $\varepsilon$, which is thus in the base component.
	Those edges of $C^*(T,\varepsilon)-\varepsilon$ that had already been skipped in the tour before reaching $\varepsilon$, are not in $T'$. By Lemma \ref{l:char_Jaeger_cuts}, the rest of the edges of $C^*(T,\varepsilon)-\varepsilon$ 
	have their heads in the base component. 
	This latter set is non-empty because $T'$ must contain at least one of its elements.
	
	\smallskip
	\noindent 
	\ref{otodik} $\Rightarrow$ \ref{elso}. Let $T_0$ and $T_1$ be the two subtrees of $T-\varepsilon$, with $T_0$ containing the base node. Let $G_0$ be the subgraph of $G$ spanned by the node set of $T_0$, and $G_1$ be the subgraph spanned by the node set of $T_1$. We build up a Jaeger tree $T'$ together with its tour, such that $\varepsilon$ is the first difference between $T'$ and $T$. To do so, we follow the tour of $T$ until reaching $\varepsilon$, but at that moment, we do not include $\varepsilon$ into $T'$. We are allowed to do this by the assumption that $\varepsilon$ has its tail in $T_0$. Instead, we stay in $T_0$ and continue with the part of the tour of $T$ that would follow the second traversal of $\varepsilon$. We stop at the first moment when an edge $\varepsilon' \in C^*(T,\varepsilon)$, standing opposite to $\varepsilon$, becomes current edge (we will include this edge into $T'$). 
	
	By the assumption \ref{otodik} there exists such an edge and \cite[Lemma 5]{Bernardi_first} guarantees that our process will find it. Let $t'$ be the endpoint of $\varepsilon'$ from $G_1$, which is its tail. If $G_1$ is not just a point, then take the first edge $t'h''\in G_1$ following $\varepsilon'$ in the cyclic order around $t'$. Take an arbitrary Jaeger tree $T'_1$ of $G_1$ with base point $t'$ and base edge $t'h''$. (Here the ribbon structure for $G_1$ is the restriction of the one for $G$.) Such a Jaeger tree exists since $G_1$ is connected and semi-balanced, whence Jaeger trees dissect the root polytope of $G_1$, which is not empty if $G_1$ is not a single point. If $G_1$ is a single point then we take $T_1'$ to also be that point.
	
	We claim that $T'=T_0\cup\{\varepsilon'\}\cup T'_1$ is a Jaeger tree. Once we prove this, it becomes immediate from the construction that the first difference between $T$ and $T'$ is $\varepsilon$.
	
	Until reaching $\varepsilon$, the tours of $T$ and $T'$ coincide. Then $\varepsilon$ is cut (at the allowed endpoint) in the tour of $T'$. Next we continue the traversal of $T_0$ until we arrive at $\varepsilon'$. During this time, any edges that we cut are edges of $G_0$ that are cut, at the same endpoint, in the tour of $T$ as well. As $\varepsilon'$ is traversed, the tour of $T'$ thus far has not cut any edge at its prohibited endpoint. If there are any edges (necessarily of $C^*(T,\varepsilon)$) between $\varepsilon'$ and $t'h''$ at $t'$, for which $t'$ is their head, then
by Lemma \ref{l:char_Jaeger_cuts}, they have already been cut at their tails.
	 Next we start the traversal of $T'_1$, which is a Jaeger tree in $G_1$. Compared to the tour of $T'_1$ with regard to $G_1$, the only difference in the tour of $T'$ is that any edges in $C^*(T,\varepsilon)$ that we encounter have to be skipped. But again by Lemma \ref{l:char_Jaeger_cuts}, all the edges from $C^*(T,\varepsilon)$ that have their head in $G_1$ have already been cut, hence none of them gets cut at its prohibited endpoint. Note that if $G_1$ is a single point, then we only have the edges of $C^*(T,\varepsilon)$ (which, in that case, is the star-cut of $t'$) to deal with, but the same reasoning applies.
	 
	 Then when we arrive back at $h'$ after traversing $\varepsilon'$ for the second time, the set of node-edge pairs that have not been current is the same as in the tour of $T$ after $\varepsilon'$ is current for the second time. Moreover, the orders in which these remaining pairs become current are the same, too. Hence during the remainder of the tour of $T'$, each edge that is cut is still cut at the allowed endpoint.
	
	\smallskip
	\noindent 
	\ref{otodik} $\Rightarrow$ \ref{masodik}. 
	By Lemma \ref{l:char_Jaeger_cuts}, those edges in $C^*(T,\varepsilon)$ that have their tails in the base component all precede, in the $T$-order, those which have their heads there. Since now there exists an edge in $C^*(T,\varepsilon)$ that has its head in the base component, necessarily the largest edge $\delta$ of $C^*(T,\varepsilon)$, according to the $T$-order, is of this kind. As $\varepsilon$ has its tail in the base component, it stands opposite to $\delta$, which means that $\varepsilon$ is internally semi-passive.
	
	\smallskip
	\noindent 
	\ref{masodik} $\Rightarrow$ \ref{otodik}.  Internal semipassivity for $\varepsilon$ implies that there must be an edge $\delta$ in $C^*(T,\varepsilon)$ standing opposite to $\varepsilon$, moreover, $\delta$ comes later than $\varepsilon$ in the $T$-order. By Lemma \ref{l:char_Jaeger_cuts}, this implies that $\varepsilon$ must have its tail in the base component. 
\end{proof}

\begin{proof}[Proof of Theorem \ref{thm:shelling}] 
	Let us fix a Jaeger tree $T$ of our semi-balanced ribbon graph $G=(V,E)$. Let $\mathcal R_T = \bigcup_{T' \prec T}\mathcal Q_{T'}$ be the union of the simplices corresponding to Jaeger trees $T'$ preceding $T$ in $\prec$. First we will show that for an internally semi-passive edge $\varepsilon\in T$, 
	the facet $\mathcal Q_{T-\varepsilon}$ of the simplex $\mathcal Q_T$ lies in  $\mathcal R_T$. Then 
	we will show that if $T$ is not the first Jaeger tree then $\mathcal Q_T$ is not disjoint from $\mathcal R_T$, 
	moreover, if $\mathbf{p}\in \mathcal Q_T\cap\mathcal R_T$, then $\mathbf p$ is in $\mathcal Q_{T-\varepsilon}$ for an internally semi-passive edge $\varepsilon$.
	
	Take an edge $\varepsilon
	\in T$ 
	such that $\varepsilon$ is internally semi-passive in $T$ with respect to the $T$-order. 
	We will use the notation of the part \ref{otodik} $\Rightarrow$ \ref{elso} of the proof of Lemma \ref{l:activities_description}. Recall that a Jaeger tree was constructed there in the form 
	\begin{equation}
	\label{eq:sokfa}
T'=T_0\cup\{\varepsilon'\}\cup T_1',
\end{equation}
	where $T_0$ is the base component of $T-\varepsilon$, the edge $\varepsilon'$ is a particular element of $C^*(T,\varepsilon)$, and $T_1'$ is a Jaeger tree in the subgraph $G_1$ determined by the fundamental cut. We will show that as $T_1'$ ranges over the set $\mathcal T_1$ of Jaeger trees for $G_1$ (with basis specified as in the proof of Lemma \ref{l:activities_description}), the union of the simplices $\mathcal Q_{T'}$ is large enough to contain $\mathcal Q_{T-\varepsilon}$. This will complete the first part of the proof since $T'\prec T$ for all choices of $T_1'$.
	
We will denote the set of Jaeger trees of the form \eqref{eq:sokfa}, for $T_1'\in\mathcal T_1$, by $\mathcal T$.
	
	If $G_1$ is a single node, then the only element of $\mathcal T_1$ is the empty set, furthermore $T-\varepsilon=T_0$, which makes our claim trivial. Otherwise take an arbitrary point $\mathbf p$ from $\mathcal Q_{T - \varepsilon}$. This means that
	\[\mathbf p=\sum_{\eta\in T - \varepsilon}\lambda_{\eta} \mathbf x_\eta=\sum_{\eta\in T_0}\lambda_{\eta} \mathbf x_\eta + \sum_{\eta\in T_1}\lambda_{\eta} \mathbf x_\eta,\]
	where $\lambda_{\eta}\geq 0$ for all $\eta$ and $\sum_{\eta\in T - \varepsilon}\lambda_{\eta}=1$. Let $\lambda_0=\sum_{\eta\in T_0}\lambda_{\eta}$ and $\lambda_1=\sum_{\eta\in T_1}\lambda_{\eta}$. If $\lambda_0=0$ then $\mathbf p\in\mathcal Q_{G_1}$, which by Theorem \ref{t:J-trees_form_quasitr} is dissected by the faces $\mathcal Q_{T_1'}$ of the simplices $\mathcal Q_{T'}$, for $T'\in\mathcal T$ as in \eqref{eq:sokfa}, making $\mathbf p\in \bigcup_{T'\in \mathcal{T}}\mathcal Q_{T'}$ obvious. If $\lambda_1=0$, then $\mathbf p\in\mathcal Q_{T_0}$, which in turn is part of all the $\mathcal Q_{T'}$. Otherwise, write
	\[\mathbf p=\lambda_0\cdot\sum_{\eta\in T_0}\frac{\lambda_{\eta}}{\lambda_0}\,\mathbf x_\eta + \lambda_1\cdot \sum_{\eta\in T_1}\frac{\lambda_{\eta}}{\lambda_1}\,\mathbf x_\eta\]
	and let $\mathbf p_0=\sum_{\eta\in T_0}\frac{\lambda_{\eta}}{\lambda_0}\,\mathbf x_\eta$ and $\mathbf p_1=\sum_{\eta\in T_1}\frac{\lambda_{\eta}}{\lambda_1}\,\mathbf x_\eta$. That is, $\mathbf p=\lambda_0\mathbf p_0 + \lambda_1\mathbf p_1$, where $\lambda_0,\lambda_1>0$ and $\lambda_0 + \lambda_1=1$. Furthermore, we have $\mathbf p_1\in\mathcal Q_{T_1}\subset\mathcal Q_{G_1}$ and $\mathbf p_0\in\mathcal Q_{T_0}$. Combining this with the convexity of $\mathcal Q_{T_0}$ and $\mathcal Q_{G_1}=\bigcup_{T'_1\in \mathcal{T}_1}\mathcal Q_{T'_1}$, we obtain
	\begin{multline*}
	\bigcup_{T'\in \mathcal{T}}\mathcal Q_{T'}\supset \bigcup_{T'\in \mathcal{T}}\mathcal Q_{T'-\varepsilon'}\\
	=\bigcup_{T'_1\in \mathcal{T}_1} \conv(\mathcal Q_{T_0}\cup\mathcal Q_{T'_1})
	=\conv\left(\mathcal Q_{T_0}\cup\bigcup_{T'_1\in \mathcal{T}_1}\mathcal Q_{T'_1}\right)
	=\conv(\mathcal Q_{T_0} \cup\mathcal Q_{G_1})\ni \mathbf p.
	\end{multline*}

	As to the second part of our proof, to show that $\mathcal Q_T$ is not disjoint from 
	$\mathcal R_T$ if the latter is nonempty, note that by Lemma \ref{l:activities_description}, for any Jaeger tree $T'\prec T$, the first difference between $T'$ and $T$ is an edge $\varepsilon\in T$ that is internally semi-passive in $T$. Hence $\mathcal Q_{T-\varepsilon}\subset\mathcal R_T$,
	which implies that the two sets are not disjoint.
	
	Finally, take a point $\mathbf{p}\in\mathcal Q_T\cap\mathcal R_T$.
	We may suppose that $\mathbf{p}\in\mathcal Q_T\cap\mathcal Q_{T'}$ for some Jaeger tree $T'\prec T$, where the first difference between $T$ and $T'$ is 
	$\varepsilon\in T\setminus T'$. Now the functional $\kappa_{T,\varepsilon}$, constructed in the proof of Proposition \ref{p:trees_interior_disjoint}, shows that $\mathcal Q_T\cap\mathcal Q_{T'}\subset\mathcal Q_{T-\varepsilon}$. 
	Here by Lemma \ref{l:activities_description}, the edge $\varepsilon$ is internally semi-passive in $T$. Thus indeed, our arbitrary intersection point $\mathbf p$ lies in a facet of $\mathcal Q_T$ whose opposite vertex corresponds to a semi-passive edge, as claimed.
\end{proof}

Let us summarize again our strategy to compute the $h^*$-vector of $\mathcal Q_G$, i.e., the interior polynomial of the semi-balanced graph $G$. We fix an arbitrary ribbon structure and basis. We generate all associated Jaeger trees $T_i$ and the corresponding numbers $r_i$ of semi-passive edges. (If we find the $T_i$ by a certain natural in-depth search, then condition \ref{elso} of Lemma \ref{l:activities_description} makes it easy to keep track of the $r_i$.)  We let $h_j$ denote the number of indices $i$ so that $r_i=j$. Then, according to \eqref{eq:generator}, we have
\begin{equation}
\label{eq:intash}
I_G(x)=\sum_jh_j\,x^j.
\end{equation}

\begin{example}
Figure \ref{fig:K23} shows three semi-balanced graphs.
Example \ref{ex:dissect} dealt with the one on the right, and we found that
its Jaeger trees had $0$, $1$, $1$, and $2$ internally semi-passive edges, respectively. Therefore $h_0=1$, $h_1=2$, $h_2=1$, and the interior polynomial of the graph in question is $1+2x+x^2$.

The other two 
graphs of Figure \ref{fig:K23} share the same interior polynomial $1+2x$.

For the graphs of Figure \ref{fig:K34}, similar but lengthier computations yield the interior polynomials (from left to right) $1+6x+3x^2$, $1+6x+9x^2+4x^3$, $1+6x+6x^2$, and $1+6x+7x^2$, respectively.
\end{example}

\section{Algorithmic aspects}
\label{sec:computing_the_Jaeger_tree}
\label{sec:alg}

Suppose that we are given a semi-balanced 
graph with a ribbon structure and a basis. In Theorem \ref{t:J-trees_form_quasitr} we saw that the root polytope is dissected into simplices indexed by Jaeger trees. 
In this section we look at algorithmic aspects of this fact,
namely, we give a simple, greedy algorithm for finding a Jaeger tree containing a given point of the root polytope in its simplex. The algorithm relies on, as a subroutine, the decision whether a given point lies in the convex hull of some other points. If the coordinates of the points are rational, then this can be done in polynomial time via linear programming. Altogether, our algorithm runs in polynomial time if the input point has rational coordinates.

In Subsection \ref{ss:Bernardi_alg_well_def}, we point out that the correctness of the Bernardi algorithm of \cite{hyperBernardi} (that finds the Jaeger tree realizing a given hypertree) follows from the correctness of our algorithm.
Thereby we give a simpler proof for the correctness of the Bernardi algorithm.

\subsection{Finding a Jaeger tree whose simplex contains a given point}

We propose the following process. 

\smallskip

\noindent\textbf{Algorithm (Jaeger-tree-for-point):} \\
INPUT: a connected semi-balanced 
graph $G$, 
a ribbon structure for $G$, a basis $(b_0,b_0b_1)$, and a point $\mathbf{p}\in\mathcal Q_G$. \\
OUTPUT: a Jaeger tree (with respect to the given ribbon structure and basis) $T\subset G$ such that $\mathbf{p}\in\mathcal Q_T$.

We maintain a current graph $G_{curr}$, a current node $x$ of $G_{curr}$ and a current edge $xy$ of $G_{curr}$ incident to our current node. At the beginning, $G_{curr}=G$, the current node is $b_0$ and the current edge is $b_0b_1$. 

As a step of the algorithm, if the current node $x$ is the tail of the current edge $xy$, then
we check whether $G_{curr}-xy$ is connected and $\mathbf{p}$ is in the root polytope of $G_{curr}-xy$. If both answers are yes, then we take $G_{curr}-xy$ as the next current graph, $x$ as the next current vertex and $xy^+_{G_{curr}}$ as the next current edge.
If $\mathbf{p}\notin\mathcal Q_{G_{curr}-xy}$ or $G_{curr}-xy$ is not connected, then the current graph stays $G_{curr}$, the current node becomes $y$ and the current edge becomes $yx^+_{G_{curr}}$.

If the current node $x$ is the head of the current edge $xy$, then
$G_{curr}$ stays the current graph, the next current node is $y$ and the next current edge is $yx^+_{G_{curr}}$.
We stop when an edge becomes current with its tail as current node for the second time. The output is the current graph at the end of the algorithm.

\begin{thm}
	For any input as above, the algorithm produces a Jaeger tree $T$ of $G$ such that $\mathbf{p}\in\mathcal Q_T$. Moreover, it returns the smallest tree, with respect to the order $\prec$ of Definition \ref{def:treeorder}, among those spanning trees $F$ of $G$ that satisfy $\mathbf{p}\in\mathcal Q_{F}$. 
\end{thm}

\begin{proof}
	Let $T$ be the smallest tree with respect to $\prec$ in 
	the set of spanning trees $F$ of $G$ so that $\mathbf{p}\in\mathcal Q_{F}$.
	Then $T$ is a Jaeger tree by Lemma \ref{l:smallest_tree_is_Jaeger}.
	Take the incident node-edge pairs of $G$ in the order as they become current in the tour of $T$, and remove the pairs $(h,\varepsilon)$ where $\varepsilon\notin T$ and $h$ is the head of $\varepsilon$. Let $S$ be the sequence of the remaining node-edge pairs.
	We will show that the Jaeger-tree-for-point algorithm removes precisely the edges in $E(G)-T$, hence it returns $T$, and 
	its sequence of current node-edge pairs is exactly $S$.
	
	Suppose for a contradiction that this is not true, and take the last step before the sequence of current node-edge pairs of the algorithm differs from $S$. Let the last common node-edge pair be $(v,\varepsilon)$. 
	There are two ways for the two processes to part: either the algorithm does not remove $\varepsilon$ and $\varepsilon\notin T$, or it removes $\varepsilon$ and $\varepsilon\in T$.
	
	Regarding the first case, 
	as $T$ is a Jaeger tree and $\varepsilon\notin T$, in the tour of $T$, the edge $\varepsilon$ is first reached from the tail direction. Since so far all the edges of $T$ were kept, before this step, we have $T\subseteq G_{curr}$. But then $T\subseteq G_{curr}-\varepsilon$, in particular $G_{curr}-\varepsilon$ is connected and $\mathbf{p}\in \mathcal Q_T\subseteq \mathcal Q_{G_{curr}-\varepsilon}$, whence the algorithm would need to remove $\varepsilon$, which is a contradiction.
	
	Now suppose that the second case happens. 
	Let $G'$ be the current graph of the algorithm after removing $\varepsilon$. By our assumption, $S$ agrees with the sequence of node-edge pairs of the algorithm up to $(v,\varepsilon)$, 
	hence in the tour of $T$, until reaching $(v, \varepsilon)$, all the edges of $T$ are in $G'$, and all the non-edges of $T$ are not in $G'$. Let $u$ be the other endpoint of $\varepsilon$. (Note that as the algorithm was able to remove $\varepsilon$ while $(v,\varepsilon)$ was current, we know that $v$ is the tail and $u$ is the head of $\varepsilon$.)
	We claim that $(u,\varepsilon)$ cannot precede $(v,\varepsilon)$ in $S$.
	Otherwise, $\varepsilon$ (from $u$ to $v$), and the component of $T-\varepsilon$ containing $v$ would be traversed by the algorithm and the tour of $T$ alike, removing all the edges of $C^*(T,\varepsilon)-\varepsilon$ in the process.
	Hence removing $\varepsilon$ would make the graph disconnected, which contradicts the Jaeger-tree-for-point algorithm. We conclude that $(v,\varepsilon)$ precedes $(u,\varepsilon)$ in $S$.
	
	The algorithm maintains the properties that the current graph is connected and its root polytope contains $\mathbf{p}$. Hence, according to Theorem \ref{t:J-trees_form_quasitr}, there is a tree $T'\subseteq G'$ that is a Jaeger tree of $G'$ such that $\mathbf{p}\in \mathcal Q_{T'}$. (Here $T'$ is a spanning tree for $G$ as well, but we do not yet know if $T'$ is Jaeger for $G$.)
	Let us compare $T$ to $T'$.
	The first difference between the tours of $T$ and $T'$ can be of two types: Either the tour of $T'$ differs from the tour of $T$ before reaching $(v,\varepsilon)$, or $\varepsilon$ is the first difference.
	In the first case (since $G'$ does not contain any non-edges of $T$ that are seen in the tour of $T$ before $(v,\varepsilon)$) the first difference is an edge that $T'$ does not contain but $T$ contains.
	All the edges of $T$ that are traversed up to this point from the tail direction must be in $T'$ since these edges were needed either for keeping the (then) current graph connected or for having $\mathbf{p}$ in the root polytope of the (then) current graph, and this remains true also for the potentially smaller graph $G'$. Hence the first difference between $T$ and $T'$ is that $T'$ does not contain some edge of $T$ (and of $G'$) that is first seen from the head direction. 
	This cannot happen since $T'$ is a Jaeger tree within $G'$. 
	
	Finally, if $\varepsilon$ is the first difference, then $T' \prec T$. This cannot happen as $T$ is the smallest tree containing $\mathbf{p}$ in its root polytope. Thus the sequence $S$ describes the Jaeger-tree-for-point algorithm as we claimed, and the proof is complete.
\end{proof}

\subsection{A short proof for the correctness of the Bernardi process}
\label{ss:Bernardi_alg_well_def}

This subsection is a small detour in which we recall the Bernardi algorithm of \cite{hyperBernardi}, and point out that its 
convergence
follows from the correctness of the Jaeger-tree-for-point algorithm. This approach significantly shortens the proof given in \cite[Section 4]{hyperBernardi}. In this subsection we deal with bipartite graphs with the standard orientation. 
Concretely, let us orient each edge toward the partite class $W$.

We will introduce the required notions as briefly as possible; for a more detailed explanation, see \cite
{alex,hiperTutte,KP_Ehrhart,hyperBernardi}.

\begin{defn}\cite{hiperTutte}
	\label{def:hypertree}
	Let $G$ be an undirected bipartite graph with vertex classes $U$ and $W$. We say that the vector $\mathbf h\colon U\to\Z_{\ge0}$ is a  \emph{hypertree} on $U$ if there exists a spanning tree $T$ of $G$ that has degree $d_T(v)=\mathbf h(v)+1$ at each node $v\in U$. 
    In this case, we also say that the spanning tree $T$ \emph{realizes} the hypertree $\mathbf h$.
\end{defn}

Hypertrees are a generalization of spanning trees. Indeed, if the bipartite graph $G$ is obtained from a graph $H$ by subdividing each edge with a new point, and $U$ is the color class of these new points, then hypertrees on $U$ are exactly the characteristic vectors of the spanning trees of $H$.

It is proved in \cite{hyperBernardi} that (given a ribbon structure and basis for $G$) for each hypertree on $U$, there is a unique Jaeger tree in $G$ that realizes it. The Bernardi algorithm of \cite{hyperBernardi} is a process that computes this Jaeger tree for a given hypertree.
Clearly, by symmetry, the same can be said for hypertrees on $W$.

Let the base node be $b_0$ and the base edge be $b_0b_1$. 

\begin{defn}[Bernardi algorithm] \label{d:Bernardi,em,vi}
	Given is a hypertree $\mathbf f$ on $U$ (or on $W$). The process maintains a current incident node-edge pair and a current graph at any moment. At the beginning, the current graph is $G$, the current node is $b_0$ and the current edge is $b_0b_1$.
	
	In a step, if the current node $x$ is in $U$, then we check whether for the current graph $G_{curr}$ and the current edge $xy$, the vector $\mathbf f$ is a hypertree on $U$ (or on $W$) in the graph $G_{curr}-xy$. If the answer is yes, let the current graph of the next step be $G_{curr}-xy$ and the current node-edge pair be $(x, xy_{G_{curr}}^+)$. If the answer is no, let the current graph of the next step be $G_{curr}$, and let the current node-edge pair of the next step be $(y,yx_{G_{curr}}^+)$. 
	
	If the current node $x$ is in $W$, and the current edge is $xy$, then let the current graph of the next step be $G_{curr}$, and let the current node-edge pair of the next step be $(y, yx_{G_{curr}}^+)$.
	
	The process stops when a node-edge pair becomes current for the second time.
	The output is the current graph at the end of the process.
\end{defn}

This is called the (ht:$U$,cut:$U$) Bernardi process if the hypertree is given on $U$ and the (ht:$W$,cut:$U$) Bernardi process if the hypertree is given on $W$. Clearly, by symmetry, one can also define the algorithms 
(ht:$W$,cut:$W$) and (ht:$U$,cut:$W$). 

The key observation is that the Bernardi algorithm is an instance of the Jaeger-tree-for-point algorithm, run on an appropriate point, which we introduce next. 

\begin{defn}[marker of a hypertree] 
	For a hypertree $\mathbf h\in\R^U$, 
	let its \emph{marker} be the point 
	$\frac{-1}{|W|}\mathbf h
	+ \frac{-1}{|U|\cdot|W|}\mathbf{1}_U + \frac{1}{|W|}\mathbf{1}_W$ of $\R^V=\R^U\oplus\R^W$.
	The \emph{marker} of a hypertree $\mathbf h\in\R^W$ is 
	$\frac{1}{|U|}\mathbf h
	+ \frac{1}{|U|\cdot|W|}\mathbf{1}_W + \frac{-1}{|U|}\mathbf{1}_U$.
\end{defn} 

The following proposition, as explained in \cite[section 3.4]{KP_Ehrhart}, is a direct consequence of some of Postnikov's observations \cite{alex}.

\begin{prop}\cite{alex}\label{p:marker_pont_akkor_van_benne_ha_realizal}
	For any hypertree $\mathbf h$ either on $U$ or on $W$, and any spanning tree $T$ of $G$, the marker of $\mathbf h$ is in $\mathcal Q_T$ if and only if $T$ realizes $\mathbf h$, and in this case the marker is an interior point of $\mathcal Q_T$. 
\end{prop}

We now give a new proof that the Bernardi algorithm gives the desired result.

\begin{thm}
	\label{thm:Bernardi_well_def}
	For any hypertree $\mathbf h$ on $U$ (respectively, on $W$), the (ht:$U$, cut:$U$) Bernardi process (respectively, the (ht:$W$, cut:$U$) Bernardi process)
	returns the unique Jaeger tree that realizes $\mathbf h$.
\end{thm}

\begin{proof}
	Consider the marker of the hypertree $\mathbf h$. 
	Notice that by Proposition \ref{p:marker_pont_akkor_van_benne_ha_realizal}, the (ht:$U$,cut:$U$) (resp., the (ht:$W$,cut:$U$)) Bernardi algorithm on input $\mathbf h$ runs exactly as the Jaeger-tree-for-point algorithm on the marker of $\mathbf h$. 
	Hence the correctness of the Jaeger-tree-for-point algorithm implies the well definedness of the Bernardi algorithm.
\end{proof}

\section{Semi-balanced plane 
graphs and the greedoid polynomial}
\label{sec:planar}


In this section we take a closer look at the special case where the ribbon structure of our 
graph is induced by an embedding into the plane. 
More precisely, we will consider semi-balanced graphs $G$ embedded into the plane
(i.e., edges are not allowed to cross), with the cyclic order around each vertex given by
the positive orientation of the plane. We will show that in this case Jaeger trees give a triangulation of the root polytope $\mathcal Q_G$, moreover, the interior polynomial of $G$ (i.e., the $h^*$-vector of $Q_G$) is a simple transformation of the greedoid polynomial of the branching greedoid of the planar dual $G^*$ of $G$. We will also see that the triangulation of the root polytope by Jaeger trees is a geometric embedding of the dual complex of the branching greedoid of $G^*$.

\subsection{Jaeger trees and spanning arborescences}

As usual for plane graphs, duality will play an important role. 

\begin{defn}[Dual of a directed plane graph] 
Let $G$ be a directed plane graph. Its \emph{dual} $G^*$ is the directed graph whose underlying undirected graph is the dual of the underlying undirected graph of $G$, and each edge of $G^*$ is oriented so that the edge of $G$ has to be turned in the positive direction to obtain the corresponding edge of $G^*$.
\end{defn}

See Figure \ref{f:Jaeger_tree_and_duality} for an example. We note that with this notion, the usual involutive rule gets refined to $(G^*)^*=-G$.
For an edge $\varepsilon$ of $G$, we denote by $\varepsilon^*$ the edge of the dual graph corresponding to $\varepsilon$. It is well known that for any spanning tree $T$ of $G$, the edges of $G^*$ corresponding to the non-edges of $T$ form a spanning tree of $G^*$, and vice versa. We denote this spanning tree of $G^*$ by $T^*$.

Let us point out several useful properties.

\begin{claim}
\label{cl:dual_Euler}
A directed plane 
graph is semi-balanced if and only if its dual is an Eulerian directed graph.
\end{claim}

\begin{proof}
	Recall that $G$ is semi-balanced if and only if each cycle has equal numbers of edges going in both directions. In the dual, this means that each cut has the same number of edges going in both directions, which is equivalent to the directed graph being Eulerian.
\end{proof}

Also, similarly to the case of the standard orientation \cite{KM}, Jaeger trees correspond to nice objects in the dual. 

\begin{defn}
	An \emph{arborescence} of a directed graph, rooted at a node $r_0$, is a subgraph containing $r_0$ 
	that is a tree in the undirected sense, with the property that each edge is directed away from $r_0$. If the underlying tree is a spanning tree, then we talk of a \emph{spanning arborescence}.
\end{defn}

\begin{claim}\label{cl:Jaeger_dual_arborescence}
	Let $G$ be a semi-balanced 
	plane graph, and let $(b_0,b_0b_1)$ be the basis of the ribbon structure. Let $r_0$ be the tail of $b_0b_1^*$ if $b_0$ is the tail of $b_0b_1$ and let $r_0$ be the head of $b_0b_1^*$ if $b_0$ is the head of $b_0b_1$.
	
	Then $T$ is a Jaeger tree in $G$ (with basis $(b_0,b_0b_1)$) if and only if $T^*$ is a spanning arborescence of $G^*$ rooted at $r_0$. 
\end{claim}

For an illustration of the statement, see Figure \ref{f:Jaeger_tree_and_duality}.

\begin{proof} 
Note that the tour of a spanning tree $T$ of $G$, with respect to the basis $(b_0,b_0b_1)$, corresponds naturally to the tour of $T^*$ with respect to the basis $(r_0, b_0b_1^*)$ and the \emph{negative} orientation of the plane. At the beginning, the current edges in the two tours are dual pairs, and one can show by a trivial induction that this remains so throughout the process: An edge of $G$ will be traversed if and only if the corresponding edge of $G^*$ is skipped, and one can check that both in the case when the current edge is traversed in $G$ and when it is skipped, the next current edges in the tours 
are again dual pairs. 

Now $T$ is a Jaeger tree if and only if each non-edge of $T$ is first seen at its tail in the tour. This is equivalent to the requirement that when touring $T^*$, each edge be first traversed in the direction from tail to head. That, in turn, is equivalent to $T^*$ being a spanning arborescence rooted at $r_0$.
\end{proof}

\begin{figure}
	\begin{center}
		\begin{tikzpicture}[-,>=stealth',auto,scale=0.5,
		thick]
		\tikzstyle{e}=[{circle,draw,font=\sffamily\small}]
		\tikzstyle{v}=[{circle,draw,fill,font=\sffamily\small}]
		\tikzstyle{r}=[{circle,draw,fill,color=light-gray,font=\sffamily\small}]
		\node[e,label=below:$b_1$] (0) at (0, 0.8) {};
		\node[v] (1) at (3, 2.3) {};
		\node[v,label=left:$b_0$] (2) at (-3, 2.3) {};
		\node[e] (3) at (3, 4.7) {};
		\node[e] (4) at (-3, 4.7) {};
		\node[v] (5) at (0, 6.2) {};
		\node[e] (6) at (0, 3.6) {};
		\node[r,label=right: \ $r_0$] (7) at (5, 4.2) {};
		\node[r] (8) at (1.5, 4.2) {};
		\node[r] (9) at (-1.5, 4.2) {};
		\node[r] 
		(10) at (0, 2.2) {};
		\draw[->,color=gray,line width=0.7mm,rounded corners=10pt] (7) |- (1, 7.2) -- (0, 7.2)  -|  (9);
		\draw[->,color=gray,line width=0.7mm] (7) to[bend left=30] (8);
		\draw[->,color=gray,line width=0.7mm] (8) to[bend left=30] (10);
		\draw[->,color=gray, dashed, rounded corners=10pt] (8) |- (3, 6) --  (7);
		\draw[->,color=gray, dashed] (9) to[bend right=30] (10);
		\draw[->,color=gray, dashed] (9) to[bend left=30] (8);
		\draw[->,color=gray, dashed, rounded corners=15pt] 
		(7) -- (6.2, 4.9) |- (0, 8.2) -| (-4.5, 3.3) -- (9);
		\draw[->,color=gray, dashed, rounded corners=20pt] (10) -- (5, 0) --  (7);
		\draw[->,color=gray, dashed, rounded corners=18pt] (10) -- (-2.5, -0.8) -- (6.2, -1) -- (6.2, 3) -- (7);
		\path[->,every node/.style={font=\sffamily\small},dashed]
		(5) edge node {} (4)
		(1) edge node {} (3)
		(1) edge node {} (6);
		\path[->,every node/.style={font=\sffamily\small}, line width=0.7mm]
		(0) edge node {} (2)
		(1) edge node {} (0)
		(4) edge node {} (2)
		(5) edge node {} (3)
		(6) edge node {} (2)
		(5) edge node {} (6);
		\end{tikzpicture}
	\end{center}
	\caption{A semi-balanced 
	plane graph (in black) and its dual graph (in gray). 
	The thick black edges form a Jaeger tree for the basis $(b_0, b_0b_1)$. The thick gray edges form the corresponding spanning arborescence in the dual.
	\label{f:Jaeger_tree_and_duality}}
\end{figure}
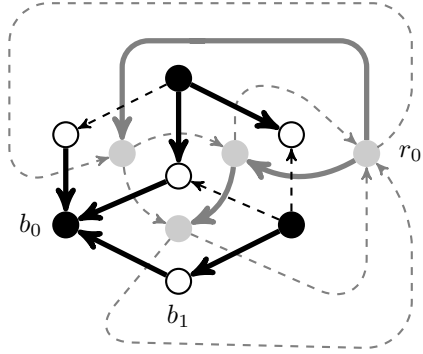

The previous claims imply that in the special case of plane graphs, the statement that the number of Jaeger trees of a semi-balanced 
graph is independent of the basis is equivalent to the well-known statement that in an Eulerian directed graph the number of spanning arborescences 
is independent of the root.

\subsection{Jaeger trees triangulate in the planar case}
\label{ss:planar_triangulation}

Recall that a triangulation is the special case of a dissection where the intersection of any pair of simplices is their
(possibly empty)
common face. In this subsection, we show the following generalization of \cite[Theorem 1.1]{KM} (which addressed the standard orientation).

\begin{prop}\label{p:planar_triangulation}
	For a semi-balanced plane graph $G$, Jaeger trees (for any basis) induce a triangulation of $\mathcal Q_G$.
\end{prop}

Theorem \ref{t:J-trees_form_quasitr} tells us that the Jaeger trees dissect $\mathcal Q_G$, hence we only need to show that 
the simplices corresponding to 
any two Jaeger trees meet in a common face. 
To do this, we first establish a necessary and sufficient condition for the simplices corresponding to two spanning trees, in a general semi-balanced directed graph, to meet in a common face.
We achieve this by generalizing \cite[Lemma 12.6]{alex} that deals with the case of the standard orientation.

Let $G$ be a semi-balanced graph and let $T_1$, $T_2$ be two spanning trees of $G$.
For a cycle $C$ of $G$ let us pick an edge $\varepsilon_0$, and call $C^+$ the set of edges of $C$ oriented consistently with $\varepsilon_0$; call also $C^-$ the set of edges of $C$ standing opposite to $\varepsilon_0$. 

\begin{defn}
\label{def:incompatible}
We say that $C$ is an \emph{incompatible cycle} for the trees $T_1$ and $T_2$ if, for some $\varepsilon_0\in T_1\cap C$, we have $C^+ \subset T_1$ and $C^- \subset T_2$.
\end{defn}

\begin{lemma}\label{l:compatibility_of_trees}
	For two spanning trees $T_1$ and $T_2$ of the semi-balanced graph $G$, the simplices $\mathcal Q_{T_1}$ and $\mathcal Q_{T_2}$ intersect in a common face if and only if there is no incompatible cycle for $T_1$ and $T_2$. 
\end{lemma}

\begin{proof}
	The proof of this statement is completely analogous to that of \cite[Lemma 12.6]{alex}.
	First we show that if there is an incompatible cycle $C$, say of length $2k$, for $T_1$ and $T_2$, then $\mathcal Q_{T_1}$ and $\mathcal Q_{T_2}$ do not meet in a common face. 
	If they did then for any point $\mathbf p\in \mathcal{Q}_{T_1}\cap \mathcal{Q}_{T_2}$, the minimal faces of $\mathcal Q_{T_1}$ and $\mathcal Q_{T_2}$, containing $\mathbf p$, would also coincide.
	Now by Lemma \ref{l:dependence_of_semibalanced_cycle}, for the cycle $C$ we have
	$$\frac1k\sum_{\varepsilon\in C^+}\mathbf x_\varepsilon =\frac1k\sum_{\varepsilon\in C^-}\mathbf x_\varepsilon=:\mathbf p.$$ 
	By our assumption, $\frac1k\sum_{\varepsilon\in C^+}\mathbf x_\varepsilon$ 
	is in $\mathcal Q_{T_1}$ while $\frac1k\sum_{\varepsilon\in C^-}\mathbf x_\varepsilon$ 
	is in $\mathcal Q_{T_2}$, thus $\mathbf p\in\mathcal Q_{T_1}\cap \mathcal Q_{T_2}$. 
	However, the minimal face containing $\mathbf p$ in $\mathcal Q_{T_1}$ is $\mathcal Q_{C^+}$, while the minimal face containing $\mathbf p$ in $\mathcal Q_{T_2}$ is $\mathcal Q_{C^-}$, which cannot coincide because $C^+\cap C^-=\varnothing$. Hence indeed $\mathcal Q_{T_1}$ and $\mathcal Q_{T_2}$ do not meet in a common face.
	
	Conversely, if there is no incompatible cycle $C$, then it suffices to provide a linear functional $f$ such that $f(\mathbf{p})=0$ for $\mathbf{p}\in \mathcal Q_{T_1}\cap \mathcal Q_{T_2}$, $f(\mathbf{p})>0$ for $\mathbf{p}\in \mathcal Q_{T_1}-\mathcal Q_{T_2}$, and $f(\mathbf{p})<0$ for $\mathbf{p}\in \mathcal Q_{T_2}-\mathcal Q_{T_1}$. 
	Let us define $f$ on standard basis vectors, and extend it linearly. Let $\overleftarrow{T}_2$ be the tree where we reverse the orientation of the edges of $T_2$, and consider $T_1\cup \overleftarrow{T}_2$. This graph might have some cycles of length two (one for each edge of $T_1\cap T_2$) but the non-existence of incompatible cycles implies that it has no directed cycle of length more than two. 
	
	By contracting the cycles of length two, $T_1\cup \overleftarrow{T}_2$ becomes an acyclic graph $A$. By the acyclicity we can assign weights $f$ to the nodes of $A$ so that each edge has a larger weight on its head than on its tail. Now undo the contractions, and define weights on the original vertex set $V$ by the obvious pull-back operation.
	Then every edge of $T_1\cap T_2$ has equal weights on its head and tail, each edge of $T_1-T_2$ has larger weight on its head, and each edge of $T_2-T_1$ has larger weight on its tail. Hence $f(\mathbf x_\varepsilon)=0$ for $\varepsilon\in T_1\cap T_2$, $f(\mathbf x_\varepsilon)>0$ for $\varepsilon\in T_1 - T_2$, and $f(\mathbf x_\varepsilon)<0$ for $\varepsilon\in T_2 - T_1$, just as we desired. 
\end{proof}

\begin{proof}[Proof of Proposition \ref{p:planar_triangulation}]
Suppose for a contradiction that there are two Jaeger trees $T_1$ and $T_2$ such that $\mathcal Q_{T_1}$ and $\mathcal Q_{T_2}$ do not meet in a common face. Then by Lemma \ref{l:compatibility_of_trees}, there is a cycle $C$ in $G$ such that $C^+\subset T_1$ and $C^- \subset T_2$. By Claim \ref{cl:Jaeger_dual_arborescence}, the duals $T^*_1$ and $T^*_2$ are spanning arborescences rooted at the same node $r_0$. 
The cycle $C$ yields a cut in the dual graph $G^*$, with $C^+$ corresponding to edges going in one direction, and $C^-$ to edges going in the other direction. For $T^*_1$ and $T^*_2$, the incompatibility means that $T^*_1$ only contains edges going in one of the directions, while $T^*_2$ only contains edges going in the other direction. But $r_0$ is on one of the two sides of the cut and both spanning arborescences should contain an edge leading from the side of $r_0$ to the other side. That is a contradiction.
\end{proof}

\subsection{Greedoid polynomial}\label{ss:greedoid_poly}
Korte and Lovász defined the notion of a greedoid \cite{KorteLovasz} as a generalization of matroids where the greedy algorithm still works.
\begin{defn}\cite{greedoid}\label{def:greedoid_def}
A \emph{greedoid} on a set $E$ is a set system $\mathcal{F}\subseteq 2^E$ whose elements are called \emph{accessible sets}. Here $\mathcal{F}$ needs to satisfy the following axioms.
\begin{enumerate}
    \item[(i)] $\varnothing \in \mathcal{F}$.
    \item[(ii)] If $X\in\mathcal{F}$ then there exists some $x\in X$ such that $X-\{x\}\in\mathcal{F}$.
    \item[(iii)] If $X,Y\in \mathcal{F}$ and $|X|=|Y|+1$, then there exists an element $x\in X-Y$ such that $Y\cup \{x\} \in \mathcal{F}$.
\end{enumerate}
The maximal accessible sets are called \emph{bases}.
\end{defn}

One important example of greedoids is the so-called directed branching greedoid.
Suppose that $G$ is a directed graph with a fixed node $r_0$, which is such that every vertex is reachable along a directed path from $r_0$.
The ground set of the \emph{branching greedoid} of $G$ is the set $E(G)$ of edges. 
The accessible sets are the arborescences rooted at $r_0$, and the bases are the spanning arborescences. It is easy to check that arborescences do indeed satisfy the axioms in Definition \ref{def:greedoid_def}.

Björner, Korte, and Lov\'asz introduced the dual complex of a greedoid as the abstract simplicial complex whose maximal simplices are the complements of the bases of the greedoid \cite{greedoid}. They 
defined the \emph{greedoid polynomial} \cite{greedoid} as 
\[\lambda(t)=\sum_{i=0}^{m-r}h_i t^{m-r-i},\] 
where $(h_0, h_1,\dots)$ is the $h$-vector of the dual complex of the greedoid, $m$ is the cardinality of the ground set, and $r$ is the rank (the common cardinality of the bases).
In the special case when the greedoid is a matroid, the greedoid polynomial agrees with $T(1,y)$, where $T(x,y)$ is the Tutte polynomial.

In this subsection we show that for an Eulerian plane directed graph $G$, the triangulation of the root polytope $\mathcal Q_{G^*}$, induced by Jaeger trees, is a geometric realization of the dual complex of the branching greedoid of $G$. We conclude that the greedoid polynomial of $G$ can be obtained as a simple transformation of the interior polynomial of $G^*$. In other words, we obtain the greedoid polynomial as a transformation of the Ehrhart series of $\mathcal Q_{G^*}$.

We saw in Claim \ref{cl:dual_Euler} that for any Eulerian plane directed graph $G$, its dual $G^*$ is a semi-balanced plane graph (note that a directed graph $H$ is semi-balanced if and only if $-H$ is).
As Claim \ref{cl:Jaeger_dual_arborescence} says, for any fixed vertex $r_0$, the complements of the spanning arborescences of $G$ rooted at $r_0$ are exactly the Jaeger trees for the basis $(b_0,b_0b_1)$, where $r_0$ is the tail (head) of $b_0b_1^*$ if $b_0$ is the tail (head) of $b_0b_1$.

By Proposition \ref{p:planar_triangulation}, Jaeger trees induce a triangulation of the root polytope $\mathcal{Q}_{G^*}$. Thus we may view our construction (of $\mathcal Q_{G^*}$ and simplices within) as an abstract simplicial complex, and we automatically
obtain the following.

\begin{thm}\label{thm:root_poly_geom_real_of_dual_complex}
	For an Eulerian plane directed graph $G$, the triangulation of $\mathcal Q_{G^*}$ by Jaeger trees (with basis $(b_0,b_0b_1)$) is 
	isomorphic to 
	the dual complex of the branching greedoid of $G$ with root $r_0$, where $r_0$ is the tail (head) of $b_0b_1^*$ if $b_0$ is the tail (head) of $b_0b_1$. 
\end{thm}

Let us denote the greedoid polynomial of the branching greedoid of an Eulerian plane graph $G$, with root $r_0$, by $\lambda_{G,r_0}$. Note that the dimension of $\mathcal Q_{G^*}$ is $|V(G^*)|-2=|E(G)|-|V(G)|$ by Euler's formula.
Then, using \eqref{eq:intash}, we obtain
\begin{multline*}
\lambda_{G,r_0}(t)=t^{|E|-|V|+1} \sum_{i=0}^{|E|-|V|+1}h_i\cdot t^{-i} = 
t^{|E|-|V|+1} I_{G^*}(t^{-1})\\
=t^{|E|-|V|+1} 
(1-t^{-1})^{|E|+|V|-1} \ehr_{\mathcal Q_{G^*}}(t^{-1}),
\end{multline*}
that is, we can write the greedoid polynomial $\lambda_{G,r_0}$ as a transformation of the interior polynomial of $G^*$, or equivalently, as a transformation of the Ehrhart series of $\mathcal Q_{G^*}$.
Note that the definition of $I_{G^*}$ is independent of the choice of $r_0$.
This shows that the greedoid polynomial of the branching greedoid of an Eulerian plane graph is independent of the root vertex. 

We note that the root-independence of the greedoid polynomial has been known in greater generality:

\begin{thm}[\cite{SweeHong_parking,PP16}]
	For any Eulerian directed graph, the greedoid polynomial of the branching greedoid is independent of the chosen root vertex.
\end{thm}

In \cite{SweeHong_parking}, 
Chan uses the recurrent configurations of the sinkless sandpile model to give a canonical definition for the greedoid polynomial. 
We wonder if one can give a geometric canonical definiton in the non-planar case.

\begin{question}
	Is there a canonical (root-independent) geometric embedding for the dual complex of the branching greedoid in the case of non-planar Eulerian directed graphs?
\end{question}

Björner, Korte, and Lov\'asz proved that the dual complex of a greedoid is always shellable \cite[Lemma 5.5]{greedoid}. Their shelling order is a lexicographic order of bases derived from a fixed linear order of the ground set. In the special case of planar branching greedoids, Theorem \ref{thm:shelling} also yields a shelling order for the dual complex. This is typically different from the shelling order given in \cite{greedoid}.

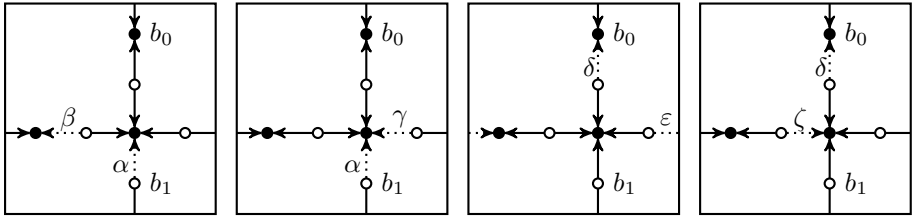
\begin{figure}
	\begin{center}
	  \begin{tikzpicture}[-,>=stealth',auto,scale=1.4,thick]
	    \tikzstyle{b}=[circle,scale=0.4,fill,draw]
	    \tikzstyle{w}=[circle,scale=0.4,draw]
	    \begin{scope}[shift={(-3.3,0)}]
	    \draw  (0, 0) rectangle (2,2);
	    \node[b] (0) at (0.29, 0.77) {};
	    \node[w] (1) at (0.77, 0.77) {};
	    \node[b] (2) at (1.23, 0.77) {};
	    \node[w] (3) at (1.71, 0.77) {};
	    \node[w,label=right:{$b_1$}] (4) at (1.23, 0.29) {};
	    \node[w] (5) at (1.23, 1.23) {};
	    \node[b,label=right:{$b_0$}] (6) at (1.23, 1.71) {};
	    \draw[->, dotted] (1) -- (0);
        \draw[->] (1) -- (2);
        \draw[->] (3) -- (2);
        \draw[->, dotted] (4) -- (2);
        \draw[->] (5) -- (2);
        \draw[->] (5) -- (6);
        \draw[->] (0,0.77) -- (0);
        \draw[-] (3) -- (2, 0.77);
        \draw[-] (4) -- (1.23, 0);
        \draw[->] (1.23, 2) -- (6);
        \node at (1.1,0.45) {$\alpha$};
        \node at (0.6,0.9) {$\beta$};
        \end{scope}
	    \begin{scope}[shift={(-1.1,0)}]
	    \draw  (0, 0) rectangle (2,2);
	    \node[b] (0) at (0.29, 0.77) {};
	    \node[w] (1) at (0.77, 0.77) {};
	    \node[b] (2) at (1.23, 0.77) {};
	    \node[w] (3) at (1.71, 0.77) {};
	    \node[w,label=right:{$b_1$}] (4) at (1.23, 0.29) {};
	    \node[w] (5) at (1.23, 1.23) {};
	    \node[b,label=right:{$b_0$}] (6) at (1.23, 1.71) {};
	    \draw[->] (1) -- (0);
        \draw[->] (1) -- (2);
        \draw[->, dotted] (3) -- (2);
        \draw[->, dotted] (4) -- (2);
        \draw[->] (5) -- (2);
        \draw[->] (5) -- (6);
        \draw[->] (0,0.77) -- (0);
        \draw[-] (3) -- (2, 0.77);
        \draw[-] (4) -- (1.23, 0);
        \draw[->] (1.23, 2) -- (6);
        \node at (1.1,0.45) {$\alpha$};
        \node at (1.55,0.9) {$\gamma$};
	    \end{scope}
	    \begin{scope}[shift={(1.1,0)}]
	    \draw  (0, 0) rectangle (2,2);
	    \node[b] (0) at (0.29, 0.77) {};
	    \node[w] (1) at (0.77, 0.77) {};
	    \node[b] (2) at (1.23, 0.77) {};
	    \node[w] (3) at (1.71, 0.77) {};
	    \node[w,label=right:{$b_1$}] (4) at (1.23, 0.29) {};
	    \node[w] (5) at (1.23, 1.23) {};
	    \node[b,label=right:{$b_0$}] (6) at (1.23, 1.71) {};
	    \draw[->] (1) -- (0);
        \draw[->] (1) -- (2);
        \draw[->] (3) -- (2);
        \draw[->] (4) -- (2);
        \draw[->] (5) -- (2);
        \draw[->, dotted] (5) -- (6);
        \draw[->, dotted] (0,0.77) -- (0);
        \draw[-,dotted] (3) -- (2, 0.77);
        \draw[-] (4) -- (1.23, 0);
        \draw[->] (1.23, 2) -- (6);
        \node at (1.15,1.4) {$\delta$};
        \node at (1.88,0.9) {$\varepsilon$};
	    \end{scope}
	    \begin{scope}[shift={(3.3,0)}]
	    \draw  (0, 0) rectangle (2,2);
	    \node[b] (0) at (0.29, 0.77) {};
	    \node[w] (1) at (0.77, 0.77) {};
	    \node[b] (2) at (1.23, 0.77) {};
	    \node[w] (3) at (1.71, 0.77) {};
	    \node[w,label=right:{$b_1$}] (4) at (1.23, 0.29) {};
	    \node[w] (5) at (1.23, 1.23) {};
	    \node[b,label=right:{$b_0$}] (6) at (1.23, 1.71) {};
	    \draw[->] (1) -- (0);
        \draw[->, dotted] (1) -- (2);
        \draw[->] (3) -- (2);
        \draw[->] (4) -- (2);
        \draw[->] (5) -- (2);
        \draw[->, dotted] (5) -- (6);
        \draw[->] (0,0.77) -- (0);
        \draw[-] (3) -- (2, 0.77);
        \draw[-] (4) -- (1.23, 0);
        \draw[->] (1.23, 2) -- (6);
	    \node at (1.15,1.4) {$\delta$};
        \node at (0.95,0.9) {$\zeta$};
	    \end{scope}
	  \end{tikzpicture}
	\end{center}
	\caption{The four Jaeger trees of a graph embedded into a torus. 
	\label{f:non-greedoid_dual}}
\end{figure}

\subsection{Some more thoughts on 
Jaeger trees and greedoids}

One might wonder if for general ribbon structures, Jaeger trees are the complements of the bases of a greedoid. In this section we give an example where the answer is negative, then we 
pose a question. 

Figure \ref{f:non-greedoid_dual} shows a graph (with a ribbon structure and a basis) where the complements of the Jaeger trees do not constitute the bases of any greedoid.
This is a semi-balanced graph embedded into the torus. The ribbon structure induced by the embedding has
four Jaeger trees, as shown in the figure. Here the basis is $(b_0,b_0b_1)$ and we refer to the figure for notation.
The complements of the Jaeger trees are $\{\alpha,\beta\}$, $\{\alpha,\gamma\}$, $\{\delta,\varepsilon\}$, and $\{\delta,\zeta\}$, respectively. Now if these sets were the bases of a greedoid, then either $\{\alpha\}$ or $\{\beta\}$ would have to be an accessible set, since $\{\alpha,\beta\}$ is accessible. But then either $\{\alpha,\delta\}$, $\{\alpha,\varepsilon\}$, $\{\beta,\delta\}$, or $\{\beta,\varepsilon\}$ would need to be accessible, since $\{\delta,\varepsilon\}$ is accessible. As none of them is the complement of a Jaeger tree, indeed, $\{\alpha,\beta\}$, $\{\alpha,\gamma\}$, $\{\delta,\varepsilon\}$, and $\{\delta,\zeta\}$ do not constitute the bases of any greedoid.

\begin{question}\label{q:triang_greedoid_relationship}
    Suppose that a semi-balanced graph is given with a ribbon structure and a basis.
    Is it true that the simplices corresponding to
    Jaeger trees form a triangulation of the root polytope if and only if there exists a greedoid on the edge set such that the Jaeger trees are exactly the complements of the bases?
\end{question}

We have checked some examples, and for those, the two phenomena always coincided.
We note that both the problem of determining if some sets constitute the bases of a greedoid (see \cite{greedoid}), and determining if Jaeger trees form a triangulation, 
seem to be computationally
hard ones. In Section \ref{sec:diss_vs_triang}, we will examine the question of whether Jaeger trees form a triangulation from a more geometric point of view.






\section{Complete and layer-complete bipartite graphs}
\label{sec:layer_complete}

It is known \cite{GGP} that for (undirected) complete bipartite graphs drawn on the plane with the two partite classes on two parallel lines, the so-called non-crossing trees induce a triangulation of the root polytope. Moreover, for a certain ribbon structure and the standard orientation, 
non-crossing trees are exactly the Jaeger trees \cite{hyperBernardi}; in particular, Jaeger trees induce a triangulation in this case.


One of the aims of this section is to show that this picture can be generalized for all other semi-balanced orientations of a complete bipartite graph, too (cf.\ Example \ref{ex:K34}). Namely, for any semi-balanced graph where the underlying graph is complete bipartite, we give a ribbon structure and basis with respect to which the Jaeger trees provide a triangulation of the root polytope. Moreover, these Jaeger trees have a simple geometric description. 

Let $G$ be a complete bipartite graph oriented in a semi-balanced way.
By Theorem \ref{thm:char_semibalanced}, there is a layering $l\colon V \to \Z$ such that $l(h)-l(t)=1$ for each edge $\overrightarrow{th}$ of $G$. As any two nodes in different 
color classes are connected, $l$ can have at most $3$ values. We may suppose that these are $0$, $1$, and $2$. 
Let us draw $G$ in the plane so that the nodes with $l(v)=i$ are on the line with second coordinate equal to $i$, 
and all edges are straight line segments.
For an example, see Figure \ref{fig:complete_bipartite_embedding}.
Clearly, this representation will have many crossing edges. In fact, we will see two complete bipartite graphs ``on top of each other'': we have a complete bipartite graph between nodes having $l=0$ and nodes having $l=1$ and another between nodes having $l=1$ and nodes having $l=2$. Our ribbon structure will be the one induced by the positive orientation of the plane. The base node is the leftmost node with $l=0$, and the base edge is the edge connecting it to the rightmost node with $l=1$.

    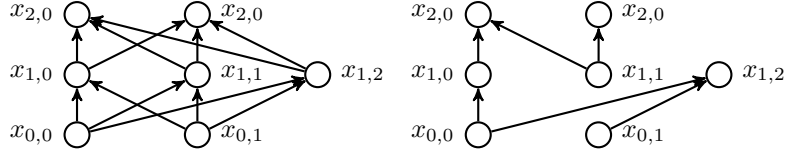
\begin{figure}
    	\begin{center}    
    		\begin{tikzpicture}[-,>=stealth',auto,scale=0.8,
    		thick]
    		\tikzstyle{o}=[circle,draw]
    		\node[o, label=left:$x_{0,0}$] (1) at (0, 0) {};
    		\node[o, label=right:$x_{0,1}$] (2) at (2, 0) {};
    		\node[o, label=left:$x_{1,0}$] (3) at (0, 1) {};
    		\node[o, label=right:$x_{1,1}$] (4) at (2, 1) {};
    		\node[o, label=right:$x_{1,2}$] (5) at (4, 1) {};
    		\node[o, label=left:$x_{2,0}$] (6) at (0, 2) {};
    		\node[o, label=right:$x_{2,0}$] (7) at (2, 2) {};
    		\path[->,every node/.style={font=\sffamily\small}]
    		(1) edge node [above] {} (3)
    		(1) edge node [below] {} (4)
    		(1) edge node [above] {} (5)
    		(2) edge node [below] {} (3)
    		(2) edge node [below] {} (4)
    		(2) edge node [below] {} (5)
    		(3) edge node [below] {} (6)
    		(3) edge node [below] {} (7)
    		(4) edge node [below] {} (6)
    		(4) edge node [below] {} (7)
    		(5) edge node [below] {} (6)
    		(5) edge node [above] {} (7);
    		\end{tikzpicture}
    		\begin{tikzpicture}[-,>=stealth',auto,scale=0.8,
    		thick]
    		\tikzstyle{o}=[circle,draw]
    		\node[o, label=left:$x_{0,0}$] (1) at (0, 0) {};
    		\node[o, label=right:$x_{0,1}$] (2) at (2, 0) {};
    		\node[o, label=left:$x_{1,0}$] (3) at (0, 1) {};
    		\node[o, label=right:$x_{1,1}$] (4) at (2, 1) {};
    		\node[o, label=right:$x_{1,2}$] (5) at (4, 1) {};
    		\node[o, label=left:$x_{2,0}$] (6) at (0, 2) {};
    		\node[o, label=right:$x_{2,0}$] (7) at (2, 2) {};
    		\path[->,every node/.style={font=\sffamily\small}]
    		(1) edge node [above] {} (3)
    		(1) edge node [above] {} (5)
    		(2) edge node [below] {} (5)
    		(3) edge node [below] {} (6)
    		(4) edge node [below] {} (6)
    		(4) edge node [below] {} (7);
    		\end{tikzpicture}
    	\end{center}
    	\caption{Left: A semi-balanced complete bipartite graph (in particular, a layer-complete directed graph with $k=2$). 
    	The basis of the ribbon structure 
    	is $(x_{0,0},x_{0,0}x_{1,2})$.
    		Right: A non-crossing comb-tree of the graph, 
    		with $U_1=\{x_{1,1}\}$ and $D_1=\{x_{1,2}\}$.}
    	\label{fig:complete_bipartite_embedding}
    \end{figure} 

We state our results in a slightly more general setting. 

\begin{defn}
A \emph{layer-complete directed graph} is a connected
semi-balanced graph $G$ with a layering $l\colon V(G)\to\Z$ so that, for each $i\in \Z$ with $l^{-1}(i)\neq\varnothing$ and $l^{-1}(i+1)\neq\varnothing$, we have a complete bipartite subgraph in $G$ between $l^{-1}(i)$ and $l^{-1}(i+1)$, with all edges oriented from $l^{-1}(i)$ to $l^{-1}(i+1)$.
\end{defn}

We suppose, without loss of generality, that $l^{-1}(0), \dots, l^{-1}(k)$ are the nonempty layers. 
We denote, for each $0\leq i\leq k$, the nodes in $l^{-1}(i)$ by $x_{i,0}, \dots, x_{i,s_i}$, and draw these on the line with second coordinate equal to $i$ in this order from left to right.

We define the ribbon structure and the basis the same way as in the special case above, that is, cyclic orders come from the positive orientation of the plane, $x_{0,0}$ is the base node, and $x_{0,0}x_{1,s_1}$ is the base edge.

For complete bipartite graphs with the standard orientation (i.e., 
layer-complete directed graphs with $k=1$),
it is known that the Jaeger trees for the above described ribbon structure and basis are exactly the so-called non-crossing trees \cite{hyperBernardi}, and that they yield a triangulation of the root polytope \cite{GGP}. Here \emph{non-crossing trees} are those spanning trees whose edges do not cross each other in the above drawing. We generalize them as follows.

\begin{defn}[non-crossing comb-tree]
For a layer-complete directed graph represented in the plane as explained above, we call a spanning tree $T$ a \emph{non-crossing comb-tree}, if for each layer $0\leq i \leq k$, there is a partition $l^{-1}(i)- \{x_{i,0}\}=D_i \sqcup U_i$, with $D_0=\varnothing$ and $U_k=\varnothing$, such that $T=T_1\cup \dots \cup T_k$ with $T_i$ being a non-crossing tree of the complete bipartite graph with vertex classes $\{x_{i-1,0}\}\cup U_{i-1}$ and $\{x_{i,0}\}\cup D_i$. 
\end{defn}

See Figure \ref{fig:complete_bipartite_embedding} for an example.
Notice that non-crossing comb-trees will always contain the edges $x_{i-1,0}x_{i,0}$ for each $i=1, \dots ,k$. This is the ``backbone'' of the comb that explains the name.

\begin{prop}\label{p:layer_complete_Jaeger}
	For a layer-complete 
	graph with plane-induced ribbon structure and basis $(x_{0,0},x_{0,0}x_{1,s_1})$, the Jaeger trees are exactly the non-crossing comb-trees. 
\end{prop}

\begin{proof} 
For each $i=1, \dots, k$, let $G_i$ be the subgraph of $G$ induced by the vertices in $l^{-1}(i-1)\cup l^{-1}(i)$, and let $G'_i$ be the subgraph of $G$ (and of $G_i$) induced by $\{x_{i-1,0},x_{i,0}\} \cup U_{i-1}\cup D_i$.

	It is easy to see that the tour of a non-crossing comb-tree $T=T_1\cup \dots \cup T_k$ first traverses $T_1$, then moves on to $T_2$ and so forth. 
	Here $T_i$ is a Jaeger tree within the subgraph $G'_i$ (this claim from \cite[Section 9]{hyperBernardi} is easy to check directly). 
	To show that $T$ is a Jaeger tree of $G$, it is enough to check that for each $i$, the edges of $G_i-G'_i$ are first seen at their tails. 
	An edge of $G_i-G'_i$ either has its tail in $T_{i-1}$ and its head in $T_{i}$ or it has its tail in $T_{i-1}$ and its head in $T_{i+1}$, or its tail in $T_{i}$ and its head in $T_{i+1}$. In each case, it is first seen at its tail.
	
	For the converse, that is, that non-crossing comb-trees exhaust all Jaeger trees, we apply the same strategy as in \cite[Section 9]{hyperBernardi}. Namely, it is enough to prove that any point $\mathbf{p}$ of the root polytope $\mathcal Q_G$ is in $\mathcal Q_T$ for some non-crossing comb tree $T$. As the simplices corresponding to Jaeger trees are interior-disjoint, this implies that there cannot be a Jaeger tree that is not a non-crossing comb-tree.
	
	Let $\mathbf{p}\in \mathcal Q_G$ be an arbitrary point. Then there is a convex combination $\mathbf{p}=\sum_{\varepsilon \in E(G)} \lambda_\varepsilon \mathbf x_\varepsilon$.
	We can interpret this sum as $\mathbf{p}$ being the convex combination of some points $\mathbf{p}_1, \dots, \mathbf{p}_k$ with $\mathbf p_i\in \mathcal Q_{G_i}$. As $G_i$ is complete bipartite, we know that $\mathcal Q_{G_i}$ is dissected by non-crossing trees, hence we may suppose that in the sum $\mathbf{p}=\sum_{\varepsilon \in E(G)} \lambda_\varepsilon \mathbf x_\varepsilon$, the support $\{\varepsilon\in E(G)\mid \lambda_\varepsilon\neq 0\}$ consists of (subgraphs of) non-crossing trees in each level. If necessary, let us add some edges to the support to obtain  a non-crossing tree $H_i$ in each level, and call the union of the $H_i$'s the graph $H$. 
	
	Next we are going to reduce $H$ by removing edges one by one, while preserving the properties that $H$ is connected, $\mathbf{p}\in \mathcal Q_H$, moreover, that each level $H_i$ is non-crossing and connected except for some isolated points (that is, the connected components of $H_i$ are isolated points except for one). Furthermore, at each stage of the reduction we require that neither $x_{i-1,0}$ nor $x_{i,0}$ is isolated in $H_i$ for any $1\leq i\leq k$. If at the end of the reduction process we achieve that each vertex $x_{i,j}$, for $j\neq 0$, either has only in-edges or only out-edges incident to it, then the $H$ of the final stage is a non-crossing comb-tree.
	
	We aim to remove edges in such a way that eliminates (non-leftmost) vertices with both in- and out-edges. We will do so layer-by-layer, starting from the bottom. 
	Suppose that $i$ is the smallest number such that there is a vertex $x_{i,j}$ with $j\neq 0$ that has both in- and out-edges incident to it. Choose such a vertex $x_{i,j}$ arbitrarily. 
	As $x_{i,j}$ is not isolated in $H_i$, there is a path from $x_{i,0}$ to $x_{i,j}$ in $H_i$. Similarly, $x_{i,j}$ is not isolated in $H_{i+1}$, hence there is a path from $x_{i,0}$ to $x_{i,j}$ in $H_{i+1}$. Altogether, there is a cycle in $H_i\cup H_{i+1}$ containing an in-edge and an out-edge of $x_{i,j}$. Let us call this cycle $C$, and suppose that the in- and the out-edge of $x_{i,j}$ are in $C^-$. Notice that an edge of $C\cap H_i$ is in $C^-$ if and only if on the path from $x_{i,0}$ to $x_{i,j}$ we traverse it according to its orientation. On the other hand, an edge of $C\cap H_{i+1}$ is in $C^-$ if and only if on the path from $x_{i,0}$ to $x_{i,j}$ we traverse it opposite to its orientation.
	
	Let $\mu = \min_{\varepsilon\in C^-} \lambda_\varepsilon$. Define the new coefficients 
	$$\lambda'_\eta = \left\{\begin{array}{cl} 
	\lambda_\eta & \text{if $\eta\notin C$,}  \\
	\lambda_\eta - \mu  & \text{if $\eta\in C^-$},\\
	\lambda_\eta + \mu  & \text{if $\eta\in C^+$}
	\end{array} \right.
	$$
	(it is possible that $\mu=0$, in which case the coefficients do not change). 
	Then the new coefficients are still non-negative and sum to 1 by the semi-balanced property. By Lemma \ref{l:dependence_of_semibalanced_cycle}, we have $\mathbf{p}=\sum_{\eta\in E(H)} \lambda'_\eta \mathbf x_\eta$.
	For the new coefficients, $\min_{\eta\in C^-} \lambda'_\eta=0$. Let $\varepsilon\in C^-$ be an edge with $\lambda'_\varepsilon=0$. We note that $\varepsilon$ cannot be incident to $x_{i,0}$ since 
	(if $x_{i,0}$ lies along $C$ at all)
	such edges are in $C^+$. From here on we separate three cases.
	
	\smallskip
	
	\noindent
	Case 1: 
	If $H-\varepsilon$ satisfies all required properties, then we remove $\varepsilon$ from $H$ and our reduction step is complete. 
	
	\smallskip
	
	\noindent
	Case 2: Suppose that $\varepsilon\in H_i$, and $H-\varepsilon$ violates some property.
	 Clearly $H-\varepsilon$ is connected since $\varepsilon$ was part of $C$.
	Thus the only deficiency that $H-\varepsilon$ can have is that $H_i-\varepsilon$ has two connected components with more than one point. 

     Denote the endpoints of $\varepsilon$ with $x_{i-1,a}$ and $x_{i,b}$. Since $\varepsilon\in C^-$, on the path from $x_{i,0}$ to $x_{i,j}$, the edge $\varepsilon$ is reached at $x_{i-1,a}$. As $x_{i,0}$ is to the left of $x_{i,b}$, and $H_i$ is non-crossing, the edge of the path 
     preceding $\varepsilon$ (which clearly also belongs to $C$) needs to be an edge $x_{i-1,a}x_{i,b''}$ for some $b''<b$. Let $x_{i-1,a}x_{i,b'}$ be the out-edge of $x_{i-1,a}$ in $H$ such that $b'< b$ is maximal. (See the left panel of Figure \ref{fig:modification_of_H}.) 
    Similarly, let $x_{i-1,a'}x_{i,b}$ be the in-edge of $x_{i,b}$ in $H$ such that $a'> a$ is minimal. There must be such an edge, since we supposed that $H_i-\varepsilon$ had two connected components of more than one vertex. Hence $x_{i,b}$ cannot be isolated in $H_i-\varepsilon$, and by the non-crossing property of $H_i$, the other edges incident to it have to come from vertices $x_{i-1,a''}$ with $a''>a$. 
    
    Then for $H'=H-\varepsilon + x_{i-1,a'}x_{i,b'}$, the subgraph $H'_i$ is once again non-crossing, and has only one connected component with more than one vertex. Note that $x_{i-1,a'}$ has an out-edge in $H_i$, thus when we pass from $H$ to $H'$, we do not introduce vertices in layer $i-1$ that have both in- and out-edges. Moreover, $H'$ is connected, and by setting $\lambda'_{x_{i-1,a'}x_{i,b'}}=0$, we have $\mathbf{p}=\sum_{\eta\in H'}\lambda'_\eta \mathbf x_\eta$.
    
    To see why $H'$ is an improvement over $H$, let us introduce the new function (akin to slope) $m(\eta)=j'-j$ on edges  $\eta=x_{i-1,j}x_{i,j'}$ of $G$. 
    Then, note that 
    \begin{multline*}
    \sum_{\eta\in H'_i} m(\eta)=\left[\sum_{\eta\in H_i} m(\eta)\right]-m(\varepsilon) + m(\varepsilon')=\left[\sum_{\eta\in H_i} m(\eta)\right] - (b-a) + (b'-a') \\
    \leq \left[\sum_{\eta\in H_i} m(\eta)\right]-2. 
    \end{multline*}
    
    \smallskip
    
    \noindent
    Case 3: When $\varepsilon\in H_{i+1}$ and $H-\varepsilon$ lacks one of the desired properties, that again can only mean that $H_{i+1}-\varepsilon$ has two connected components with more than one point.
    Completely analogously to Case 2, in this case too we can substitute $\varepsilon$ with a different edge $\varepsilon'$ so that this time, $\sum_{\eta\in H_{i+1}} m(\eta)$ increases by at least 2. 
    
    \smallskip
    
Notice that Case 2 only modifies $H_i$ and Case 3 only modifies $H_{i+1}$, furthermore that there are obvious lower and upper bounds, say 
\[\sum_{\eta\in H_i} m(\eta)\geq (s_{i-1}+1)(s_i+1)(-s_{i-1})\text{ and }\sum_{\eta\in H_{i+1}} m(\eta)\leq (s_i+1)(s_{i+1}+1)s_{i+1}.\]
Therefore Cases 2 and 3 can only happen finitely many consecutive times. Hence after a while we need to have Case 1, which decreases the number of edges. This can only happen at most finitely many times as well, hence after finitely many steps, we will have no more vertices in layer $i$ with both in- and out-edges. Repeating this for all layers $i=1$ through $k-1$ completes the proof.
    \end{proof}
    
    \begin{figure}
	\begin{center}    
	\begin{tikzpicture}[-,>=stealth',auto,scale=1,
	thick]
	\tikzstyle{o}=[circle,draw]
	\node[o, label=below:$x_{i-1,a}$] (1) at (0, 0) {};
	\node[o, label=below:$x_{i-1,a''}$] (2) at (3, 0) {};
	\node[o, label=above:$x_{i,b''}$] (3) at (-1, 1) {};
	\node[o, label=above:$x_{i,b'}$] (4) at (0.5, 1) {};
	\node[o, label=below:$x_{i-1,a'}$] (5) at (1.5, 0) {};
	\node[o, label=above:$x_{i,b}$] (6) at (2, 1) {};
	\path[->,every node/.style={font=\sffamily\small}]
	(1) edge node [above] {} (3)
	(2) edge node [below] {} (6)
	(1) edge node [above] {$\varepsilon$} (6);
	\path[->,every node/.style={font=\sffamily\small}]
	(1) edge node [below] {} (4)
	(5) edge node [above] {} (6);
	\end{tikzpicture}
	\begin{tikzpicture}[-,>=stealth',auto,scale=1,
	thick]
	\tikzstyle{o}=[circle,draw]
	\node[o, label=below:$x_{i-1,a}$] (1) at (0, 0) {};
	\node[o, label=below:$x_{i-1,a''}$] (2) at (3, 0) {};
	\node[o, label=above:$x_{i,b''}$] (3) at (-1, 1) {};
	\node[o, label=above:$x_{i,b'}$] (4) at (0.5, 1) {};
	\node[o, label=below:$x_{i-1,a'}$] (5) at (1.5, 0) {};
	\node[o, label=above:$x_{i,b}$] (6) at (2, 1) {};
	\path[->,every node/.style={font=\sffamily\small}]
	(1) edge node [above] {} (3)
	(2) edge node [below] {} (6)
	(5) edge node [above right] {$\varepsilon'$} (4);
	\path[->,every node/.style={font=\sffamily\small}]
	(1) edge node [below] {} (4)
	(5) edge node [above] {} (6);
	\end{tikzpicture}
\end{center}
\caption{Illustration of the proof of Proposition \ref{p:layer_complete_Jaeger}.}
		\label{fig:modification_of_H}
\end{figure}
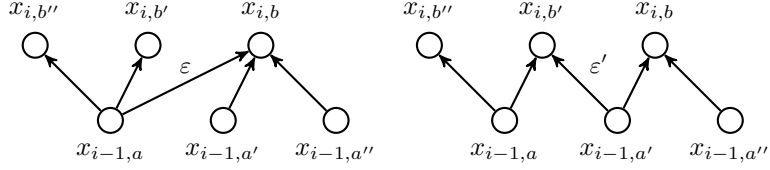 

\begin{remark}
	Notice that by Lemma \ref{lem:atcsuszo_kezdopont}, if we keep the same ribbon structure but change the basis to $(x_{i,0},x_{i,0}x_{i-1,0})$ for some $i>0$, then the Jaeger trees will once again be the non-crossing comb-trees. 
\end{remark}

\begin{remark}
In \cite{arithm_symedgepoly}, Higashitani et al.\ constructed triangulations for the facets of the symmetric edge polytope of a complete bipartite graph using Gröbner basis techniques. The facets of the symmetric edge polytope of a bipartite graph are exactly the root polytopes of the semi-balanced orientations of the graph \cite{arithm_symedgepoly}. 
The triangulation obtained in \cite{arithm_symedgepoly} agrees with the dissection by Jaeger trees that we described above for those cases. 
\end{remark}

Furthermore, the dissection by non-crossing comb-trees is in fact always a triangulation, regardless of the number of layers.

\begin{thm}
\label{t:combtrees_triangulate}
	For a layer-complete 
	graph with 
	a planar presentation as above,
	non-crossing comb-trees induce a triangulation of the root polytope.
\end{thm}

\begin{proof}
	By Theorem \ref{t:J-trees_form_quasitr} and Proposition \ref{p:layer_complete_Jaeger}, non-crossing comb-trees induce a dissection of the root polytope. Hence we only need to prove that the simplices corresponding to any two non-crossing comb-trees meet in a common face.
	
	Suppose for a contradiction that two non-crossing comb-trees, $T_1$ and $T_2$, do not satisfy this condition. Then by Lemma \ref{l:compatibility_of_trees} there is an incompatible cycle $C$, which means that for some orientation of $C$, we have $C^+ \subset T_1$ and $C^- \subset T_2$. We claim that $C$ is a subset of $G_i$, that is the subgraph of $G$ induced by the vertices along the adjacent levels $l^{-1}(i-1)$ and  $l^{-1}(i)$, for some $i$. Indeed, otherwise there would be two edges $xy$ and $yz$ along $C$ such that $l(x)+2=l(y)+1=l(z)$ and $y$ is not leftmost in its layer. This implies that $xy$ and $yz$ cannot be from the same comb-tree. On the other hand, as both edges point upward, $xy$ and $yz$ are either both in $C^+$ or they are both in $C^-$. This contradicts the assumption that $C^+$ and $C^-$ are both contained within either $T_1$ or $T_2$.
	
	Now when $C$ is a subset of $G_i$ for some $i$, 
	take an arbitrary edge $uv \in C^+$ with, say, $v\in l^{-1}(i)$. 
	Let $vw$ be the other edge of $C$ incident to $v$ (in this case, $vw\in C^-$). Suppose (without loss of generality) that $w$ is to the right of $u$. The next edge $wx$ of $C$ is again in $C^+$, hence it cannot intersect $uv$ (as they are both in $T_1$). In other words, $x$ needs to be to the right of $v$. We can continue this argument, alternately using the non-crossing property of $T_1$ and $T_2$. We obtain that edges of $C$ always ``go to the right,'' which is impossible for a cycle. This provides our final contradiction.
\end{proof}

The dissection by non-crossing comb-trees allows one to produce a formula for the interior polynomial of a layer-complete directed graph. 
We will keep using the same ribbon structure and basis as in the rest of this section.

We recall an elementary fact about non-crossing trees. Consider the complete bipartite graph $K$ with `lower' color class $\{p_0,p_1,\ldots,p_m\}$ and `upper' color class $\{q_0,q_1,\ldots,q_n\}$. Non-crossing spanning trees of $K$ are uniquely determined by a \emph{zigzag}, that is a non-crossing path whose first edge is $p_0q_0$ and whose last edge is $p_mq_n$. The zigzag has \emph{up edges}, namely those that are first reached at their lower endpoint, and \emph{down edges} that are first reached at their upper endpoint. 
The status of $p_0q_0$ is decided so that it is the opposite of that of the second edge of the zigzag.
Now the zigzag is actually part of the corresponding tree and the extension happens in an obvious way, see Figure \ref{fig:non-crossing_tree}. The extra edges may have an up edge of the zigzag on right and a down edge on the left, or vice versa. The first kind will also be called \emph{up edges} and the other kind will be called \emph{down edges}.

\begin{figure}
    \centering
    \begin{tikzpicture}[scale=.2]
    \node [circle,fill,scale=.8,draw] (1) at (0,10) {};
		\node [circle,fill,scale=.8,draw] (2) at (6,10) {};
		\node [circle,fill,scale=.8,draw] (3) at (12,10) {};
		\node [circle,fill,scale=.8,draw] (4) at (18,10) {};
		\node [circle,fill,scale=.8,draw] (5) at (24,10) {};
		\node [circle,fill,scale=.8,draw] (6) at (30,10) {};
		\node [circle,fill,scale=.8,draw] (7) at (36,10) {};
		\node [thick,circle,scale=.8,draw] (8) at (3,0) {};		
		\node [thick,circle,scale=.8,draw] (9) at (9,0) {};		
		\node [thick,circle,scale=.8,draw] (10) at (15,0) {};
		\node [thick,circle,scale=.8,draw] (11) at (21,0) {};
		\node [thick,circle,scale=.8,draw] (12) at (27,0) {};
		\node [thick,circle,scale=.8,draw] (13) at (33,0) {};
		\path [lightgray,ultra thick,->,>=stealth] (8) edge node {} (1);
		\path [ultra thick,->,>=stealth] (8) edge node {} (2);
		\path [lightgray,thick,->,>=stealth] (9) edge node {} (2);
		\path [lightgray,thick,->,>=stealth] (10) edge node {} (2);
		\path [lightgray,ultra thick,->,>=stealth] (11) edge node {} (2);
		\path [thick,->,>=stealth] (11) edge node {} (3);
		\path [thick,->,>=stealth] (11) edge node {} (4);
		\path [thick,->,>=stealth] (11) edge node {} (5);
		\path [ultra thick,->,>=stealth] (11) edge node {} (6);
		\path [lightgray,thick,->,>=stealth] (12) edge node {} (6);
		\path [lightgray,ultra thick,->,>=stealth] (13) edge node {} (6);
		\path [ultra thick,->,>=stealth] (13) edge node {} (7);
    \end{tikzpicture}
    \caption{A non-crossing tree. The zigzag inside the tree is indicated by thick lines. Up edges are shown in black and down edges are shown in gray.}
    \label{fig:non-crossing_tree}
\end{figure}
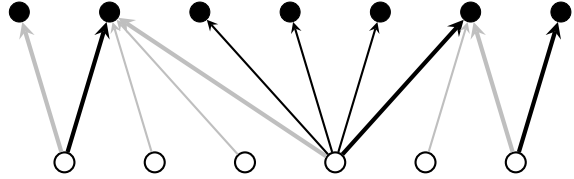

A similar description applies to each `tooth' $T_i$ of a non-crossing comb-tree $T$. We will use it to sort out the internal semi-passivity of the edges of $T$, as follows. First, the edges $x_{0,i-1}x_{0,i}$ are never internally semi-passive, per condition \ref{elso} of Lemma \ref{l:activities_description}. We will put them aside and concentrate on the rest of $T$. 

Now in the tour of $T$, down edges are first reached at their heads and thus they are not internally semi-passive by condition \ref{otodik} of Lemma \ref{l:activities_description}. Up edges are first reached at their tails and if they belong to one of the layers $i=1,\ldots,k-1$ (i.e., any layer but the top) then their fundamental cuts are not directed. (If the upper endpoint of the edge is $x_{i,j}$, where $j\ge1$, then $x_{i,j}x_{i+1,0}$ is in the fundamental cut.) That is, by part \ref{otodik} of Lemma \ref{l:activities_description}, such edges are internally semi-passive. Note that the internally semi-passive edges found so far are in a one-to-one correspondence with the set of their upper endpoints, that is, with $D_1\cup\cdots\cup D_{k-1}$.

In the top layer, 
internally semi-passive edges are exactly those 
up edges that are also part of the corresponding zigzag, except for the first and last edges of the zigzag. (In such cases, the fundamental cut of $x_{k-1,a}x_{k,b}$ contains the edge $x_{k-1,a+1}x_{k,b-1}$.) 
In sum, any fixed system of partitions into sets $U_i$ and $D_i$ (as used earlier in this section) contributes
\[\left[\prod_{j=1}^{k-1}{|D_j|+|U_{j-1}| \choose |D_j|}\right]x^{|D_1|}\,x^{|D_2|}\,\cdots\, x^{|D_{k-1}|}\sum_{i=0}^{\min\{s_k,|U_{k-1}|\}}{s_k \choose i}{|U_{k-1}| \choose i}\,x^i\]
to the interior polynomial. Here
\begin{itemize}
    \item ${|D_j|+|U_{j-1}| \choose |D_j|}$ is the number of non-crossing trees in the $j$'th layer;
    \item $x^{|D_1|+\cdots+|D_{k-1}|}$ is the contribution of all layers below the top;
    \item the last sum is the contribution of the top layer.
\end{itemize}
Indeed, regarding the last bullet point above, $s_k$ is the size of $D_k$ and $|U_{k-1}|$ can be understood as the size of the set obtained by replacing the rightmost element of $U_{k-1}$ with $x_{k-1,0}$.
One may independently choose subsets of size $i$ from the two named sets, consider their unique non-crossing perfect matching, and extend it uniquely to a zigzag in which exactly the $i$ edges of the matching (and possibly the first and last edges) are up edges. Let us note that this explanation 
can be viewed as an alternative computation (to \cite[Example 7.2]{hiperTutte}, where hypertrees were used) of the interior polynomial of a complete bipartite graph.

Thus the interior polynomial of our layer-complete directed graph is 
\begin{multline*}
    \sum_{i_1=0}^{s_1}\sum_{i_2=0}^{s_2}\cdots\sum_{i_{k-1}=0}^{s_{k-1}}{s_1\choose i_1}{s_2\choose i_2}\cdots{s_{k-1}\choose i_{k-1}}
    {i_1+s_0\choose i_1}{i_2+s_1-i_1\choose i_2}\cdots\\
    {i_{k-1}+s_{k-2}-i_{k-2}\choose i_{k-1}}
    \cdot x^{i_1+i_2+\cdots i_{k-1}}
    \sum_{i_k=0}^{s_k}{s_k\choose i_k}{s_{k-1}-i_{k-1}\choose i_k}\,x^{i_k}\\
    =\sum_{i_1=0}^{s_1}\sum_{i_2=0}^{s_2}\cdots\sum_{i_{k-1}=0}^{s_{k-1}}\sum_{i_k=0}^{s_k}
    {s_1\choose i_1}{s_2\choose i_2}\cdots{s_{k-1}\choose i_{k-1}}{s_k\choose i_k}\\
    \cdot{i_1+s_0\choose i_1}{i_2+s_1-i_1\choose i_2}\cdots
    {i_{k-1}+s_{k-2}-i_{k-2}\choose i_{k-1}}{s_{k-1}-i_{k-1}\choose i_k}\,x^{i_1
    +\cdots
    +i_k}.
\end{multline*}

\section{When do Jaeger trees induce a triangulation?}
\label{sec:diss_vs_triang}

In Theorem \ref{t:J-trees_form_quasitr} we saw that simplices corresponding to the Jaeger trees of a semi-balanced graph dissect the root polytope. This dissection might not be a triangulation, that is, the intersection of two simplices might fail to be a face in either simplex.

This section contains some partial results about when simplices of Jaeger trees form a triangulation.

We have already seen two special cases when Jaeger trees do give a triangulation: if the semi-balanced graph is planar and the ribbon structure comes from an embedding into the plane (see Proposition \ref{p:planar_triangulation}); moreover, if the semi-balanced graph is layer-complete, and the ribbon structure and basis are as specified in Section \ref{sec:layer_complete} (see Theorem \ref{t:combtrees_triangulate}).


Recall also the
necessary and sufficient condition 
of Lemma \ref{l:compatibility_of_trees} 
for the simplices of two trees to meet in a common face:
a cycle $C$ is called incompatible for two trees $T_1$ and $T_2$ if $C^+ \subset T_1$ and $C^- \subset T_2$,
and the condition is that such a cycle should not exist.

Fix a ribbon structure and a basis for the semi-balanced graph $G$. Let us call a cycle $C$ an \emph{incompatible cycle for the ribbon structure and basis} if there are two Jaeger trees 
for which $C$ is an incompatible cycle. By Lemma \ref{l:compatibility_of_trees}, Jaeger trees induce a triangulation of the root polytope if and only if there is no incompatible cycle for the given ribbon structure and basis.

Let $C$ be a cycle in a ribbon 
graph $G$. Let us denote its vertices by $v_0, \dots,v_{r-1}$, with edges $v_{i}v_{i+1}$ for each $i=0,\dots,r-1$ modulo $r$. 
In the ribbon structure,
the edges incident to each $v_i$ are separated by $v_iv_{i-1}$ and $v_iv_{i+1}$ into two intervals. Specifically, let us say that the edges strictly after $v_iv_{i-1}$ and 
up to and including 
$v_iv_{i+1}$ are to the \emph{right} of $v_i$, while 
the rest
are to the \emph{left} of $v_i$. The cycle $C$ is called \emph{non-separating} if there is any (potentially closed) path $P=u_0,\varepsilon_1,u_1,\varepsilon_2,\dots,\varepsilon_k, u_k$ (given as a sequence of vertices and edges) such that $k\ge1$,
$u_0$ and $u_k$ lie along $C$, otherwise $P$ is vertex-disjoint and edge-disjoint from $C$, and $\varepsilon_1$ is to the right of $u_0$ while $\varepsilon_k$ is to the left of $u_k$. 
We call $C$ \emph{separating} if there is no such path. A ribbon graph is planar if any only if it has no separating cycle.

\begin{prop}\label{prop:incompatible_cycle_is_nonseparating}
For a semi-balanced ribbon graph, a separating cycle of the ribbon structure
cannot be an incompatible cycle (for any basis).
\end{prop}

\begin{proof} 
Suppose that $C$ is a separating cycle in the semi-balanced ribbon graph $G$. Then for any vertex $v \notin C$, either each path $P$ from $v$ to $C$ reaches $C$ from the left, or each path reaches $C$ from the right. Let us call the latter set of vertices 
the \emph{exterior} of $C$. 
Let us also say that an incident node-edge pair $(v,\varepsilon)$ 
is in the exterior of $C$ if either 
$v$ is in the exterior of $C$, or 
$v$ lies along $C$ and $\varepsilon$ is to the right of $v$.

By symmetry we may suppose that the basis $(b_0,b_0b_1)$ is in the exterior of $C$ in the above sense.
Let $T$ be an arbitrary Jaeger tree. As $T$ is a tree, there needs to be an edge in $C-T$. Using the notation of the paragraph above the Proposition, we claim that the first current node-edge pair in the tour of $T$, so that the current edge is from $C-T$, will be of the form $(v_i,\overrightarrow{v_iv_{i+1}})$ for some $i\in \{0, \dots, r-1\}$. 

Firstly, the tour of $T$ stays in the exterior of $C$ until reaching the first current edge from $C-T$. This follows by a trivial case-by-case analysis of the position of the current node-edge pair and the types of transition steps in the tour of $T$. For example, if the current node-edge pair is $(v_j,v_jv_{j+1})$ for some $v_j$ along $C$ and $v_jv_{j+1}\in T$, then the next current node-edge pair is $(v_{j+1},v_{j+1}v_j^+)$, where the edge $v_{j+1}v_j^+$ lies to the right of $v_{j+1}$. We leave the rest of the cases to the reader.

Hence the first current edge from $C-T$ is 
in a current pair of the form $(v_i,v_iv_{i+1})$. As $T$ is a Jaeger tree, the orientation of the edge needs to be $\overrightarrow{v_iv_{i+1}}$, as claimed. In other words, $v_iv_{i+1}\in C^+$; in particular, $C-T$ needs to contain an edge from $C^+$. That is, no Jaeger tree can contain the entire set $C^+$ and therefore $C$ cannot be an incompatible cycle.
\end{proof}

We note that Proposition \ref{prop:incompatible_cycle_is_nonseparating} implies 
Proposition \ref{p:planar_triangulation}.

An equivalent formulation of Proposition \ref{prop:incompatible_cycle_is_nonseparating} 
is that incompatible cycles of semi-balanced ribbon graphs are non-separating.
The converse of this is not true: non-separating does not imply incompatible. Indeed, we saw that for complete bipartite graphs certain ribbon structures yield triangulations (i.e., no incompatible cycles), even though these ribbon structures cannot be planar (i.e., there will be non-separating cycles) if both vertex classes have size at least $3$.

We exhibit a special case 
where
topology does
nevertheless sufficiently control incompatible cycles. 

\begin{lemma} \label{l:incomp_cycle_non_planar_standard_graph}
Let $G=(U,W,E)$ be a directed bipartite graph where each vertex of $W$ either has indegree $0$ and outdegree $2$ or indegree $2$ and outdegree $0$.
Let us fix an arbitrary ribbon structure and basis $(b_0,b_0b_1)$ such that $b_0\in U$. Let $C$ be a 
cycle that does not pass through $b_0$.
If there is a path leading from $b_0$ to $C$ that reaches $C$ from the left, and there is also a path leading from $b_0$ to $C$ that reaches $C$ from the right, then $C$ is an incompatible cycle for the ribbon structure and basis.
\end{lemma}

Note that the directed graphs of Lemma \ref{l:incomp_cycle_non_planar_standard_graph} are all semi-balanced, with a potential that maps all of $U$ to $0$, the sinks in $W$ to $1$, and the sources to $-1$.

\begin{remark}
The case of bipartite graphs $G$ with each indegree in $W$ equal to $2$ and each outdegree in $W$ equal to $0$
(which is just a standard orientation)
plays an important role in the preliminary papers \cite{hiperTutte,KP_Ehrhart,hyperBernardi}.
In those works, bipartite graphs are thought of as hypergraphs with $U$ corresponding to vertices and $W$ corresponding to hyperedges. In that model, the case where each degree in $W$ is $2$ corresponds to (ordinary) graphs $\tilde G$. In that case, the interior polynomial of $G$
(i.e., 
the $h^*$-vector of $\mathcal Q_G$) 
is equivalent to $T(x,1)$, where $T(x,y)$ is the Tutte polynomial of $\tilde G$.
\end{remark}

\begin{proof}[Proof of Lemma \ref{l:incomp_cycle_non_planar_standard_graph}]
Take the undirected graph $\tilde{G}$ whose vertex set is $U$, and we connect two vertices $u_1, u_2 \in U$ if in $G$ they have a common 
neighbor. Then $G$ can be obtained from $\tilde{G}$ by subdividing each edge with a point and either orienting both new edges toward the middle, or both of them away from the middle. Endow $\tilde{G}$ with the ribbon structure inherited from $G$, and let the basis be $(b_0,b_0u)$, where $u$ is the other neighbor of $b_1$ in $G$. Let $\tilde{C}$ be the subgraph of $\tilde{G}$ corresponding to $C$ (that is, we obtain $C$ by subdividing each edge of $\tilde{C}$). As $C$ is non-separating in $G$, we have that $\tilde{C}$ is also a non-separating cycle in $\tilde G$. 

Let $P_l$ be a path in $G$ leading from $b_0$ to $C$ and reaching $C$ from the left, and let $P_r$ be a path in $G$ leading from $b_0$ to $C$ and reaching $C$ from the right. Let $\tilde{P_l}$ and $\tilde{P_r}$, respectively, be the corresponding paths in $\tilde{G}$.

Consider an arbitrary edge $\tilde{\varepsilon}$ of $\tilde{C}$. 
Then $\tilde{P_l}\cup \tilde{C} - \tilde{\varepsilon}$ is a cycle-free subgraph of $\tilde{G}$, and the same holds for $\tilde{P_r}\cup \tilde{C} - \tilde{\varepsilon}$. Extend $\tilde{P_l}\cup \tilde{C} - \tilde{\varepsilon}$ in an arbitrary way to a spanning tree $\tilde{T}_l$ of $\tilde{G}$. Similarly, extend $\tilde{P_r}\cup \tilde{C} - \tilde{\varepsilon}$ in an arbitrary way to a spanning tree of $\tilde{G}$ and call it $\tilde{T}_r$.

Let $T_r$ and $T_l$ be the corresponding trees in $G$. These are not spanning trees, because they do not contain those vertices of $W$ that subdivide the non-edges of $\tilde{T}_l$ and $\tilde{T}_r$, respectively. Let $u_1w$ and $u_2w$ be two edges of $G$ such that the corresponding edge $u_1u_2$ of $\tilde{G}$ is not in, say, $\tilde{T}_r$. Suppose that in the tour of $\tilde{T}_r$, the pair $(u_1, u_1u_2)$ appears before $(u_2,u_2u_1)$. Then we can add $u_1w$ to $T_r$ if $u_1w$ and $u_2w$ are oriented away from $w$ and we can add $u_2w$ to $T_r$ if $u_1w$ and $u_2w$ are oriented towards $w$. This way we still have a tree, 
which now contains
$w$ as well, and we do not violate the Jaeger rule at the remaining incident non-edge. Moreover, as $w$ is added as a new leaf, the tour of the extended tree remains essentially the same as before the addition. 

Hence we can extend both $T_r$ and $T_l$ to Jaeger trees $T'_r$ and $T'_l$, respectively. For both $\tilde{T}_r$ and $\tilde{T}_l$, we have that $\tilde{\varepsilon}$ is the only edge of $\tilde{C}$ missing from the tree. As the tour of $T_r$ reaches $C$ for the first time from the right, and the tour of $T_l$ reaches $C$ for the first time from the left, $T'_r$ and $T'_l$ will contain a different one of the two edges of $G$ corresponding to $\tilde{\varepsilon}$. Since one of these edges is in $C^+$ and the other is in $C^-$, we conclude that one of $T'_r$ and $T'_l$ contains $C^+$ and the other one contains $C^-$. In other words, $C$ is an incompatible cycle for $T'_r$ and $T'_l$.
\end{proof}


\begin{thm}\label{thm:exist_graph_with_no_nice_ribbon_str}
There exists a semi-balanced graph so that for any ribbon structure and basis, the dissection given by Jaeger trees is not a triangulation.
\end{thm}

\begin{proof}
We will show that the required condition holds for the directed graph $G$ obtained by subdividing each edge of the complete graph $K_{16}$ by a new vertex, and orienting each edge towards the subdividing vertex (the standard orientation). Let us denote the vertex set of $K_{16}$ by $U$, and the set of the subdividing vertices by $W$.

Fix an arbitrary ribbon structure for $G$. We need to show that for any basis $(b_0,b_0b_1)$, there will be an incompatible cycle.

If $b_0\in W$ then $b_0$ is the head of $b_0b_1$, hence by Lemma \ref{lem:atcsuszo_kezdopont}, the Jaeger trees agree with those for basis $(b_1,b_1b_0^+)$. Therefore we may suppose that $b_0\in U$, that is, a vertex of $K_{16}$.

By Lemma \ref{l:incomp_cycle_non_planar_standard_graph} it is enough to show that 
there is a non-separating cycle $C$ in $G$ such that $b_0\notin C$, with a path from $b_0$ to $C$ that reaches $C$ from the left, as well as  
a path from $b_0$ to $C$ that reaches $C$ from the right. 
Since the ribbon structure at a degree $2$ vertex is unique,
our ribbon structure of $G$ is equivalent to one of $K_{16}$, 
with an obvious one-to-one correspondence between
non-separating cycles of the two.
Hence it suffices to guarantee the existence of a cycle $\tilde C$ in $K_{16}$ with the stated properties.
Notice also that if we remove some vertices and edges from a graph, a ribbon structure is 
retained on the rest,
and if there is a cycle that is non-separating in the subgraph, then it will also be a non-separating cycle in the original graph.

Let us now remove $b_0$ from $K_{16}$, which leaves us with a copy of $K_{15}$.
Remove further edges to get a subgraph $H$ isomorphic to a $K_5$ where each edge is subdivided by a vertex.
As $K_5$ is not planar, $H$ necessarily contains a non-separating cycle $\tilde C$ (with respect to the ribbon structure induced on it).
Let $P$ be a path in $H$ that `witnesses' this, i.e., $P$ connects two vertices along $\tilde C$ so that its initial and final edges are on opposite sides of $\tilde C$.
Then $P$ has at least one vertex $v$ that is not along $\tilde C$ (indeed, there must be such a degree $2$ vertex). Now $\tilde C$ is also non-separating in $K_{16}$, in which $b_0v$ is an edge, thus by adding to $b_0v$ the portion of $P$ from $v$ to the right side of $\tilde C$ we get a path from $b_0$ to $\tilde C$ that reaches $\tilde C$ from the right, and 
similarly there also exists
a path from $b_0$ to $\tilde C$ that reaches $\tilde C$ from the left.

By further `doubling' the edges of $\tilde C$ (and the edges of $P$ and the edge $b_0v$), we obtain a cycle $C$ in $G$ that satisfies all the conditions of Lemma \ref{l:incomp_cycle_non_planar_standard_graph}. Thus an incompatible cycle does exist for our arbitrary ribbon structure and basis for $G$, which implies that Jaeger trees never induce a triangulation.
\end{proof}

We did not optimize the above proof to get a small example, and it is likely that much smaller graphs may be used to establish
the Theorem.

It can also happen that for a semi-balanced graph, each ribbon structure and basis yield a triangulation.
For example, take an even length cycle with a semi-balanced orientation.
Such graphs have a unique ribbon structure, which is planar so that Proposition \ref{p:planar_triangulation} applies. 

In our limited experience, the two extremes mentioned above are rare. That is, whether or not we get a triangulation typically depends not only on the graph but also on its ribbon structure, and even the basis. 

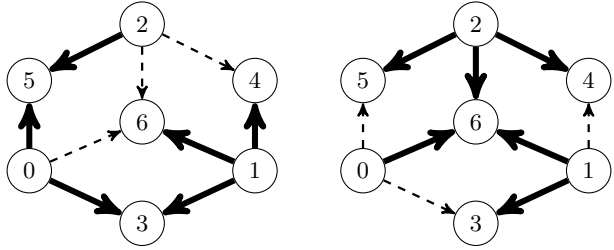
\begin{figure}[ht]
	\begin{center}
		\begin{tikzpicture}[-,>=stealth',auto,scale=0.5]
		\tikzstyle{e}=[{circle,draw}]
		\tikzstyle{v}=[{circle,draw}]
		\node[e] (3) at (0, 0.9) {{\small 3}};
		\node[v] (1) at (3, 2.3) {{\small 1}};
		\node[v] (0) at (-3, 2.3) {{\small 0}};
		\node[e] (4) at (3, 4.7) {{\small 4}};
		\node[e] (5) at (-3, 4.7) {{\small 5}};
		\node[v] (2) at (0, 6.2) {{\small 2}};
		\node[e] (6) at (0, 3.6) {{\small 6}};
		\path[->, every node/.style={font=\sffamily\small},line width=0.8mm]
		(0) edge node {} (3)
		(0) edge node {} (5)
		(1) edge node {} (3)
		(1) edge node {} (4)
		(1) edge node {} (6)
		(2) edge node {} (5);
		\path[->,dashed,every node/.style={font=\sffamily\small},line width=0.3mm]
		(0) edge node {} (6)
		(2) edge node {} (4)
		(2) edge node {} (6);
		\end{tikzpicture}
		\hspace{0.6cm}
		\begin{tikzpicture}[-,>=stealth',auto,scale=0.5]
		\tikzstyle{e}=[{circle,draw}]
		\tikzstyle{v}=[{circle,draw}]
		\node[e] (3) at (0, 0.9) {{\small 3}};
		\node[v] (1) at (3, 2.3) {{\small 1}};
		\node[v] (0) at (-3, 2.3) {{\small 0}};
		\node[e] (4) at (3, 4.7) {{\small 4}};
		\node[e] (5) at (-3, 4.7) {{\small 5}};
		\node[v] (2) at (0, 6.2) {{\small 2}};
		\node[e] (6) at (0, 3.6) {{\small 6}};
		\path[->, every node/.style={font=\sffamily\small},line width=0.8mm]
		(0) edge node {} (6)
		(1) edge node {} (3)
		(1) edge node {} (6)
		(2) edge node {} (4)
		(2) edge node {} (5)
		(2) edge node {} (6);
		\path[->,dashed,every node/.style={font=\sffamily\small},line width=0.3mm]
		(0) edge node {} (3)
		(0) edge node {} (5)
		(1) edge node {} (4);
		\end{tikzpicture}
	\end{center}
	\caption{Two incompatible Jaeger trees. 
	The incompatible cycle is the one with vertices $\{0,3,1,6\}$. 
	}
	\label{f:incomp_Jaeger_trees}
\end{figure}

\begin{prop}
For a fixed ribbon structure, whether the Jaeger trees induce a triangulation may depend on the choice of basis. 
\end{prop}

\begin{proof}
As an example, consider the semi-balanced directed graph of Figure \ref{f:incomp_Jaeger_trees}. 
We let the ribbon structure be clockwise around the vertex $0$ and counterclockwise around every other vertex.
With base vertex $5$ and base edge $(5,2)$ we do not get a triangulation, for the figure shows two incompatible Jaeger trees. On the other hand, it is easy to check that the same ribbon structure with base vertex $5$ and base edge $(5,0)$ yields a triangulation.
Hence whether we get a triangulation may depend even on the choice of base edge.
\end{proof}

Note that the example used above also has a plane-induced ribbon structure, for which Proposition \ref{p:planar_triangulation} guarantees a triangulation with any basis. We end this section with recalling that Question \ref{q:triang_greedoid_relationship} 
also addressed the problem of when Jaeger trees induce a triangulation.


\section{Product and recursion formulas}
\label{sec:formula}


The interior polynomial is a meaningful invariant of (undirected) bipartite graphs due to its natural properties. The goal of this section is to extend two of those to all semi-balanced graphs. In the process we will also address interior polynomials for disconnected graphs.

One of the properties is a certain product formula \cite[Theorem 6.7]{hiperTutte}, which turns out to have a verbatim extension to the general semi-balanced case. The statement is independent of ribbon structures, but they (and their Jaeger trees) do come in handy in the proof.

\begin{prop}
\label{p:product}
Let $G_1$ and $G_2$ be connected semi-balanced graphs with $V(G_1)\cap V(G_2)=\{h,t\}$, so that $E(G_1)\cap E(G_2)=\{\overrightarrow{th}\}$.
Then we have $I_{G_1 \cup G_2}=I_{G_1}\cdot I_{G_2}$.
\end{prop}

It will soon turn out that 
connectedness is unimportant, cf.\ Corollary \ref{cor:product}.
Also, the same formula holds when two graphs are fused at a vertex. Indeed, the root polytope of a tree is a single unimodular simplex and thus its interior polynomial is $1$, cf.\ Example \ref{ex:trees}. Consider a tree of just two edges and fuse one of those edges with an edge of $G_1$ and the other edge with an edge of $G_2$. (If $G_1$ or $G_2$ is a single point then the claim is trivial.)
Applying Proposition \ref{p:product} twice gives the desired result.

\begin{proof}
Let us choose an arbitrary ribbon structure for $G$, and let the basis be $(h, ht)$. For any spanning tree $T$ of $G$, let $T_1$ be the part of $T$ falling into $G_1$, and let $T_2$ be the part of $T$ falling into $G_2$. As the only common vertices of $G_1$ and $G_2$ are $t$ and $h$, if $th\in T$, then both $T_1$ and $T_2$ are connected, hence they are spanning trees of $G_1$ and $G_2$, respectively.

Since $h$ is the head of $ht$, each Jaeger tree of $G$ with basis $(h,ht)$ must contain $ht$. We will show that a tree $T$ of $G$ containing $ht$ is a Jaeger tree (for the basis $(h,ht)$) if and only if $T_1$ is a Jaeger tree of $G_1$ for the basis $(h,ht)$ and $T_2$ is a Jaeger tree of $G_2$ for the basis $(h,ht)$, moreover, in this case the internal semi-passivity of $T$ is the sum of the internal semi-passivities of $T_1$ and of $T_2$, respectively. These assertions imply the statement of the Proposition by Theorem \ref{thm:shelling}.

Consider the tour of $T$ in $G$ with basis $(h,ht)$. If we delete from this the edges that are not in $G_1$, we get the tour of $T_1$ in $G_1$. Likewise, if we delete the edges that are not in $G_2$, then we get the tour of $T_2$ in $G_2$. Hence both of these tours do satisfy the Jaeger property. 

Conversely, if we have two Jaeger trees $T_1$ and $T_2$ of $G_1$ and $G_2$, respectively (for the same basis $(h.ht)$), then both of them have to contain $ht$, and their union $T$ is therefore a tree. Moreover, the restriction of the tour of $T$ to $G_i$ is the tour of $T_i$, and each non-edge of $T$ is either within $G_1$ or $G_2$, hence each non-edge is first seen from its tail. Thus, $T$ is a Jaeger tree of $G$. 

For any Jaeger tree $T$ of $G$, the edge $ht$ is semi-active, because it is first seen at its head. Any other edge $ht \neq \varepsilon\in T$ is such that $\varepsilon$ is in $G_i$ for $i=1$ or $2$, in particular its fundamental cut lies entirely within $G_i$. Hence in this case, the condition \ref{otodik} of Lemma \ref{l:activities_description} for the semi-activity of $\varepsilon$ in $T$ agrees with the condition for the semi-activity of $\varepsilon$ in $T_i$. (Indeed, $\varepsilon$ is reached at its tail in $G$ if and only if it is reached at its tail in $G_i$, and the question whether the fundamental cut is a directed cut is also answered the same way for $G$ and for $G_i$.)

Thus indeed, the internal semi-passivity of $T$ is the sum of the internal semi-passivities of $T_1$ and $T_2$. 
\end{proof}

According to \eqref{eq:h-csillag}, the interior polynomial of a connected semi-balanced graph $G$ with vertex set $V$, is $I_G(t)=(1-t)^{|V|-1}\ehr_{\mathcal Q_G}(t)$. Here $|V|-1$ is indeed $1$ more than the dimension of the root polytope. However when $G$ is no longer connected, then Corollary \ref{cor:max_simplices:of_root_polytope} provides that $\dim\mathcal Q_G$ drops from $|V|-2$ to $|V|-1-c(G)$, where $c(G)$ is the number of connected components. That is, the $h^*$-vector of $\mathcal Q_G$ becomes $(1-t)^{|V|-c(G)}\ehr_{\mathcal Q_G}(t)$. However instead of this, we choose to keep 
\begin{equation}
\label{eq:disconnect}
I_G(t)=(1-t)^{|V|-1}\ehr_{\mathcal Q_G}(t)=(1-t)^{c(G)-1}h^*_{\mathcal Q_G}(t)
\end{equation}
as our definition of the \emph{interior polynomial} of a (not necessarily connected) semi-balanced graph. The extra factor of $(1-t)^{c(G)-1}$ will help keep the formula of Theorem \ref{thm:rekurzio} below simple, and it is also in line with Kato's extension $I'$ \cite{kato_signed} of the interior polynomial to disconnected graphs.

Let us make a few simple observations about our extended invariant $I$.

\begin{lemma}
\label{lem:bridge}
Suppose that the semi-balanced graph $G$ contains a bridge edge $\varepsilon$. Then we have $I_{G-\varepsilon}(t)=(1-t)I_G(t)$.
\end{lemma}

\begin{proof}
A bridge $\varepsilon$ can be viewed as a one-element cut. As explained in Remark \ref{rem:hyperplane_from_cut}, there is a hyperplane that contains $\mathcal Q_{G-\varepsilon}$ but does not contain $\mathbf x_\varepsilon$; in other words, $\mathcal Q_G$ is a cone over $\mathcal Q_{G-\varepsilon}$. Any triangulation of $\mathcal Q_{G-\varepsilon}$ has the same $h$-vector as the triangulation of $\mathcal Q_G$ obtained by coning. Here if we triangulate without introducing new vertices, then Proposition \ref{prop:aff_indep_in_root_polytope} and Corollary \ref{cor:unimod} guarantee that all simplices are unimodular. Thus by \eqref{eq:generator}, the $h^*$-vectors of the two polytopes also coincide. The difference in interior polynomials is then only due to the difference in the number of connected components, in the way we claimed.
\end{proof}

\begin{lemma}
\label{lem:disjoint_union}
If $G_1,\ldots,G_k$ are the connected components of the semi-balanced graph $G$, then we have
\[I_G(t)=(1-t)^{k-1}I_{G_1}(t)\cdots I_{G_k}(t).\]
Consequently, if $G=G_1\sqcup G_2$ is the disjoint union of the (not necessarily connected) semi-balanced graphs $G_1$ and $G_2$, then we have $I_G(t)=(1-t)I_{G_1}(t)I_{G_2}(t)$.
\end{lemma}

\begin{proof}
The first claim follows by induction on $k$. Indeed it is trivial for $k=1$. 
Let us now assume that $k=2$. For $i=1,2$, construct the graph $G'_i$ by adding a new vertex to $G_i$ and connecting it to an arbitrary old vertex. The resulting graphs $G_1'$ and $G_2'$ satisfy $I_{G_1}=I_{G_1'}$ and  $I_{G_2}=I_{G_2'}$ by Proposition \ref{p:product}. If we fuse $G'_1$ and $G'_2$ along the new edges so that the new graph $G_{12}$ is $G_1$ and $G_2$ joined by a single edge, then again by Proposition \ref{p:product}, we have $I_{G_{12}}=I_{G_1'}I_{G_2'}=I_{G_1}I_{G_2}$. By Lemma \ref{lem:bridge} this implies that 
\[I_{G_1\sqcup G_2}(t)=(1-t)I_{G_{12}}(t)=(1-t)I_{G_1}(t)I_{G_2}(t),\]
just as we claimed. Applying the above formula to the graphs $G_1\sqcup\cdots\sqcup G_i$ and $G_{i+1}$ provides the general inductive step. Finally, the second claim of the Lemma is an obvious application of the first.
\end{proof}

\begin{cor}
\label{cor:product}
The formula of Proposition \ref{p:product} (multiplicativity of the interior polynomial when fusing graphs at an edge or vertex) holds without assuming that $G_1$ and $G_2$ are connected.
\end{cor}

\begin{proof}
Apply Lemma \ref{lem:disjoint_union} to $G_1$ and $G_2$ separately on the right hand side, then Proposition \ref{p:product} to the two components that are being merged, finally Lemma \ref{lem:disjoint_union} again, this time to $G_1\cup G_2$, to arrive at the left hand side.
\end{proof}

Finally, let us turn to the other remarkable property that we alluded to earlier. 
That is, our last claim is a generalization of Kato's recursion formula \cite[Corollary 1.3]{kato_signed}. The proof follows his argument closely.

\begin{thm}
\label{thm:rekurzio}
Let $G$ be a semi-balanced graph and $C$ a cycle in $G$. If $C^+$ denotes the half of the edges of $C$ that point in one of the cyclic directions, then we have
\[\sum_{S\subset C^+}(-1)^{|S|}I_{G-S}=0.\]
\end{thm}

We call this a recursion because if we isolate the $S=\emptyset$ term on one side, we obtain an expression of the interior polynomial of $G$ in terms of polynomials attached to smaller graphs.

\begin{proof}
Let $E$ be the set of edges of $G$. We start by considering a fairly obvious 
simplicial complex $\Delta$ made up of some facets of the standard simplex in $\R^E$, namely those that are opposite to vertices representing elements of $C^+$. 
More precisely, if $\varepsilon_1,\ldots,\varepsilon_k$ are the edges in $C^+$ and $\delta_1,\ldots,\delta_n$ denote the rest of the edges of $G$, then we write $\mathbf y_{\varepsilon_i}$ and $\mathbf y_{\delta_j}$ for the corresponding generators of $\R^E$ and put
\[\Delta=\bigcup_{i=1}^k\conv\{\mathbf y_{\varepsilon_1},\ldots,\mathbf y_{\varepsilon_{i-1}},\mathbf y_{\varepsilon_{i+1}},\ldots,\mathbf y_{\varepsilon_k},\mathbf y_{\delta_1},\ldots,\mathbf y_{\delta_n}\}.\]
Let us denote the various intersections of the maximal simplices of $\Delta$, indexed by sets $\emptyset\ne S\subset C^+$, with
\[\Delta_S=\conv\left(\{\mathbf y_\varepsilon\mid\varepsilon\in C^+\setminus S\}\cup\{\mathbf y_{\delta_1},\ldots,\mathbf y_{\delta_n}\}\right).\]
Then by simple inclusion-exclusion, we get an equation of characteristic functions
\begin{equation}
\label{eq:szita}
[\Delta]=\sum_{\emptyset\ne S\subset C^+}(-1)^{|S|-1}[\Delta_S].
\end{equation}
Here all the $\Delta_S$ are simplices but $\Delta$ itself is not, although it is contractible.

Next we map this `abstract' construction $\Delta$ into our usual space $\R^V$, by linearly extending the natural association of the vertex $\mathbf x_\varepsilon$ of $\mathcal Q_G$ to the free generator $\mathbf y_\varepsilon$, where $\varepsilon\in E$ is arbitrary. Let us denote this mapping by $\pi\colon\R^E\to\R^V$. A crucial point (in fact, the only time in the proof that the cycle $C$ plays a role) is that the image of $\Delta$ coincides with $\mathcal Q_G$. Indeed $\pi(\Delta)\subset\mathcal Q_G$ is obvious and conversely, given any convex combination of the $\mathbf x_\varepsilon$, the coefficients can be modified by subtracting the left hand side of the equation \eqref{eq:kombinacio} where $\lambda$ is chosen as the minimal coefficient attached to some element of $C^+$. This shows that any point of $\mathcal Q_G$ is part of $\pi(\Delta_S)$ for some singleton $S\subset C^+$.

Now by \cite[Theorem 3.1]{barvinok}, \eqref{eq:szita} and the obvious $\pi(\Delta_S)=\mathcal Q_{G-S}$ imply
\[[\mathcal Q_G]=\sum_{\emptyset\ne S\subset C^+}(-1)^{|S|-1}[\mathcal Q_{G-S}].\]
By the definition of Ehrhart series, from this we immediately obtain 
\[\ehr_{\mathcal Q_G}(t)=\sum_{\emptyset\ne S\subset C^+}(-1)^{|S|-1}\ehr_{\mathcal Q_{G-S}}(t).\]
Multiplying both sides of this relation by $(1-t)^{|V|-1}$ completes the proof.
\end{proof}

\bibliographystyle{plain}
\bibliography{Bernardi}

\end{document}